\documentclass[11pt, letterpaper]{amsart}

\usepackage[marginpar=1.5in,left=1.5in, right=1.5in, top=1in, bottom=1in]{geometry}
\usepackage{alltt}

\usepackage{float}   
\newfloat{winninglist}{H}{lol}
\floatname{winninglist}{List}

\usepackage{marginnote}

\let\oldtocsection=\tocsection

\let\oldtocsubsection=\tocsubsection

\let\oldtocsubsubsection=\tocsubsubsection

\renewcommand{\tocsection}[2]{\hspace{0em}\oldtocsection{#1}{#2}}
\renewcommand{\tocsubsection}[2]{\hspace{1em}\oldtocsubsection{#1}{#2}}
\renewcommand{\tocsubsubsection}[2]{\hspace{2em}\oldtocsubsubsection{#1}{#2}}

\usepackage{amsfonts}
\usepackage{euscript}
\usepackage{enumitem}
\usepackage{latexsym,epsfig,verbatim}
\usepackage{amsmath,amsthm,amssymb}
\usepackage{colonequals}
\usepackage{units}
\usepackage{blkarray} 

\usepackage{url}
\usepackage{color}
\usepackage[dvipsnames]{xcolor}
\definecolor{pastelpink}{RGB}{233,175,221}
\definecolor{barneypurple}{RGB}{153,85,255}
\definecolor{oceanblue}{RGB}{0,0,255}

\usepackage[colorlinks,citecolor=blue,linkcolor=red!80!black]{hyperref}
\usepackage[nameinlink]{cleveref}

\usepackage{comment}
\usepackage{amscd}
\usepackage{hyperref}

\usepackage{accents}

\usepackage{tikz}
\usetikzlibrary{matrix,arrows,decorations.pathmorphing}
\theoremstyle{plain}
\newtheorem{theorem}{Theorem}[section]
\newtheorem{lemma}[theorem]{Lemma}
\newtheorem{corollary}[theorem]{Corollary}
\newtheorem{proposition}[theorem]{Proposition}

\newtheorem{claim}[theorem]{Claim}
\newtheorem*{claime}{Claim}

\newtheorem{MAINTHM}{Theorem}

\Crefname{claim}{Claim}{Claims}
\Crefname{lemma}{Lemma}{Lemmas}
\Crefname{winninglist}{List}{Lists}
\newtheorem*{claim*}{Claim}

\theoremstyle{definition}

\newtheorem{definition}[theorem]{Definition}

\crefname{convention}{Convention}{Conventions}

\newcommand{\G}{\Gamma}

\newcommand{\R}{\mathbb{R}}
\newcommand{\Z}{\mathbb{Z}}
\newcommand{\Q}{\mathbb{Q}}

\newcommand{\bbP}{\mathbb{P}}
\newcommand{\bbF}{\mathbb{F}}

\newcommand{\calC}{\mathcal{C}}

\newcommand{\cL}{\mathcal{L}}
\newcommand{\cR}{\mathcal{R}}
\newcommand{\cP}{\mathcal{P}}
\newcommand{\cQ}{\mathcal{Q}}
\newcommand{\calF}{\mathcal{F}}
\newcommand{\calG}{\mathcal{G}}
\newcommand{\calM}{\mathcal{M}}
\DeclareMathOperator{\diam}{diam}
\DeclareMathOperator{\diag}{diag}
\DeclareMathOperator{\rk}{rk}

\DeclareMathOperator{\Out}{Out}

\DeclareMathOperator{\Curr}{Curr}

\DeclareMathOperator{\vol}{vol}

\usepackage{enumitem}

\renewcommand{\epsilon}{\varepsilon}
\newcommand{\eps}{\epsilon}
\newcommand{\ov}[1]{\overline{#1}}
\newcommand{\wt}[1]{\widetilde{#1}}
\newcommand{\ph}[1]{\phantom{#1}}
\newcommand{\rvline}{\hspace*{-\arraycolsep}\vline\hspace*{-\arraycolsep}}

\title{Nonunique Ergodicity on the Boundary of Outer Space}
\author[Bestvina, Field, and Kwak]{Mladen Bestvina, Elizabeth Field and Sanghoon Kwak}

\date{\today}

\begin{document}

\begin{abstract}
  To an $\R$-tree in the boundary of Outer space, we associate two simplices:
  the simplex of projective length measures, and the simplex of projective dual
  currents. For both kinds of simplices, we estimate the dimension of maximal
  simplices for arational $\mathbb{R}$-trees in the boundary of Outer
  space.
\end{abstract}

\maketitle

\addtocontents{toc}{\protect\setcounter{tocdepth}{0}}

\section*{Introduction}

Given a complete hyperbolic surface $S$ of finite area and a fixed geodesic
lamination $\Lambda$ on $S$, one can  consider the space of measures transverse
to $\Lambda$. This space of measures naturally has the structure of a simplicial
cone, and so, when considered projectively, sits as a simplex inside the
boundary of the Teichm\"{u}ller space of the surface, $\mathcal{PML}(S)$. The
vertices of this simplex are exactly the projective classes of ergodic measures
which are transverse to $\Lambda$. The most interesting case to consider is when
$\Lambda$ is a \emph{filling} geodesic lamination. In \cite{gabai2009almost},
Gabai shows that given $k$ nonisotopic pairwise disjoint simple closed curves on $S$, one
can construct a filling geodesic lamination on $S$ which admits $k$ distinct
projective classes of ergodic measures.
On the other hand, Levitt~\cite{levitt1983} shows that the number of pants
curves on $S$ gives the upper bound of the number of ergodic measures supported on noncompact
leaves of a foliation on $S$ using a cohomological argument.
Moreover, Lenzhen and Masur
\cite{LM2010criteria} show that any filling geodesic lamination on $S$ which
admits $k$ projective classes of ergodic
measures can be approximated by a sequence of multicurves with $k$ simple closed
curve components.
It then follows that the maximal number of distinct projective classes of
ergodic measures transverse to a filling geodesic lamination on $S$ is precisely
the number of curves in a pants decomposition of $S$.

In this paper, we investigate the corresponding question for trees in the
boundary of Culler--Vogtmann's Outer space, $CV_n$ \cite{culler1986moduli}. The
boundary of $CV_n$ consists of the set of \textit{very small $F_n$-trees} which
are either not free or not discrete. 
For such $F_n$-trees, the most interesting case to
consider is when the tree is \textit{arational}, meaning the restriction of the
$F_n$ action to any proper free factor is free and discrete. On the other hand,
the $F_n$-trees come equipped with a natural length measure.
Projective classes of length measures on an $F_n$-tree form a simplex in
$\partial CV_n$, with vertices of this simplex corresponding to ergodic ones.
When the simplex is a point, we say $T$ is \textbf{uniquely ergometric}.
The first examples of nonuniquely ergometric arational trees, not dual to a
measured lamination on a surface, were constructed by Martin \cite{martin1997nonuniquely}.
The first systematic study on the nonunique ergometricity of $F_n$-trees
was done by Guirardel \cite[Corollary 5.4]{guirardel2000dynamics}.
We estimate the maximum nonunique ergometricity for arational trees.
\begin{MAINTHM}\label{thm:main}
  For $n \ge 5$, the maximal number of projective classes of ergodic length functions
  on an arational $F_n$-tree is in the interval $[2n-6,2n-1]$.
\end{MAINTHM}

One can also consider the space of projective \emph{currents} $\nu$ that are dual to an
$F_n$-tree $T$, in the sense that the length pairing $\langle T,
\nu \rangle = 0$. This space again forms a simplex, with the vertices of this simplex
representing ergodic currents. If the simplex of currents dual to $T$ is a point, then we say the
tree $T$ is \textbf{uniquely ergodic}. Nonunique ergodicity of $F_n$-trees was
first studied by Coulbois--Hilion--Lustig \cite[Section 7]{CHL2008RtreeII}. We give a similar estimate for the
maximum nonunique ergodicity for arational trees.
\begin{MAINTHM}\label{thm:main_current}
  For $n \ge 5$,
  the maximal number of projective classes of ergodic
  currents dual to an arational $F_n$-tree is in the interval
  $[2n-6,2n-1]$. 
\end{MAINTHM}

For smaller cases of $n$, we make the following remarks. When $n=1$, Outer space $CV_1$ is a
point, so the boundary $\partial CV_1$ is empty. For $n=2$,
any arational $\R$-tree is geometric, in the sense that it arises as the dual
tree of a filling measured lamination on a punctured torus. Since the set of
minimal filling laminations is homeomorphic to $\R \setminus \Q$ for a punctured
torus (see e.g. \cite{gabai2014topology}), any measured filling lamination is
uniquely ergodic. Hence, when $n=2$, any arational tree is uniquely ergometric
and uniquely ergodic.  When $n=3$ or $4$, we can obtain a better lower bound of $n-1$ by using a rose graph with positive train track structure following an idea similar to one presented in \Cref{sec:lowerbound,sec:lowerbound_unfolding}.

Given a measured geodesic lamination on a surface, there is an $\mathbb{R}$-tree
which is dual to this lamination. When the fundamental group of the surface is
free of rank $n$, this tree represents a point in $\partial CV_n$. Trees in
$\partial CV_n$ which arise as the dual tree to a measured geodesic lamination on a
surface are said to be \textit{of surface type}. 
The space of ergodic currents dual to a surface-type tree coincides
precisely with the space of ergodic length measures on the tree, and the bounds
given by Gabai \cite{gabai2009almost}, Lenzhen--Masur~\cite{LM2010criteria}, and Levitt~\cite{levitt1983} hold. By Reynolds
\cite{reynolds2012reducing}, any \emph{geometric} arational $F_n$-tree $T$ is
dual to a filling measured geodesic lamination on a surface $S$ with one boundary
component. Say $S=S_{g,1}$ has genus $g$ and one boundary component, so that 
$T$ is an $F_{n}$-tree with $n=2g$.  As a pants decomposition of $S$ contains $3g - 2$
curves, we see that the maximum number of projectively distinct ergodic length
measures (or ergodic currents) is $3g-2$, which is on the order of
$\frac{3}{2}n$. Thus, our results show that nongeometric arational trees can
admit more ergodic length measures and ergodic currents, on the order of $2n$.

 It is still an open question whether or not the
simplex of projective classes of measured currents dual to an
$F_n$-tree $T$ always has the same dimension as
the simplex of projective classes of length functions on $T$ for
arational trees which are nongeometric.

Our approach is inspired by the work of Namazi--Pettet--Reynolds \cite{NPR2014}.
For the lower bound in \Cref{thm:main}, we consider a folding path in $CV_n$ so
that the transition matrices satisfy conditions where diagonal entries dominate
others. We phrase these conditions conveniently in terms of an ``Alice and Bob
game,'' see \Cref{thm:TheGame_folding}.
The difficulty here is in finding a folding path that realizes these conditions, and we suspect that our bound of $2n-6$ can be improved. For \Cref{thm:main_current}, the lower bound is obtained similarly, but instead using unfolding paths in $CV_n$. In both constructions, a technical difficulty lies in ensuring that the limiting trees are arational. We do this by enumerating free factors and guiding the (un)folding path to ensure that these factors act discretely in the limit, one at a time, while ensuring that the matrix conditions still hold. 

We remark that the matrices we construct have deep connections to the field of dynamics and have been used to build
examples of non-uniquely ergodic systems. For instance, in
\cite{keane1977nonergodic}, Keane constructed a minimal 4-interval
exchange with two ergodic measures, and
Yoccoz \cite[Section 8.3]{yoccoz2010interval} constructed minimal interval exchange transformations with the maximal
number of ergodic measures.

For the upper bound in both \Cref{thm:main} and \Cref{thm:main_current}, we use
Transverse Decomposition theorems, \Cref{thm:NPRdecomp-folding,thm:NPRdecomp-unfolding}, to obtain a nice transverse
decomposition of graphs induced from linearly independent ergodic measures.
Namazi--Pettet--Reynolds \cite{NPR2014} gave the first construction for such a
transverse decomposition, but we give a more direct
argument. We then put an order on the edges
with respect to the decomposition so that adding the edges in the increasing
order strictly raises the ``complexity'' of the resultant graph. Since the number of steps to obtain the
whole graph depends only on the number of different ergodic components from the
transverse decomposition, this gives an upper bound on the number of ergodic
measures.

In \Cref{sec:prelim}, we give the necessary background on Outer space, length
measures, currents,
folding/unfolding sequences, arationality and ergodicity. We also give our proofs
of the Transverse Decomposition theorems, \Cref{thm:NPRdecomp-folding,thm:NPRdecomp-unfolding}. Then, toward the proof
of \Cref{thm:main}, \Cref{sec:lowerbound}  establishes the lower bound $2n-6$
for the maximal number of projective classes of length measures on an $\R$-tree, $T$. This bound is obtained by constructing $T$ as a limiting tree of a folding sequence whose transition matrices have diagonal-dominant entries.
 In \Cref{sec:lowerbound_unfolding}, we begin our proof of \Cref{thm:main_current} by establishing the same lower bound for the maximal
number of projective classes of measured currents dual to an $\R$-tree in the
accumulation set of an unfolding sequence. We do this by constructing an unfolding sequence with
diagonal-dominant transition matrices, which is a bit more involved than the
folding sequence construction. In \Cref{sec:upperbound}, we provide the other
bound for \Cref{thm:main} and \Cref{thm:main_current} by establishing the upper bound
$2n-1$ for both simplices of nonuniquely ergodic and nonuniquely ergometric trees using the
ergodic transverse decomposition and the complexity of graphs. To finalize the proofs of
\Cref{thm:main} and \Cref{thm:main_current}, we prove in 
\Cref{sec:arationality} that we can arrange the folding and unfolding
sequences constructed in \Cref{sec:lowerbound} and
\Cref{sec:lowerbound_unfolding} to ensure that the limiting objects are
arational and non-geometric, while preserving the number of linearly independent
measures.

\subsection*{Acknowledgments}
We thank Jon Chaika for helpful conversations throughout this project,
pointing us to the papers \cite{keane1977nonergodic,yoccoz2010interval} and
providing novel perspectives.
The third author would like to thank Carlos Ospina for helpful conversations while building
the proof of \Cref{thm:arational_limiting_unfolding}, and thank Jing Tao for
helpful discussions on \cite{bestvina2024limit} and arational trees. We
thank Gilbert Levitt for directing us to the reference \cite{levitt1983}. The authors would also like to thank the referee for useful comments that improved the readability of this paper.

The authors gratefully acknowledge support: MB from NSF DMS-2304774, EF from NSF
DMS-1840190 and DMS-2103275, and SK from the University of Utah Graduate
Research Fellowship, from KIAS Individual Grant (HP098501) via the June E
Huh Center for Mathematical Challenges at Korea Institute for Advanced Study and from the New Faculty Startup Fund (700-20250069) at Seoul National University.

\tableofcontents

\addtocontents{toc}{\protect\setcounter{tocdepth}{3}}

\section{Preliminaries}
\label{sec:prelim}

\subsection{Graphs and Outer space} \label{ssec:graphsandCVn}
A \textbf{graph} is a CW complex of dimension at most 1. Denote by $R_n$ the
rose graph with $n$ petals, which is a graph with one 0-cell (called a
\textbf{vertex}) and $n$ 1-cells (called \textbf{edges}). For a graph $G$, we denote by $VG$ the
set of vertices of $G$ and $EG$ the set of edges of $G$. A \textbf{marking} of
a graph $G$ of rank $n$ is a homotopy equivalence $m : R_n \to G$. A \textbf{metric} on $G$ is a function
$\ell:EG \to \R_{>0}$, which assigns to each edge a positive number. Indeed,
$\ell$ endows a path metric on $G$: the distance between two points is defined
as the length of the shortest (weighted) path joining the two vertices. The
\textbf{volume}, $\vol(G,\ell)$, of a metric graph $(G,\ell)$ is the sum $\sum_{e \in EG}\ell(e)$.
Culler--Vogtmann \cite{culler1986moduli} introduced 
\textbf{Outer space}, $CV_n$, as the space of equivalence classes of
finite marked metric
graphs of rank $n$ and of unit volume:
\[
  CV_{n}:= \{(G, m, \ell)\ |\ \vol(G,\ell)=1\}/\sim,
\]
where $(G,m,\ell) \sim (G',m',\ell')$ if there is an isometry $i:G \to G'$ such
that $i\circ m$ is homotopic to $m'$. The collection of marked metric graphs of
rank $n$ without the unit volume condition is denoted by $cv_n$, the
\textbf{unprojectivized Outer space}. Culler--Vogtmann also showed $CV_n$ is contractible.
For a more detailed exposition on $CV_n$, see \cite{bestvina2014PCMI}.

Let $G$ be a finite graph of rank $n$ such that all vertices have valence at least 3,
and fix a marking $m: R_n\to G$. Given any other metric graph $G'$, a map $f:
G \to G'$ is said to be a \textbf{morphism} if it maps vertices to
vertices and it is locally injective on the
interior of every edge (and hence has no backtrackings). The \textbf{transition matrix}
of a morphism $f: G \to G'$ is the matrix which records how the morphism $f$ maps
edges of $G$ over edges of $G'$. In particular, if we label the (unoriented)
edges of $G$ by $e_1, e_2, \ldots, e_m$ and the edges of $G'$ by $e_1', e_2',
\ldots, e_n'$, then the transition matrix $M_f$ for the morphism $f: G \to G'$
is the $n \times m$ matrix where the $ij$-entry records the number of times the path
$f(e_j)$ crosses over the edge $e_i'$.
Per Stallings \cite{stallings1983topology}, we have the following definitions.
\begin{definition}
  Given a graph $G$ and two edges $e,e'$ of $G$, a \textbf{fold} is the natural
  quotient morphism $G \to G/(e \sim e')$. Moreover, if $e,e'$ only share one
  vertex, then the fold is said to be of \textbf{type I}. Otherwise, if $e,e'$
  share both vertices, then the fold is said to be of \textbf{type II}.
\end{definition}

Both types of folds induce surjective homomorphisms on fundamental groups, but only type I folds
induce injective ones. Stallings \cite[Algorithm 5.4]{stallings1983topology}
proved that any morphism which is a homotopy equivalence can be decomposed as a
sequence of type I folds. 

\subsection{Train tracks and the geometry of Outer space}
For a Lipschitz map $f$ between two metric spaces, we denote by $\sigma(f)$ its
Lipschitz constant. We can now define the distance between two marked
metric graphs $(G,m,\ell), (G',m',\ell')$ in $CV_n$ as follows.
\[
  d(G,G') := \inf \left\{\log \sigma(f)\ |\ f:G \to G' \text{ is Lipschitz and
    $f m \simeq m'$}\right\}.
\]
This metric is called the \textbf{Lipschitz metric} on $CV_n$, but it is asymmetric:
$d(G,G')$ may not be equal to $d(G',G)$.
By the Arz{\'e}la--Ascoli theorem, the infimum above is realized for every $G,G'
\in CV_n$. We will call such a witness $f: G \to G'$ an \textbf{optimal map}
if, in addition to realizing the infimum, it is linear on each edge. For an optimal
map $f: G \to G'$ with $\sigma(f)=\lambda$, 
we define the \textbf{tension graph} of $f$, denoted by $\Delta_f$, to be the subgraph of $G$ induced by those edges $e\in EG$ with slope $\lambda$
(Namely, the union of $e \in EG$ such that $\sigma(f|_e)=\lambda$.)

Now we define a train track structure on a graph.
For each vertex $v$ in a graph $G$, a \textbf{direction} at $v$ is an oriented edge
of $G$ with initial vertex $v$. Denote by $T_v(G)$ the set of all directions at
$v$, which can be regarded as the unit tangent space of $G$ at $v$. Then we give an
equivalence relation on $T_v(G)$. A \textbf{train track structure} on a graph
$G$ is a collection $\{(T_v(G), \sim_v)\}_{v \in VG}$, where $\sim_v$ is an equivalence
relation on $T_v(G)$ for each vertex $v$ of $G$. The equivalence classes are
called \textbf{gates}.
  A \textbf{turn} at $v$ is an unordered pair of directions at $v$, and it is
  called \textbf{legal} if the directions belong to different gates.
An immersed path in $G$ is called \textbf{legal} if it makes legal turns at
each vertex. 
Given a morphism $f: G \to G'$
and a vertex $v$ in $G$, we note that $f$ induces a map $Df: T_v(G) \to T_{f(v)}(G')$, as $f$ maps vertices to vertices. We then have a natural train track structure on $G$ induced by $f$ where for each vertex $v$ of $G$, we define an equivalence relation $\sim_v$ on $T_v(G)$ as:
\[
  d_1 \sim_v d_2 \quad \text{if and only if} \quad Df(d_1) = Df(d_2).
\]
Given a morphism $f:G \to G'$, we will always consider the induced train track structure on $G$, unless otherwise specified.
With this induced structure, we
notice that an immersed edge path $p$ in $G$ is legal if $f(p)$ is a reduced path, i.e., $f$ is
locally injective on the vertices of $p$ (as well as on the interior of the edges).

\subsection{Folding and unfolding sequences}
\label{ssec:sequences}

A \textbf{folding/unfolding} sequence is a sequence of graphs,
\[
\cdots \to G_{-2} \to G_{-1} \to G_0 \to G_1 \to G_2 \to \cdots,
\]
together with morphisms $\{\psi_s: G_s \to G_{s+1}\}$ that are homotopy
equivalences, and so that all finite compositions are morphisms.
We consider two special cases: \textbf{folding} sequences, which are
  enumerated by nonnegative integers and
denoted by
$(G_s)_{s \ge 0}$, and \textbf{unfolding} sequences, enumerated by nonpositive integers and denoted by $(G_s)_{s \le 0}$.

For each $r, s\in \mathbb{Z}$ with $r < s$,
we write the composition of maps as
\[
  \phi_{r,s}:= \psi_{s-1}  \circ \cdots \circ \psi_{r+1} \circ \psi_{r}: G_r \to G_s 
\]
For each morphism
$\phi_{r,s}:G_r \to G_s$, denote by $M=M^{r,s}$ its transition matrix. Note that
as we compose functions right to left, we have that $\phi_{r,t} = \phi_{s,t}\circ
\phi_{r,s}$ and $M^{r,t} = M^{s,t}\cdot M^{r,s}$.
We say that an unfolding sequence $(G_s)_{s \le 0}$ is \textbf{reduced} if 
there is no proper invariant subgraph in the sequence. That is, if there is a
sequence $\{H_s\}_{s \le 0}$ of subgraphs $H_t\subseteq G_t$ for $t\leq 0$
and there is $s_0 < 0$ such that $\phi_{r,s}(H_r)\subseteq H_s$ for $r<s\leq s_0$, then either
$H_s=G_s$ for all $s$ or $H_s$ contains no edges for all large $-s$. 
Likewise, a folding sequence $(G_s)_{s \ge 0}$ is \textbf{reduced} if for any sequence $\{H_s\}_{s \ge 0}$ of subgraphs $H_t\subseteq G_t$ for
$t\geq 0$, if there is $r_0>0$ such that
$\phi_{r,s}(H_r)\subseteq H_s$ for all $s>r\ge r_0$, then either $H_r = G_r$ for
all large $r$ or $H_r$ contains no edges for all $r$.

In a folding sequence, train track structures on any $G_r$ induced by
the morphisms $G_r\to G_s$ eventually stabilize as $s\to\infty$, and
in practice we construct folding sequences by fixing a train track
structure on $G_0$ and then folding compatibly.

A \textbf{length measure} on a folding/unfolding sequence $(G_s)_s$ is defined
to be a sequence of vectors $(\vec{\ell}_s)_{s}$ in $\R^{|EG_s|}$, where for each
$s$, the vectors satisfy the compatibility condition (`lengths pull
back'):
\[
  \vec{\ell}_s = (M^{s,s+1})^T\vec{\ell}_{s+1}.
\]
We consider the set $\cL((G_s)_{s\ge 0})$ of length measures on the folding sequence
$(G_s)_{s \ge 0}$ as the cone of non-negative vectors in the finite dimensional
vector space whose vectors are sequences $(\vec{v}_s)_{s \ge 0}$ with $\vec{v}_s
\in \R^{|EG_s|}$ for every $s \ge 0$ satisfying the compatibility condition
\[
  \vec{v}_{s} = (M^{s,s+1})^T\vec{v}_{s+1}
\]
for every $s \ge 0$. The dimension of this space is nonzero and bounded above by $\liminf_{s
  \to \infty} |EG_s|>0$.

Recall that a morphism that is a homotopy equivalence can be realized
as a composition of type I 
folds, after possibly
subdividing, (see \Cref{ssec:graphsandCVn}). Then the lengths
of conjugacy classes in $F_n$ 
are nonincreasing 
in the folding sequence. Hence, seeing
$\ov{CV_n} \subset \mathbb P\R_{\geq 0}^{F_n}$ (see \Cref{ssec:boundaryCVn}), it follows that 
a folding sequence $(G_s)_{s \ge 0}$ always projectively converges to an $\R$-tree
$T$, which we call the
\textbf{limit tree}. 
More precisely, to a reduced folding sequence $(G_s)_{s \ge 0}$, Namazi--Pettet--Reynolds associate a
{\it topological limit tree} $\hat{T}$, \cite[Section 5.1]{NPR2014}. This is a topological tree (a space that is uniquely path connected and locally path connected) admitting an $F_n$-action
obtained by considering the universal cover $T_0$ of $G_0$, identifying points
in $T_0$ whose images coincide in the universal cover $T_n$ of $G_n$, for some
$n > 0$, and finally ensuring that this quotient space is Hausdorff. They show
that the natural $F_n$-action on $\hat{T}$ has dense orbits.
Now, assigning a length measure to the folding sequence $(G_s)_{s \ge 0}$  yields a pseudo-metric on $\hat{T}$ as follows. For each $s \ge 0$, denote by $d_s$
the induced metric on the universal cover $T_s:=\wt{G_s}$. Then, $\hat{T}$ has a
pseudo-metric defined as
  \[
    d(x,y):= \inf_{s \ge 0}d_s(x_s,y_s),
  \]
  where $x_s,y_s \in T_s$ are points mapped to $x,y \in \hat{T}$, respectively,
  under the quotient map $T_s \to \hat{T}$. Finally, let $T$ be the $\R$-tree
  obtained from $\hat{T}$ by identifying pairs of points at pseudo-distance 0, and let $d$
  be its induced metric. Then,

\begin{lemma}[{\cite[Lemma 5.1]{NPR2014}}]\label{lem:folding_limiting_tree}
  A folding sequence $(G_s,\vec{\ell}_s)_{s \ge 0}$ of marked metric
  graphs, after projection to $CV_n$, converges to an $\R$-tree $(T,d)$, called the \textbf{limit tree},
  where $d$ is the induced metric obtained from the pseudo-metric on the topological limit
  tree $\hat{T}$ of the folding sequence.
\end{lemma}

Recall a \textbf{length measure} on an $F_n$-$\R$-tree $T$ is a collection of finite Borel measures $\{\mu_I \mid \text{$I$ is a compact interval of $T$}\}$ such that $\mu_J = (\mu_I)|_J$ for all $J \subset I$. It is $F_n$-equivariant if $\mu_{g(I)}=(g|_I)_*\mu_I$ for all $I$ and $g \in F_n$.
The following theorem formalizes the relationship described above between the set of length
measures on the limiting tree $T$ to those on the folding sequence $(G_s)_{s
  \ge 0}$.

\begin{theorem}[{\cite[Proposition 5.4]{NPR2014}}] \label{thm:NPR-foldingsequence}
    Given a reduced folding sequence $(G_s)_{s \ge 0}$ with limiting tree $T$,
    there is a natural linear isomorphism between the space of length measures
    on $(G_s)_{s \ge 0}$ and the space of equivariant length measures on
    $T$.
    
\end{theorem}

A nonzero length measure on the limit tree $\hat{T}$ may well assign 0 to a non-degenerate
interval, so the associated $\R$-tree is obtained from $\hat{T}$ by 
collapsing subtrees.  
However, if $\hat T$ is arational (or, indeed \textit{mixing}; see e.g.\  \cite{morgan1988ergodic,morgan1992trees})  
then every nontrivial pseudometric will be a metric.

    Dual to any length measure is a \emph{width measure}. A \textbf{width measure}
    (or a \textbf{current}) on a folding/unfolding sequence $(G_s)_s$ is a
    sequence of non-negative vectors $(\vec{\mu}_s)_s$ in $\R^{|EG_s|}$ where for each $s$, the vectors satisfy the compatibility condition (`currents push forward'):
    \[
        \vec{\mu}_s = M^{s-1,s}\vec{\mu}_{s-1}.
    \]
    We consider the set $\calC((G_s)_{s\le 0})$ of width
    measures on the unfolding sequence $(G_s)_{s \le 0}$ as the cone
    of non-negative vectors in the finite-dimensional vector space
    whose vectors are sequences $(\vec{v}_s)_{s \le 0}$ with
    $\vec{v}_s \in \R_{\geq 0}^{|EG_s|}$  for every $s \le 0$ satisfying the compatibility condition
    \[
        \vec{v}_{s} = M^{s-1,s}\vec{v}_{s-1}
    \]
    for every $s \le 0$. The dimension of this space is nonzero and bounded above by $\liminf_{s \to -\infty} |EG_s|>0$.

    Unlike the case of folding sequences, the lengths of graphs in an unfolding
    sequence increase. Hence, an unfolding sequence does not necessarily converge
    to a tree, but after rescaling, it may have a nontrivial accumulation set.
    Instead of this accumulation set for an unfolding sequence, we consider a \emph{legal lamination} defined as follows.
        In an unfolding sequence of graphs, $(G_s)_{s \le 0}$, we say a path $p$ in
    $G_s$ is \textbf{legal} if it is legal with respect to $\phi_{s,0}:G_s\to
    G_0$ (equivalently, $\phi_{s,0}$ immerses $p$ in $G_0$). Denote by $\Omega_{\infty}^L(G_s)$ the set of bi-infinite legal
    paths in $G_s$. Now define the \textbf{legal lamination} $\Lambda$ of an unfolding
    sequence $(G_s)_{s \le 0}$ as:
    \[
      \Lambda:= \bigcap_{s \le 0}\phi_{s,0}(\Omega_{\infty}^L(G_s)),
    \]
    which is a subset of $\Omega_{\infty}^L(G_0)$. 
    Similar to the limiting tree
    of a folding sequence, the set of \emph{currents} (see \Cref{ssec:arationalprelim}) on the legal lamination characterizes
    the set of width measures on the unfolding sequence:

\begin{theorem}[{\cite[Theorem 4.4]{NPR2014}}] \label{thm:NPR-unfoldingsequence}
    Given a reduced unfolding sequence $(G_s)_{s \le 0}$ with legal lamination
    $\Lambda$, there is a natural linear isomorphism between the space of
    width measures supported on $(G_s)_{s \le 0}$ and the space of currents supported on $\Lambda$.
  \end{theorem}
  We sketch the proof of \Cref{thm:NPR-unfoldingsequence} at the end of \Cref{ssec:arationalprelim} as it is in unpublished work.

\subsection{Boundary of Outer space and laminations}
\label{ssec:boundaryCVn}

Let $\partial F_n$ be the Gromov boundary of the Cayley graph of $F_n$, which is
well-defined up to homeomorphism. Then define the \textbf{double boundary}
$\partial^2F_n$ as $\partial F_n \times \partial F_n$ minus its diagonal. It
is equipped with a natural \textbf{flip} action denoted by $i : \partial^2F_n \to
\partial^2F_n$, which swaps the two factors. We further observe that for a
finitely-generated subgroup $H \le F_n$, we can embed $\partial^2H \subset
\partial^2 F_n$ as $H$ is quasiconvex in $F_n$.
Moreover, the action of $F_n$ on the Cayley graph of $F_n$ by homeomorphisms
extends to $\partial F_n$ and hence to $\partial^2 F_n$. We define a
\textbf{lamination} $L$ as a nonempty, closed, $F_n$-invariant, flip-invariant
subset of $\partial^2F_n$.

We now want to change our perspective on Outer space. Recall
each element of $CV_n$ is a marked metric graph $(G, m, \ell)$ of unit volume. Then the
universal cover of $G$ is a tree, $T$. Denoting the covering map by $p:T \to G$,
we can define the length function $\wt{\ell}$ on the set of edges of $T$
by letting each edge $e$ of $T$ have length
$\wt{\ell}(e):=\ell(p(e))$. The length function $\wt{\ell}$ then extends to a path metric $d_T$ on $T$.
Note that by construction, $\wt{\ell}$ is equivariant under the $F_n$-covering space action. Because the covering space action is discrete and free, we can
regard each element in $CV_n$ as an $F_n$-tree with a discrete, free action by
isometries. Conversely, one has to restrict the set of $F_n$-trees to
\textbf{minimal} ones, namely those for which there is no $F_n$-invariant proper subtree. Then we can view $CV_n$ as:
\[
  CV_n = \{\text{$F_n$-tree } T\ |\ \text{discrete, free,
    minimal, and isometric $F_n$-action}\}/\sim,
\]
where $T \sim T'$ if and only if there is an $F_n$-equivariant homothety between
two $F_n$-trees
$T$ and $T'$. Weakening the equivalence condition to $F_n$-equivariant isometry
gives $cv_n$, the unprojectivized Outer space (See \Cref{ssec:graphsandCVn}.)

This new perspective on $CV_n$ helps us to discuss its boundary. Indeed,
Culler--Morgan \cite{culler1987groupactions} showed that minimal $F_n$-trees are
characterized by the associated \textbf{translation length function} $\ell_T: g \mapsto
\inf\left\{ d_t(x, g\cdot x)\ |\ x \in T\right\}$ on $F_n$. Therefore, with the
$F_n$-tree perspective on $CV_n$, we can embed $CV_n$ into $\R_{\geq 0}^{F_n}$, which
stays embedded in the projectivization $\bbP\R_{\geq 0}^{F_n}$. As
the image of $CV_n$ in $\bbP\R_{\geq 0}^{F_n}$
has a compact closure, we can define the boundary as
$\partial CV_n:= \ov{CV_n} \setminus CV_n$. The points on the
boundary also correspond to $F_n$-trees, but they are either not free or
not discrete. The $F_n$-trees in the compactification $\ov{CV_n}$ are collectively
characterized as \emph{very small} trees, by Bestvina--Feighn \cite{bestvina1994outerlimits},
Cohen--Lustig \cite{cohen1995verysmall}, and Horbez \cite{horbezverysmall}.

Like in the boundary of Teichm{\" u}ller space, we can associate a lamination
$L(T)$ to a point $T$ on the
boundary of Outer space. We construct $L(T)$ as
follows. From the $F_n$-action on the Cayley graph of $F_n$, each nontrivial
element $g \in F_n$ has a repelling fixed point $g^{-\infty}$ and an attracting
fixed point $g^\infty$ on the boundary $\partial F_n$. We define the
\textbf{dual lamination} $L(T)$ of $T$ as:
\[
  L(T):= \bigcap_{m=1}^\infty \overline{\left\{(g^{-\infty},g^\infty) \in \partial^2F_n\ |\ \ell_T(g)<\frac{1}{m}\right\}},
\]
which is indeed a lamination by construction. Intuitively, $L(T)$ records which
elements of $F_n$ act on $T$ with arbitrarily small translation length. 
Each element of $L(T)$ is called a \textbf{leaf}. We say a finitely-generated
subgroup $H$ of $F_n$ \textbf{carries a leaf of $L(T)$}, if there is a leaf
$\ell \in L(T)$ such that $\ell \in \partial^2H$. (Recall that $\partial^2H \subset
\partial^2F_n$.) This is equivalent to saying that either some element of $H$ fixes a point
in $T$ (thus having an element of translation length 0), or the
restricted action of $H$ to a minimal $H$-invariant subtree $T_H \subset
T$ is not discrete (thus having elements whose translation lengths converge to zero).

A train track structure on a graph is called \textbf{recurrent} if there is a legal loop
that crosses every edge. The following lemma will be used to prove \Cref{thm:arational_limiting}.

\begin{lemma}[{\cite[Lemma 7.1]{bestvina2015boundary}}]
  \label{lem:recurrentloop}
  Let $T \in CV_n$ and let $\Sigma$ be a fixed simplex in $CV_n$. Then there
  exists a point $T_0$ in $\Sigma$ such that every optimal map $f: T_0 \to T$
  induces a recurrent train track structure on $T_0/F_n$.
\end{lemma}

\subsection{Arational limiting objects}
\label{ssec:arationalprelim}

An $\R$-tree $T \in \ov{CV_n}$ is \textbf{arational} if for every proper free factor $H \le F_n$,
the $F_n$ action restricted to $H$ is free and discrete. According to Reynolds
\cite[Theorem 1.1]{reynolds2012reducing}, there are two kinds of arational
trees. The \textbf{geometric} arational trees are dual to a filling measured
lamination on a punctured surface (so the conjugacy class of the puncture is
elliptic). The \textbf{non-geometric} arational trees are free. A lamination $L$
is called \textbf{arational} if no leaf of $L$ is carried by a proper free
factor of $F_n$. Hence, $T \in \partial CV_n$ is arational if $L(T)$ is
arational. Bestvina--Reynolds \cite[Theorem 1.1]{bestvina2015boundary} proved
the boundary of the free factor complex, $FF_n$, 
for $F_n$ is homeomorphic to the quotient
space of the space of
arational trees by equating trees with the same dual lamination.
For $n \ge 3$, the free factor complex is the simplicial complex with vertices corresponding to
the conjugacy classes of nontrivial proper free factors of $F_n$ and, for $k=1,\ldots,n-2$, with
$k$-simplices $[A_0,\ldots,A_k]$ corresponding to the filtration $A_0 < A_1
<\ldots <A_k$ of proper containment up to conjugation. When $n=2$, we modify the
edge relation so that $A,B$ are joined by an edge if and only if $A,B$ form a
basis for $F_2$.

We consider a coarsely well-defined projection $\pi : cv_n \to FF_n$, given by
$\pi(T)$ for $T \in cv_n$ as the collection of free factors represented by
subgraphs of $T/F_n$. Note that the map $\pi$ factors through $CV_n$.

The usefulness of this projection $\pi$ to $FF_n$ is that
it sends geodesics in $cv_n$ (so, in particular, folding paths) to
quasi-geodesics in $FF_n$, up to reparametrization.
A \textbf{folding path} (or \textbf{unfolding path}) is a collection of metric graphs $G_t$
parametrized by $t \in [0,\infty]$ ($t \in (-\infty, 0]$ respectively), such that
for $t<t'$ we have maps $f_{tt'}:G_t\to G_{t'}$ and for $t<t'<t''$, we have 
$f_{tt''}=f_{t't''}f_{tt'}$ such that $f_{tt'}$ are isometries on edges,
but do not necessarily send vertices to vertices. Then a folding sequence (or
unfolding sequence) we defined earlier in \Cref{ssec:sequences} is just a discrete subset of a folding path (unfolding path,
respectively) such that vertices are mapped to vertices.

More precisely, for $K >0$ and a metric space $X$, we say a function $f:[0,\infty) \to X$ is a
\textbf{reparametrized K-quasi-geodesic} if there is a (finite or infinite) sequence $0=t_0<t_1<\ldots<\infty$ such that $\diam(f([t_i,t_{i+1}]))
\le K$ for all $i$ and $|i-j| \le d(f(t_i),f(t_j))+2$, and either:
\begin{itemize}
\item the sequence $(t_i)$ is infinite and $t_i \to \infty$, or
\item the sequence $(t_i)$ is finite and $\diam(f([t_m,\infty))) \le K$ where
  $t_m$ is the largest term in the sequence.
\end{itemize}
In other words, $f$ is a $K$-quasi-geodesic after some reparametrization by
$t_i$'s. Then as stated above, we have 
\begin{proposition}[{\cite[Corollary 6.5 and Proposition 9.2]{bestvina2014hyperbolicity}}]
  \label{prop:projectionQG}
  Let $G_t$ be a folding path in $cv_n$. Then $\pi(G_t)$ in $FF_n$ is a reparametrized $K$-quasi-geodesic
  with $K$ depending only on $n$.
\end{proposition}

To study arational trees further, we consider another `measure' on an
$\R$-tree other than length: a \emph{current}. 
Define a \textbf{current} as an $F_n$-invariant and flip-invariant Radon measure $\mu$ on
$\partial^2F_n$. See \cite{reiner:thesis,ilya:currents} for an
introduction to currents on free groups.
Denote by $\Curr(F_n)$ the set of currents, and by
$\bbP\Curr(F_n)$ the set of projective classes of currents.

For each element $g \in F_n$, we can construct a current $\eta_g \in \Curr(F_n)$
as follows. First, suppose $g$ is \textbf{primitive}, that is, there is no
element $h \in F_n$ and $m \ge 2$ such that $g=h^m$. Denote by $[g]$ the
conjugacy class of $g$. Then we define
\[
  \eta_g := \sum_{h \in [g]}\left( \delta_{(h^{-\infty},h^\infty)} +
    \delta_{(h^{\infty}, h^{-\infty})} \right),
\]
where $\delta_{(x,y)}$ denotes the Dirac measure concentrated at the point
$(x,y) \in \partial^2F_n$: namely, for $A \subset \partial^2F_n$:
\[
  \delta_{(x,y)}(A) =
  \begin{cases}
    1 & \text{ if $(x,y) \in A$,}\\
    0 & \text{ otherwise}.
  \end{cases}
\]
Now for a non-primitive element $g = h^m$ for some $h \in F_n$ and $m \ge 2$, we
define $\eta_g := m\eta_h$. Then we call $\eta_g$ the \textbf{counting current}
given by $g$.

Kapovich--Lustig \cite[Theorem A]{kapovich2009intersection} constructed the
$\Out(F_n)$-invariant continuous pairing, the \textbf{intersection form},
\[
  \langle  -, - \rangle: \ov{cv_n} \times \Curr(F_n) \to \R_{\ge 0},
\]
which is $\R_{\ge 0}$-homogeneous in the first coordinate, and $\R_{\ge 0}$-linear
in the second coordinate. Moreover, $\langle T, \eta_g \rangle =
\ell_T(g)$ for $T \in cv_n$ and all counting currents $\eta_g$. The intersection
form relates $\R$-trees and currents as dual objects.

Denote by $M_n$ the unique minimal closed invariant subset of
$\Curr(F_n)$ under the $\Out(F_n)$-action~\cite[Section 5.8]{martin1997nonuniquely}. 
Now, for each $F_n$-tree $T \in \ov{cv_n}$, we define $T^*$ as the set of
currents dual to $T$, namely:
\[
  T^*:= \{\mu \in  M_n\ |\ \langle T, \mu \rangle = 0\}.
\]

Here we collect some properties of $T^*$. 

\begin{lemma}[{\cite[Theorem 4.4 and Remark 4.6]{bestvina2015boundary}}]\label{lem:propsTstars}
  Let $T, T_1, T_2 \in \ov{cv_n}$. Then the following hold.
  \begin{enumerate}[label=(\roman*)]
  \item $T^* = \{0\}$ if and only if $T \not\in \partial cv_n$.
    \item $T$ is arational if and only if $T^* \neq \{0\}$ and any two
      nonzero $\mu_1,\mu_2 \in T^*$ have the same support.
    \item Suppose $T_1$ is arational and $T_1^* \subseteq T_2^*$. Then 
      $T_2$ is arational, $T_1^*=T_2^*$ and $L(T_1)=L(T_2)$.
    \item Suppose $T_2$ is arational and $\{0\} \neq T_1^* \subseteq T_2^*$. Then
      $T_1$ is arational, $T_1^*=T_2^*$ and $L(T_1)=L(T_2)$.
  \end{enumerate}
\end{lemma}

Note that for any trees $T_1, T_2\in \overline{cv_n}$ which are projectively equivalent, $T_1^* = T_2^*$. In particular, the property of a tree being arational is scale-invariant by \Cref{lem:propsTstars}(ii). Hence it is reasonable to consider the simplices of arational trees lying in $\overline{CV_n}$.

It is convenient to call two arational $\R$-trees $T_1,T_2$ {\bf
  equivalent} if $T_1^*=T_2^*$, or, equivalently, if $L(T_1)=L(T_2)$. An
equivalence class of arational trees forms a simplex in $\partial
CV_n$, and the subspace of arational trees with the equivalence
classes identified to a point is homeomorphic to the boundary of the
free factor complex, (see \cite{bestvina2015boundary} and
\cite{hamenstaedt2014boundary}).
Furthermore, equivalent arational trees are equivariantly homeomorphic
with respect to the observers' topology, (see \cite[Proposition
3.2]{bestvina2015boundary} and \cite[Proposition 7.1]{hamenstaedt2014boundary}).

In \Cref{thm:arational_limiting} below, we realize arational trees as limits of folding sequences. 
For the proof, we begin with producing a {\it folding path} $G_t$ converging to the desired
arational tree. 
 When a folding path $G_t$ in $cv_n$ converges to a tree $T
 \in \ov{CV_n}$, we mean the path given by the projection 
 $\hat{G}_t$ of $G_t$ to $CV_n$ converges to $T$.
Then, in \Cref{lem:arational_path_to_seq}, we will modify the folding path so that when
restricted to a discrete set of times we obtain a folding sequence.

The following two lemmas are the tools from Bestvina--Reynolds
\cite{bestvina2015boundary} that we need to produce the desired folding path.
\begin{lemma}[{\cite[Theorem 6.6]{bestvina2015boundary}}, U-turn lemma]
  \label{lem:U-turn}
  Let $\{[S_i, T_i]\}_{i \ge 0}$ be a sequence of folding paths in $CV_n$ beginning at $S_i$ and ending at $T_i$, and $U_i \in [S_i,T_i]$
  be a point on the path. Suppose $S_i \to S \in \ov{CV_n}$, $T_i \to T \in
  \ov{CV_n}$ and $U_i \to U \in \ov{CV_n}$. Then either $U^* \subset T^*$, or
  $S^* \cap U^* \neq \{0\}$.
\end{lemma}

\begin{lemma}[{\cite[Lemma 7.3]{bestvina2015boundary}}, realizing folding path]
  \label{lem:realizing_folding}
  Let $S_i,T_i \in CV_n$ for $i \ge 0$, and assume that
  \begin{enumerate}[label=(\alph*)]
  \item all $S_i$ are in the same simplex as $S_0$, 
  \item there is a morphism $f_i : S_i \to T_i$ such that the induced train track
    structure on $S_i$ is recurrent, and
  \item $T_i \to T \in \partial CV_n$.
  \end{enumerate}
  If $T$ is arational, then $S_i \to S \in CV_n$ and certain initial segments of the folding path
  $\gamma_i$ induced  by $f_i$ converge uniformly on compact sets to a folding
  path $\gamma$ from $S$ that converges to $U \in \partial CV_n$ with $U^*
  \subseteq T^*$.
  \end{lemma}

Now we are ready to prove that every arational tree on
the boundary of Outer space is \emph{reachable} by a folding
sequence.

\begin{theorem}
  \label{thm:arational_limiting}
  For each arational tree $T$ on $\partial CV_n$, there exists a reduced folding
  sequence limiting to $T$. \end{theorem}
\begin{proof}
  Let $T \in \partial CV_n$ be arational.
  Then there exists a sequence $\{T_i\}_{i \ge 0}$ in $CV_n$ converging to $T$. Fix a
  simplex $\Sigma$ in $CV_n$ represented by a marked rose $(R_n,m)$ with $n$
  petals. By \Cref{lem:recurrentloop}, for each $T_i$ we can find an $F_n$-tree $S_i \in
  \Sigma$ with an optimal map $f_i:S_i \to T_i$ inducing a recurrent train track
  structure on $S_i/F_n$. Let $c_i$ be a \emph{recurrent loop} in $S_i/F_n$; i.e. a legal loop in $S_i/F_n$ crossing every
  edge of $S_i/F_n$. As $S_i/F_n$ is a finite graph, so it can have at most finitely many train track
  structures. Hence, we may pass to a subsequence and assume $\{S_i/F_n\}_{i \ge 0}$
  have the same train track structure and the same recurrent
  loop. Taking a further subsequence, we may assume $S_i$
  converges to $S \in \ov{CV_n}$. We will show $S \in CV_n$ by proving that none of
  the petals in $S_i/F_n$ have lengths limiting to 0 or $\infty$. To that end, first note that 
  having a recurrent loop implies the length
  of each petal in $S_i/F_n$ does not blow up to infinity, before normalizing. On the other hand, if a petal's length limits to zero, the
  element of $F_n$ represented by it would have length 0 in $T$,
  which contradicts the assumption that $T$ is arational. All in all, $S
  \not\in\partial CV_n$, so $S \in CV_n$. In particular, by
  \Cref{lem:propsTstars}(i), it follows that $S^*=\{0\}$.

  Now, each optimal map $f_i:S_i \to T_i$ induces a folding path
  $[S_i,T_i]$. Then by \Cref{lem:realizing_folding}, we can find a folding
  path $\gamma$ from $S$ converging to some $U\in\partial CV_n$ and
  with $U^*\subseteq T^*$. By \Cref{lem:propsTstars}(iv), we have that
  $U$ is arational and equivalent to $T$. Then, \Cref{lem:arational_path_to_seq} promotes this
  folding path to a reduced folding sequence.
  Finally, by \Cref{thm:NPR-foldingsequence},
  modifying the length measures on the reduced folding sequence $(U_i)_{i \ge 0}$ converging to $U$, we obtain a reduced folding sequence converging to $T$.
\end{proof}

\begin{lemma}\label{lem:arational_path_to_seq}
  Suppose $G_t$ is a folding path converging to an arational tree
  $T$. Then there is a reduced folding sequence converging to $T$.
\end{lemma}

\begin{proof}
  Let $T\in \overline{CV_n}$ be arational and let $G_t$ be a folding path
  converging to $T$.
  First note that for any point $v_t\in G_t$, there is some $t'>t$ such
  that $f_{tt'}(v_t)\in G_{t'}$ is a vertex. Indeed, otherwise, the
  images of $v_t$ would define a free splitting of $F_n$ which does not
  change along the folding path. In particular, the path is eventually
  constant in the free factor complex (and even in the free splitting
  complex) so the limiting tree cannot be arational (\cite[Corollary
  4.5]{bestvina2015boundary}.) 
   Similarly, for every $t$ there is $t'>t$ so that every edge of $G_t$ maps onto $G_{t'}$, for otherwise we would have
  proper subgraphs $H_t\subset G_t$ that are preserved by $f_{tt'}$
  and eventually contain loops, so the path would be coarsely constant
  in $FF_n$ and would not converge to an arational tree.  

  Call a vertex of $G_t$ \textbf{rigid} if it has at least 3
  gates. Images of rigid vertices along the folding path are rigid
  vertices. First, suppose that $G_t$ eventually contains at least one
  rigid vertex. For a given $t$, choose $t'\gg t$ so that $G_{t'}$ contains
  a rigid vertex $v$ and, further, that the preimage of $v$ intersects every open edge of
  $G_t$. Note that the second condition will hold for some such $t'$
  by the eventual surjectivity of the maps $f_{tt'}$ when restricted
  to each edge of $G_t$. 
  Now,
  fold each illegal turn 
  based at non-rigid vertices of $G_t$ until
  either the folding reaches a point in the preimage of $v$, or else the illegal turn
  becomes legal and thus creates another rigid vertex. We call this folded graph
  $H$. 
  This folding procedure gives a factorization of 
   $G_t\to G_{t'}$ through $G_t\to H\to G_{t'}$, where the
  second map is a morphism which maps every vertex to
  a rigid vertex. Iterating this process for $G_{t'} \to G_{t''}$ for some $t''>t'$, the sequence of graphs $H$ produced in
  this way is the desired folding sequence.

  Now, suppose no $G_t$ has any rigid vertex. Pick a point $v_0\in
  G_0$ and let $v_t\in G_t$ be its image in $G_t$. Again by the first
  paragraph there is a sequence $t_i\to\infty$ so that $v_{t_i}$ is a
  vertex. Moreover, after passing to a subsequence, the preimage of $v_{t_{i+1}}$
  in $G_{t_i}$ intersects every open edge. Fold every illegal turn of
  $G_{t_i}$ not
  based at $v_{t_i}$ until a point in the preimage of $v_{t_{i+1}}$ is reached. This time, an
  illegal turn cannot become legal as that would create a rigid
  vertex. We again get a factorization as above where now all vertices
  of $H$ map to $v_{t_{i+1}}$ and this gives the desired sequence.

  In either case, the folding sequence is necessarily reduced, again by the first
  paragraph. 
\end{proof}

Similarly, we prove that every lamination dual to an arational tree is
realizable as the \emph{diagonal closure} of the legal lamination of an unfolding sequence.
Here we first relate the legal lamination in terms of trees in the
accumulation set of the unfolding sequence.

\begin{lemma}[{\cite[Page 5]{NPR2014}}]\label{lem:arational_legal_lamination}
  Let $(G_s)_{s \le 0}$ be an unfolding sequence with the legal lamination
  $\Lambda$. Then for any tree $T$ in the accumulation set of $(G_s)_{s \le 0}$,
  \[
    \Lambda \subseteq L(T).
  \]
\end{lemma}

For a subset $X \subset \partial^2F_n$, we say a leaf $\ell=(g,h) \in \partial^2F_n$ is
\textbf{diagonal over $X$} if there exist $p>1$ and $x_1,\ldots, x_{p-1} \in
\partial F_n$ such that
\[
  (g, x_1), (x_1, x_2), \ldots, (x_{p-1},h) \in X.
\]
We say $X$ is \textbf{diagonally closed} if every leaf in
$\partial^2F_n$ diagonal over $X$ is contained in $X$. We remark that dual
laminations of trees are diagonally closed \cite[Page 751]{CHL2008RtreeII}, but
legal laminations may not be.

A result of Coulbois--Hilion--Reynolds \cite[Theorem A]{CHR2015indecomposable},
shows that the dual lamination of an indecomposable $F_n$-tree is minimal up to
diagonal leaves. That is, if there is a sublamination $\Lambda \subseteq L(T)$,
then $\overline{\Lambda} = L(T)$, where $\overline{\Lambda}$ denotes the
diagonal closure of $\Lambda$. As every arational tree is indecomposable
\cite[Theorem 1.1]{reynolds2012reducing}, we have the following lemma which
states that the dual lamination of an arational tree is \emph{almost
  minimal}.

\begin{lemma}\label{lem:arational_minimal}
  Let $T$ be an arational tree in $\partial CV_n$. Then $L(T)$ is minimal up to
  diagonal leaves.
\end{lemma}
\begin{proof}
  An arational tree is indecomposable by Reynolds \cite[Theorem
  1.1]{reynolds2012reducing}, and an indecomposable tree has dual lamination that
  is minimal up to diagonal leaves by \cite[Theorem A]{CHR2015indecomposable}.
\end{proof}

We are now ready to show that the dual lamination of an arational tree can be realized as the diagonal closure of the legal lamination for an unfolding sequence.

\begin{theorem}\label{thm:arational_limiting_unfolding}
  Let $\Lambda$ be a lamination that is dual to an arational tree $T$.
  Then there exists a reduced unfolding sequence whose legal lamination's diagonal
  closure is $\Lambda$.
\end{theorem}

\begin{proof}
  Let $\Lambda = L(T)$ for some arational tree $T$. Then $T\in \partial CV_n$,
  so we can take a sequence of trees $(T_i)_{i \ge 0}$ in $CV_n$ converging to $T$.
  Now fix a tree $S \in CV_n$. Using \Cref{lem:recurrentloop}, for each $i \ge 0$ we can find
  $T_i'$ from the simplex of $CV_n$ containing $T_i$ such that there is an
  optimal map $f_i : T_i' \to S$ inducing a recurrent train track structure on $T_i'$. Then both $(T_i)_{i \ge
    0}$ and $(T_i')_{i \ge 0}$ project to the same sequence of points in the free factor
  complex converging to a point on the boundary. Hence, denoting by $T'$ the
  limit of $(T_i')_{i \ge 0}$, we have that $T$ and $T'$ lie in the same simplex
  in $\partial CV_n$, so $T'$ is arational. Now let $\gamma_i$ be the
  folding (also unfolding, as it is a finite path)
  path induced by $f_i:T_i' \to S$. We claim:
  \begin{claime}[Realizing  unfolding path]
    After taking a subsequence of $(\gamma_i)_{i \le 0}$, it converges uniformly on compact sets
    to an unfolding path (ray) $\gamma$ to $S$ with accumulation set containing
    $T'$.
  \end{claime}
  The claim will conclude our
  proof, as the legal lamination $\Lambda'$ of an induced unfolding sequence
  from $\gamma$ has to be contained in $L(T')$ by
  \Cref{lem:arational_legal_lamination}.
  Since $L(T')$ is diagonally closed, the
  diagonal closure $\ov{\Lambda'}$ is also contained in $L(T')$. Then we can conclude
  $\ov{\Lambda'} = L(T')=L(T)$ by \Cref{lem:arational_minimal}. 

  Hence, it suffices to prove the claim.
  Giving the natural parametrization on $(\gamma_i)_{i \ge 0}$, followed by the
  reparametrization so that the terminal points of $\gamma_i$'s are parametrized by 0, we can write 
  $\gamma_i:[t_i,0] \to \ov{CV_n}$.
  We then extend each folding path $\gamma_i$ by a constant, so that we can rewrite
  $\gamma_i:(-\infty,0] \to \ov{CV_n}$ as an unfolding ray. Note that $(\gamma_i)_{i \ge 0}$ is
  pointwise bounded by the fact that the function $d_S: \tau\in CV_n \mapsto d(\tau,
  S)$ is proper. Indeed, for each $u \in (-\infty, 0]$, we have
  $K=d_S^{-1}([0,-u])$ is compact, so for each $i \le 0$,  $\gamma_i(u) \in K$
  is bounded. Also, the $\gamma_i$'s are equicontinuous by the same natural
  parametrization. Therefore, by the Arz{\'e}la--Ascoli theorem (See e.g.  \cite{collins2007asymmetric}), our rewritten sequence of unfolding rays 
  $(\gamma_i)_{i \ge 0}$ has a
  subsequence converging to a ray $\gamma:(-\infty, 0] \to \ov{CV_n}$ with $\gamma(0)=S$. We note
  that as $\gamma_i(t_i)=T_i' \to T'$ as $i \to -\infty$, we have that the accumulation
  set of $\gamma$ contains $T'$. The following lemma promotes this
  unfolding path to an unfolding sequence. \end{proof}

\begin{lemma}
  Let $G_t$ for $t\in (-\infty,0]$  be an unfolding path whose legal
    lamination, after taking its diagonal closure, is dual to an arational tree (or,
    equivalently, the projection of $G_t$ to $FF_n$ converges to a point in $\partial FF_n$). Then,
    there is a reduced unfolding sequence whose legal lamination has the same diagonal closure as that of $G_t$.
\end{lemma}

\begin{proof}
  This is a variant of the proof of
  \Cref{lem:arational_path_to_seq}. First note that the volume and the injectivity radius
  of $G_t$ go to $\infty$ as $t\to -\infty$ (or else the projection to
  $FF_n$ would eventually be constant). We again consider two
  cases.

  The first case is that for arbitrarily large $-t$, the graph $G_{t}$
  contains a rigid vertex. Fix some $t$. Then there is $t'<t$ so that
  every noncontractible subgraph of $G_{t'}$ maps onto all of $G_t$
  (if this fails for all $t'$, then a truncation of the path would be coarsely constant in
  $FF_n$). Consider all points in $G_{t'}$ that map to a rigid vertex
  in $G_t$. By the choice of $t'$, all complementary components of this
  set are trees. Now fold each illegal turn in each component until
  it either becomes legal and turns into a rigid vertex, or it reaches
  the boundary of the component, i.e., one of the points that maps to a
  rigid vertex in $G_{t}$. Thus we factored $G_{t'}\to G_t$ as
  $G_{t'}\to H_t\to G_t$ with the second map sending all vertices to
  rigid vertices. The
  graphs $H_t$ obtained in this way when $t\to -\infty$ will form an
  unfolding sequence with the same legal lamination.

  The second case is that after truncation there are no rigid vertices
  in any $G_t$. We choose $v_t\in G_t$ for every $t\leq 0$ so that for $t'<t$,
  $G_{t'}\to G_t$ maps $v_{t'}$ to $v_t$.

  Then, we can pass to a subsequence $t_i$ with $t_i\to-\infty$ such that each $v_{t_i}$ is a vertex of $G_{t_i}$ 
  (otherwise a truncation
  is coarsely constant in $FF_n$). As above, after a further
  subsequence, we may assume that the preimage of $v_{t_i}$ in
  $G_{t_{i-1}}$ has contractible complementary components. Folding
  illegal turns within each component produces a factorization
  $G_{t_{i-1}}\to H\to G_{t_i}$ with all vertices of $H$ mapping to
  $v_{t_i}$, and we get an unfolding sequence as above.

  It remains to arrange that the unfolding sequence, which we rename $G_n$, is reduced. Here this is not
  automatic as in the folding sequence case. However, any proper invariant
  subgraphs $H_n\subset G_n$ must be forests for large $-n$, since
  otherwise a truncated sequence would be coarsely constant in
  $FF_n$. After truncating we may assume all $H_n$ are forests. Let
  $G_n'$ be obtained from $G_n$ by collapsing each component of
  $H_n$. There are induced maps $G_{n-1}'\to G_n'$. The new sequence
  $G_n'$ is a reduced unfolding sequence, see \cite[Lemma
    4.1]{handel-mosher:FSishyperbolic}. Moreover, the new sequence
  fellow travels the old in $FS_n$ and hence in $FF_n$, so it
  converges to the same point in $\partial FF_n$ per \Cref{prop:projectionQG}.
\end{proof}

We also note that it is possible that an unfolding sequence
accumulates on the entire simplex of equivalent arational trees, see
\cite{bestvina2024limit}.

Before we move on the next section, we give a sketch of the proof of \Cref{thm:NPR-unfoldingsequence}.
\begin{proof}[Proof of \Cref{thm:NPR-unfoldingsequence}]
    Given a compatible collection of width measures $\{\mu_n\}_{n=0}^{-\infty}$ on a reduced unfolding sequence, we construct a current $\mu$ on the legal lamination $\Lambda$ of the sequence as follows. For finite directed paths $\gamma$ and $\eta$ in $G_0$, denote by $\langle\gamma,\eta\rangle$ the number of occurrences of $\gamma$ in $\eta$ as subwords, where letters are the edges of $G_0$. Then for each $n \in \Z_{\le 0}$ and a finite path $\gamma$ in $G_0$, we define
  \[
    \alpha_n(\gamma):= \sum_{e_n \in G_n} \mu_n(e_n) \langle \gamma, \phi_{n,0}(e_n)\rangle.
  \]
  It can be shown that the $\alpha_n(\gamma)$ is increasing as $n \to -\infty$, and $\{\alpha_n(\gamma)\}_{n \in \Z_{\le 0}}$ is bounded above. Hence, for each finite path $\gamma$ in $G_0$, the quantity $\mu(\gamma):= \lim_{n \to -\infty} \alpha_n(\gamma)$ is well defined. Then the set of numbers 
  \[
  \{\mu(\gamma) \mid \text{$\gamma$ is a finite path in $G_0$}\}
  \]
  satisfies the following: For each finite path $\gamma$ in $G_0$ and $\gamma_1\ldots \gamma_k$ are all possible extensions of $\gamma$ by an edge at the end of $\gamma$,
  \[
    \mu(\gamma) = \sum_{i=1}^k \mu(\gamma_i).
  \]
  By \cite[Lemma 7]{reiner:thesis} (called ``Kolmogorov extension property'' in \cite[Section 2.2]{NPR2014}), it follows $\mu$ is a current on $\pi_1(G_0)=F_n$. Finally, to show $\mu$ is supported on $\Lambda$, let $\eta$ be a leaf in the support of $\mu$. This means for every finite subpath $\gamma$ of $\eta$, we have $\mu(\gamma)>0$. By definition of $\mu$, this implies for every $m$ with sufficiently large $|m|$, there exists an edge $e'_m$ of $G_m$ such that $\langle \gamma, \phi_{m,0}(e'_m) \rangle>0$. Since $\phi_{m,0}(e'_m)=\phi_{n,0}(\phi_{m,n}(e'_m))$ is a concatenation of the $\phi_{n,0}$-image of an edge of $G_n$, taking $\gamma$ to be longer than the twice of the simplicial length of $\phi_{n,0}(e_n)$ for every $e_n \in G_n$, it follows that $\phi_{n,0}(e_n)$ must be a subpath of $\gamma$ for some $e_n$ of $G_n$, so it is a subpath of $\eta$. This implies that $\eta$ is a limit of $\phi_{n,0}(e_n)$ for a sequence of edges $e_n$ in $G_n$, so $\eta$ is contained in $\Lambda$, concluding the proof.
\end{proof}

\subsection{Transverse decomposition from ergodic measures}
\label{ssec:ergodic}
In this section, we present the Transverse Decomposition theorem for
(un)folding sequences. This decomposition of a graph into subgraphs reflects 
the simplex of measures on the corresponding limiting objects. It was first
proved by
Namazi--Pettet--Reynolds \cite{NPR2014}, and it was motivated by
the corresponding McMullen's theorem \cite{mcmullen2013diophantine} in the setting of Teichm\"uller
geodesics. We give an alternative proof of the Transverse
Decomposition theorem 
using only linear algebraic arguments. This is a vital ingredient toward
establishing the upper bound in \Cref{thm:Main_Upperbound}.

Recall that for the limiting tree $T$ of a reduced folding sequence and the
legal lamination $\Lambda$ of a reduced unfolding sequence, we denote by
$\cL(T)$ and by $\calC(\Lambda)$ the set of length measures on $T$ and the set
of currents on $\Lambda$, respectively. We then denote by $\bbP\cL(T)$ and
$\bbP\calC(\Lambda)$ the projectivizations of these spaces. These projectivized spaces can
be thought of concretely as spaces of length measures or currents,
normalized so that a particular conjugacy class has length 1 or
compact set has measure 1, respectively. In this way, $\bbP\cL(T)$ and $\bbP\calC(\Lambda)$ become compact, convex
sets in a locally convex topological vector space,
and we call their extremal points \textbf{ergodic}. If
$\bbP\cL(T)$ or $\bbP\calC(\Lambda)$ is a singleton, then we call $T$ or $\Lambda$, respectively,
\textbf{uniquely ergodic} and otherwise call $T$ or $\Lambda$ \textbf{nonuniquely ergodic}.

A \textbf{transverse decomposition} of a graph $G$ is a collection 
$H^0,H^1,\ldots,H^k$ of subgraphs of $G$ such that the edge set of $G$ is the \emph{disjoint union} of
the edge sets of $H^0,H^1,\ldots,H^k$. We denote the decomposition by $G =
H^0\sqcup H^1\sqcup\ldots \sqcup H^k$. Note that the vertices may overlap between the subgraphs
of the transverse decomposition. 

To state the transverse decomposition
results, we make some notations as follows.
Given a folding/unfolding sequence $(G_s)_s$, we
pass to a subsequence so that all $G_s$ are isomorphic graphs and we
take a fixed isomorphism from $G=G_0$ to $G_s$. Denote
by $e_s$ the edge of $G_s$ identified to $e$ in $G$ by the isomorphism. Then, for
a length measure $\ell$ on the limiting tree $T$ of a folding sequence $(G_s)_{s \ge
  0}$, we denote by $\ell_s$ the induced length measure on $G_s$ under the
isomorphism $\cL(T) \cong \cL((G_s)_{s \ge 0})$ given in
\Cref{thm:NPR-foldingsequence}. Also, for $e_s \in G_s$, we write
$\ell(e_s):=\ell_s(e_s)$.
Similarly, for a current $\mu$ on the legal lamination $\Lambda$ of an unfolding
sequence $(G_s)_{s \le 0}$, we denote by $\mu_s$ the induced current on $G_s$
under the isomorphism $\calC(\Lambda) \cong \calC((G_s)_{s \le 0})$ given in
\Cref{thm:NPR-unfoldingsequence}. We also write $\mu(e_s):=\mu_s(e_s)$ for $e_s \in
G_s$.

Now, we state the Transverse Decomposition theorems and give another proof of
the existence of such a decomposition. The original proofs are given in
\cite[Theorems 4.7, 5.6]{NPR2014}.

\begin{theorem}[Transverse Decomposition theorem for a Folding sequence] \label{thm:NPRdecomp-folding}
  Let $(G_s)_{s \ge 0}$ be a reduced folding sequence with limiting tree 
  $T$ and let $\ell^1, \ldots, \ell^k$ be the collection of all the distinct (up to scale) ergodic length measures supported on $T$. 
  Let $(\mu_s)_{s \ge 0}$ be the \emph{frequency current} on the folding
  sequence, namely $\mu_0(e)=1$ for every edge $e$ of $G_0$. Then, after passing
  to a subsequence, there is a transverse decomposition
  \[
    G_0 = H^0 \sqcup H^1 \sqcup \ldots \sqcup H^k,
  \]
  and 
  \begin{align*}
    H^i &:= \{e \in EG_0\ |\ \lim_{s \to \infty}\mu(e_s)\ell^i(e_s)>0\},\  \text{for $i=1,\ldots,k$}, \text{and}\\
    H^0 &:= \{e \in EG_0\ |\ \lim_{s \to \infty}\mu(e_s)\ell^i(e_s)=0,\ \text{for
          all $i=1,\ldots,k$}\}.
  \end{align*}
\end{theorem}

\begin{theorem}[Transverse Decomposition theorem for an Unfolding sequence] \label{thm:NPRdecomp-unfolding}
  Let $(G_s)_{s \le 0}$ be a reduced unfolding sequence with legal lamination
  $\Lambda$ and let $\mu^1,\ldots,\mu^k$ be the collection of all distinct (up to scale) ergodic probability currents supported on $\Lambda$. Let
  $(\ell_s)_{s \le 0}$ be the \emph{simplicial length measure} on the unfolding
  sequence, namely $\ell_0(e)=1$ for every edge $e$ of $G_0$. Then, after passing
  to a subsequence, there is a transverse decomposition
  \[
    G_0 = H^0 \sqcup H^1 \sqcup \ldots \sqcup H^k,
  \]
  where
  \begin{align*}
    H^i &:= \{e \in EG_0\ |\ \lim_{s \to -\infty}\ell(e_s)\mu^i(e_s)>0\}, \ \text{for $i=1,\ldots,k$}, \text{and} \\
    H^0 &:= \{e \in EG_0\ |\ \lim_{s \to -\infty}\ell(e_s)\mu^i(e_s)=0,\ \text{for
          all $i=1,\ldots,k$}\}.
  \end{align*}
\end{theorem}

We say an edge $e \in H^i$ in the transverse decomposition of $G_0$ for a folding (or unfolding) sequence with $i>0$ \textbf{supports} the ergodic length measure $\ell^i$ (or the ergodic probability current $\mu^i$, respectively.) This term will be used in \Cref{sec:upperbound}.

We start by giving a proof for the unfolding sequence \Cref{thm:NPRdecomp-unfolding}, and
afterwards indicate the small changes needed in the proof for the folding sequence \Cref{thm:NPRdecomp-folding}.

The setup is as follows.
Suppose we are given an unfolding sequence, $(G_s)_{s \le 0}$, where each $G_i \cong G$ and
length functions are given by pulling back the simplicial length function
$\ell_0$ on $G_0$.
Let $V$ be a vector space with basis $\{e^1,\ldots,e^n\}=EG$.
To each $G_s$ we assign the vector space $V_s \cong V$, but we put a norm on $V_s$
with basis $\{e_s^1,\ldots,e_s^n\}=EG_s$, such that $\| e^i_s \|_{V_s}:= \ell_{s}(e^i_s)$ and extend it to an $\ell^1$-norm on
$V_s$.

A nonnegative linear combination $a_1e^{1}_s+\ldots + a_ne^{n}_s$ with $a_i \ge 0$
can be interpreted as an assignment of thicknesses to edges, which can
then be viewed as rectangles, and then
$a_1\|e^{1}_s\|_{V_s}+\ldots+a_n\|e^{n}_s\|_{V_s}$ is the total area of the graph $G_s$. In other words,
an element $\mu$ in the nonnegative cone $V_s^{\ge 0}$ of $V_s$ is an assignment of thicknesses
on $EG_s$,
and $\|\mu\|_{\ell^1}$ is the area, used to define the $H^i$'s in the transverse decomposition.

Transition maps on the unfolding sequence $(G_s)_{s\le 0}$ induce linear maps
\[
  \ldots \longrightarrow V_{-2} \longrightarrow V_{-1} \longrightarrow V_0
\]
that are norm nonincreasing, and on the positive cones they are
norm (and area) preserving.
Hence, we can restrict to the simplex $\Delta_s:=\{\mu \in V_s^{\ge 0}\ |\ \|\mu\|_{\ell^1}=1\} \subset V_s$ to obtain
affine maps
\[
  \ldots \longrightarrow \Delta_{-2} \longrightarrow \Delta_{-1} \longrightarrow \Delta_0.
\]
Denote these affine maps and their compositions by $S_{r,s}:\Delta_r \to \Delta_s$
for $r < s \le 0$,
so that $S_{s,t}\circ S_{r,s} = S_{r,t}$ for $r < s < t \le 0$. Since all $V_s$'s are isomorphic, it
follows that all $\Delta_s$'s are isomorphic. Identify all $V_s$'s with $V$ and $\Delta_s$ with a
fixed $\Delta \subset V$. Note that each $S_{r,s}$ is 1-Lipschitz. Hence by Arz{\'e}la--Ascoli, after taking a
subsequence, we have the limits
\begin{align*}
  S_s &:= \lim_{r \to -\infty} S_{r,s}:\Delta \to \Delta,\\
  S &:= \lim_{s \to -\infty} S_{s}: \Delta \to \Delta.
\end{align*}
\begin{lemma}\label{lem:S_Retraction}
  The limit $S:\Delta \to \Delta$ is a retraction. That is, $S^2=S$.
\end{lemma}
\begin{proof}
  From the compatibility condition $S_{s,t}\circ S_{r,s} = S_{r,t}$, we pass to the
  limit $r \to -\infty$ to get
  \[
    S_{s,t} \circ S_{s} = S_{t}.
  \]
  Now take the limit $s \to -\infty$ to get
  \[
    S_t \circ S = S_t.
  \]
  Finally, taking $t \to -\infty$ gives the desired identity $S^2 = S$. 
\end{proof}

\begin{corollary}\label{cor:simplex_decomposition}
  The image of $S$ is a simplex. Moreover, $\Delta$ can be written as a join of
  disjoint faces $\Delta^0, \ldots, \Delta^k$ of $\Delta$:
  \[
    \Delta = \Delta^0 \ast \Delta^1 \ast \ldots \ast \Delta^k,
  \]
  such that the image of $S$ is the convex hull of a finite set that consists of
  exactly one point $q_i \in \mathrm{int}(\Delta^i)$ for $i=1,\ldots,k$.
\end{corollary}

\begin{proof}
  Since $S^2=S$, the points in the image of $S$ are fixed. Now take a segment
  $[x,y]\subset \Delta$ with $x \neq y \in \textrm{Im}(S)$. Then since $V$
  is a vector space, there is a unique line $l\subset V$
  containing $[x,y]$. As $S$ is an affine map, it then follows that $S$ fixes $l
  \cap \Delta$. Hence, $l \cap \Delta \subset \textrm{Im}(S)$. 
  Therefore, we can
  consider the smallest affine plane $H$ in $V$ that contains $\textrm{Im}(S)$.
  This gives $\textrm{Im}(S)=H \cap \Delta$.

  Since $\textrm{Im}(S)$ is an affine image of a simplex, it is the convex hull of finitely many extremal points. 
  Let $q_1,\ldots,q_k$ be the vertices of
    $\textrm{Im}(S) \subset \Delta$. Let $\Delta^i$ be the unique
  smallest face of $\Delta$ that contains $q_i$. By the previous
  paragraph, the intersection $\Delta^i\cap \textrm{Im}(S)$ is precisely $\{q_i\}$.
  From the fact that $q_i$ is fixed and that $S$ is affine, we can see that $S$ maps
  $\Delta^i$ to $q_i$. Indeed, if we write $q_i$ as the
  convex combination of the vertices $v_k$ of $\Delta^i$, $q_i=\sum a_kv_k$, then $a_k>0$ and
  $q_i=S(q_i)=\sum a_k S(v_k)$. This shows that $S(v_k)$ is a convex
  combination of vertices in $\Delta^i$, so we must have that $S(v_k)=q_i$ for each vertex $v_k$ of $\Delta^i$. 

  In particular, this implies that $\Delta^i$ and $\Delta^j$ for $i \neq j$ do not have
    common vertices. Now let $\Delta^0$ be the span of all vertices of $\Delta$ that are
    \emph{not} in any of $\Delta^1,\ldots,\Delta^k$. This finishes the
    proof.
    \end{proof}  
    We note that the vertices $q_i$ in the above proof represent ergodic measures.

It remains to relate this linear algebra with the statement of the Transverse
Decomposition theorem for an unfolding sequence.

\begin{proof}[Proof of \Cref{thm:NPRdecomp-unfolding}]
  Let $C_s \subset \Delta$ denote the image of $S_s$. The maps $S_{r,s}$
  restrict to surjective affine maps $C_r\to C_s$, so
  $\dim C_r \ge \dim C_s$ for $r
  \le s$. Hence, $\dim C_r\leq \dim \Delta$ stabilizes for sufficiently large $-r$.

  However, if $C_s \to C_{s+1}$ is affine, onto, and $\dim C_s = \dim C_{s+1}$,
  then the map $S_{s, s+1}$ is also one to one. Hence, eventually all $C_s$ are affinely
  isomorphic. By taking limits, we also get that $\textrm{Im}(S)$ is naturally
  affinely isomorphic to each $C_s$ for sufficiently large $-s$. Indeed, we have
  that $S_t\circ S = S_t$, so $S_t: \textrm{Im}(S) \to \textrm{Im}(S_{t})$ is
  surjective, implying for large $-t$, this map is an isomorphism.

  For $i=0,\ldots,k$ we will define $H^i$ to be the subgraph whose edges
  consist of the vertices of $\Delta^i$. See \Cref{fig:decomposition}.
  Since the image
of $S$ is contained in $\Delta^1*...*\Delta^k$ every column of $S$ has zero
$e$-coordinate for $e\in \Delta^0$. Let $e\in \Delta^0$. Since the
  $e$-coordinates of $S_s$ have the form $\ell(e_s)\mu^j(e_s)$ and
  $\lim S_s=S$, we deduce
  $\lim_{s \to -\infty}\ell(e_s)\mu^j(e_s)=0$ for each $j$, justifying
  $e \in H^0$. Similarly, $H^i$ for $i \neq 0$ is simply the set of
  vertices of $\Delta^i$. This means for $e \in H^i$, the $e$-coordinate of
  $S(e)$ is nonzero, and this similarly implies $\lim_{s \to
    -\infty}\ell(e_s)\mu^i(e_s)\neq 0$. Likewise, if $j>0$ and $j\neq
  i$, and if $e'\in H^j$, then the $e$-coordinate of $S(e')$ is 0, so
  $\lim_{s \to -\infty}\ell(e_S)\mu^j(e_s) = 0$.   This
  concludes the proof. 
  \begin{figure}[ht!]
    \centering
    \includegraphics[width=\textwidth]{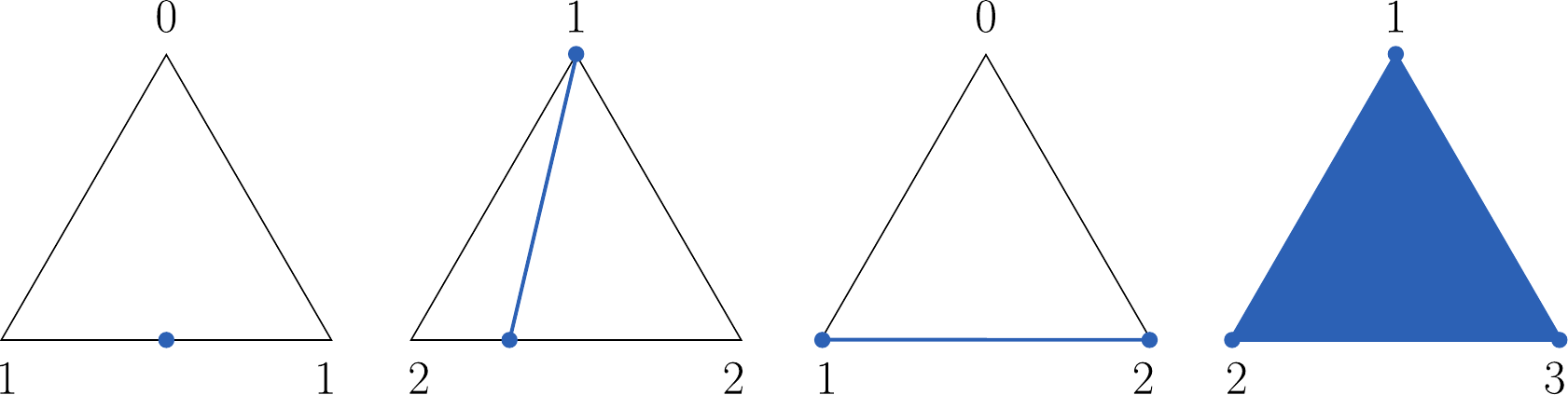}
    \caption{The image of $S$ (in blue) and how the vertices of $\Delta$, which
      correspond to the edges of $G$, are mapped via $S$. The
      vertices labeled by $i$ correspond to the edges of the component $H^i$ in the transverse
      decomposition. The map $S$ is the retraction that maps $\Delta$ onto the
      simplex $\textrm{Im}(S)$.}
    \label{fig:decomposition}
  \end{figure}
\end{proof}

\begin{proof}[Proof of \Cref{thm:NPRdecomp-folding}]
This time, we give a norm on the basis $EG$ for the vector space $V_s$
associated with $G_s$ using the frequency current, and interpret a nonnegative
linear combination in $V_s$ as assigning lengths to edges by their coefficients.
We then run the same argument by using that the \emph{transpose} of the transition
matrices induce the linear inverse system $\ldots \to V_2 \to V_1 \to V_0$, which
further reduces to an affine inverse system $\ldots \to \Delta_2 \to \Delta_1
\to \Delta_0$.
(Recall the compatibility condition for the length measures on the folding
sequence in \Cref{ssec:sequences} uses the \emph{transpose}; length measures
pull back.)
For $s > r$, let
$S_{s,r}:\Delta_s \to \Delta_r$ and identify $\Delta_r \cong \Delta$ for all $r
\ge 0$. Then define $S_{r}$ and $S$ as before; $S_r:=\lim_{s \to \infty}S_{s,r}$
and $S:=\lim_{r \to \infty}S_r$, after taking subsequences. Then $S$ is a
retraction. Since the folding sequence is reduced, the inverse
limit of the image $S_r$ exists, and it is exactly the space of length measures
on the folding sequence. We now follow the same proof of
\Cref{cor:simplex_decomposition} to establish the desired decomposition.
\end{proof}

For the proof of the upper bound given later, we mark here an observation that appeared in
the proofs of \Cref{thm:NPRdecomp-folding,thm:NPRdecomp-unfolding}.
For folding/unfolding sequences, for $r<s$, the transition map $G_r \to G_s$
induces a linear map between the associated vector spaces $V_r$ and $V_s$. For an
unfolding sequence, it is $V_r \to V_s$ and for a folding sequence, it is $V_s \to
V_r$. With normalized bases, denote by $\wt{M}^{r,s}$ the transition matrix
associated to the linear map induced from $G_r \to G_s$.
Denoting by $M^{r,s}$
the transition matrix for the transition map $G_r \to G_s$, the two
transition matrices are related as follows.
\begin{lemma}\label{lem:transition_relation}
  For the unfolding sequence, we have
  \begin{align*}
    \wt{M}^{r,s} &= \diag\left(\ell_s(e^1),\ \cdots, \ \ell_s(e^n)\right) M^{r,s} \diag\left(\frac{1}{\ell_r(e^1)},\ \cdots, \ \frac{1}{\ell_r(e^n)}\right).
  \end{align*}
  For the folding sequence,
  \begin{align*}
    \wt{M}^{r,s} &= \diag\left(\mu_r(e^1),\ \cdots, \ \mu_r(e^n)\right) (M^{r,s})^T \diag\left(\frac{1}{\mu_s(e^1)},\ \cdots, \ \frac{1}{\mu_s(e^n)}\right).
  \end{align*}
\end{lemma}
\begin{proof}
  This is because the bases for $V_s$ are normalized by the
  simplicial length for the unfolding sequence, and by the frequency current for
  the folding sequence.
\end{proof}
Now, the transverse decomposition results tell that up to relabeling the edges,
$\wt{M}^{r,s}$ is almost block diagonal with respect to the transverse decomposition.

\begin{corollary}\label{cor:normalized_diagonal}
  Relabel the edges $e^1,\ldots,e^n$ of $G$ so that for the transverse decomposition
  $EG=H^0\sqcup H^1 \sqcup \ldots \sqcup H^k$, each $H^i$ consists of edges labeled by
  consecutive integers. Then the transition matrix corresponding to $S$ from
  \Cref{lem:S_Retraction} equals
  \[
    \wt{M}:=\lim_{s \to -\infty}\lim_{\substack{r<s\\ r\to -\infty}}\wt{M}^{r,s}
  \]
  and it has a block diagonal submatrix with positive matrices corresponding to the
  edges in $G \setminus H^0$. A similar statement holds for folding sequences.
  \end{corollary}

  \begin{proof}
    This follows from the fact that $S$ is a retraction and $\textrm{Im}(S)$ is a simplex, whose
    vertices are linear combinations of a disjoint set of vertices of $\Delta$.
    Each face $\Delta^i$ of $\Delta$
    in \Cref{cor:simplex_decomposition} forms a diagonal block for
    $\wt{M}$. Since each vertex corresponding to $H^i$ with $i>0$ is a positive
    linear combination of vertices in $H^i$, the diagonal blocks for $H^1,\ldots,H^k$ are positive matrices.
    \end{proof}

\section{Lower Bound $2n-6$:  Folding sequence}  \label{sec:lowerbound}
In this section, we give a lower bound for the maximum number of linearly independent ergodic measures supported on the limiting tree of a given folding sequence of graphs of rank $n$.

\subsection{Diagonal-dominant matrices for folding sequences}
\label{ssec:ddmatrix_folding}

Let $n$ be a positive integer. Here we consider a sequence of $n \times n$
matrices with diagonal entries being dominant over other off-diagonal entries,
such that any consecutive product remains diagonal-dominant. 

\begin{definition} \label{def:klp_folding}
  Let $\eps>0$, and let $m$ be a positive integer such that $m \le n$. An $n \times n$ nonnegative matrix $A =
  (a_{ij})$ with positive diagonal entries is called an
  \textbf{$(m,\eps)$-matrix} if the following holds:
  \[
    \frac{a_{ij}}{a_{jj}} < \eps, \qquad \text{for $j \le m$ and $i \neq j$.}
  \]
\end{definition}
In \Cref{sec:lowerbound}, we
will primarily be considering $(m, \eps)$-matrices with small $\eps$'s
converging to zero, and will
refer to these matrices as \emph{diagonal-dominant} matrices.
We will refer to the first $m$ columns of $A$ as \textbf{Tier 1} columns and will call the remaining columns \textbf{Tier 2} columns.
Then for an $n \times n$ matrix $A=(a_{ij})$, denote by $K_A$ the upper
bound of the ratios of any entry to any diagonal entry in a Tier 1 column. Namely, 
\[
  K_A := \max_{\substack{i,j\le n\\k \le m}} \frac{a_{ij}}{a_{kk}}.
\]

 Given a diagonal entry in a Tier 1 column, the ratio of the other entries of $A$ by this entry is
bounded according to the following schematic in 
\Cref{fig:ratios_folding}.
\begin{figure}[ht!]
  \centering
  \includegraphics[width=.3\textwidth]{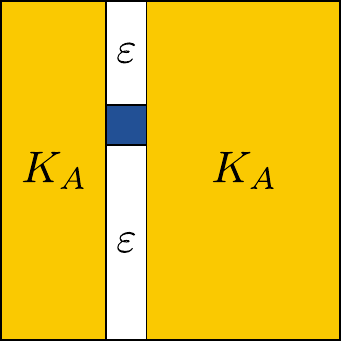}
  \caption{The blue cell represents a given diagonal entry in a Tier 1 column.
    Then each value in the other cells shows an upper bound of the ratio of the entry in
  that cell to the given diagonal entry.}
  \label{fig:ratios_folding}
\end{figure}

\begin{lemma}
  \label{lem:matrixproduct_folding}
  Let $n \ge 1$, and fix $1 \le m \le n$. Suppose $A$ is an $n \times n$, $(m,\eps_A)$-matrix, and let $K_A$ be as defined
  above. If $A'$ is an $n \times n$, $(m, \eps_{A'})$-matrix, then there exists $\eps_{AA'}>0$ such that $AA'$ is an
  $(m,\eps_{AA'})$-matrix, with
  \[
    \eps_{AA'} = \eps_A + n K_A \eps_{A'}.
  \]
\end{lemma}

\begin{proof}
  Let $A$, $A'$, $K_A$, $\epsilon_A, \epsilon_{A'}$, and $m$ be as given in the statement. Then, $AA'$ is a nonnegative matrix and its diagonal entries are positive. Indeed, $(AA')_{ii}= \sum_{j=1}^n a_{ij}a'_{ji} \ge a_{ii}a'_{ii}>0$.  Thus, there exists some $\eps_{AA'}>0$ such that $AA'$ is an $(m, \eps_{AA'})$-matrix. 
  To bound $\eps_{AA'}$, note for $j \le m$ and $i \neq j$:
  \begin{align*}
    \frac{(AA')_{ij}}{(AA')_{jj}} \le \frac{\sum_{k}a_{ik}a'_{kj}}{a_{jj}a_{jj}'}
    &= \sum_{k=1}^n \frac{a_{ik}}{a_{jj}} \frac{a_{kj}'}{a_{jj}'}.
  \end{align*}
  We now consider bounds on the terms in this summand depending
  on the value of $k$.
  \begin{itemize}
  \item $\mathbf{(k=j)}$ Here the summand is equal to $\frac{a_{ij}}{a_{jj}}$, which is bounded above by $\eps_A$.
  \item $\mathbf{(k \neq j)}$ We have 
    $\frac{a_{ik}}{a_{jj}} \le K_A$ and $\frac{a_{kj}'}{a_{jj}'} < \eps_{A'}$, so
    the summand is bounded above by $K_A \eps_{A'}$.
  \end{itemize}
  All in all, 
  \[
    \frac{(AA')_{ij}}{(AA')_{jj}} \le \sum_{k} \frac{a_{ik}}{a_{jj}}
    \frac{a_{kj}'}{a_{jj}'} < \eps_A + (n-1)K_A\eps_{A'}\leq \eps_A + nK_A \eps_{A'},
  \]
  so we may take $\eps_{AA'} = \eps_A + nK_A\eps_{A'}$, concluding the proof.
\end{proof}

\subsection{The game: Constructing matrices with increasingly dominant
  diagonal entries}\label{ssec:thegame_folding}

In this section, we introduce a game between Alice and Bob which will involve
constructing $n\times n$, $(m, \epsilon)$-matrices with increasingly dominant
Tier 1 diagonal entries. Similar constructions were used classically,
see e.g. \cite{yoccoz2010interval}. The purpose of this game, which will henceforth be called the \emph{Alice-Bob game}, is to construct a folding sequence whose transition matrices will satisfy the conditions in \Cref{thm:klp_then_ergodic_folding}. 

Fix $n\geq 2$ and $1\leq m\le n$. The rules of the game are as follows.

First, Alice chooses any $\eps_0 > 0$, and Bob chooses an $n \times n$,
$(m,\eps_0)$-matrix $A_0$. Next, Alice chooses $\eps_1>0$, and
Bob chooses an $n \times n$, $(m,\eps_1)$-matrix $A_1$. Continuing
inductively, they construct a sequence of $n \times n$ matrices
$\{A_j\}_{j=0}^\infty$. For $r, s\in \mathbb{Z}_{\geq 0}$ with $r<s$, put
$A_{r,s} = A_rA_{r+1}\cdots A_s$. Also, say $A_{r,r}:= A_r$.
Alice wins the game if there exist $\{\ov{\eps}_r\}_{r \ge 0}$ such that 
\begin{winninglist}
\begin{enumerate}
\item For every $r \ge 0$ and
  for each $s>r$, $A_{r,s}$ is an $(m,\ov{\eps}_r)$-matrix, and
\item $\ov{\eps}_r \to 0$ as $r \to \infty$.
\end{enumerate}
\caption{Conditions on $\{\ov{\eps}_r\}_{r \ge 0}$ for Alice to win the Alice-Bob game.}
\label{list:Alice_win_folding}
\end{winninglist}
Otherwise, Bob wins the game.

\begin{theorem}\label{thm:TheGame_folding}
  Alice has a winning strategy.
\end{theorem}

\begin{proof}
    We will prove that Alice can win the game with $\ov\eps_r = 2^{-r}$ for each
    $r\geq 0$. For ease of exposition, we will denote $\eps_{A_{r, s}}$ by
    $\eps_{r, s}$ and $K_{A_{r,s}}$ by $K_{r, s}$. Note then $\eps_r =
    \eps_{r,r}$ and $K_{r,r}=K_{A_r}$. To that end, let Alice choose $\eps_0 =
    2^{-1}$, and for each subsequent $s > 0$, choose $\eps_s$ so
    that \[
      nP_s\eps_s < 2^{-(s+1)},
    \] where $P_s = \max\{K_{r,p} \mid 0 \le r \le p < s\}$. Note, in particular, that for each $\eps_s$, we
    have that $\eps_s \leq 2^{-(s+1)}$.

    We claim that for Alice's choices of $\{\eps_p\}_{p=0}^s$ as above, $A_{r,
      s}$ is an $(m, 2^{-r})$-matrix. To see this, note that by
    \Cref{lem:matrixproduct_folding}, we have for each $r \le p <s$:
    \[
      \eps_{r,p+1} = \eps_{r,p} + nK_{r,p}\eps_{p+1} < \eps_{r,p} + 2^{-(p+2)}.
    \]
    Hence, by telescoping, we see that for each $r \ge 0$ and $s>r$:
    \begin{align*}
        \eps_{r, s} &< \eps_{r,r} + 2^{-(r+2)} + \ldots + 2^{-s} + 2^{-(s+1)} \\
        &\le 2^{-(r+1)} +  2^{-(r+2)} + \ldots + 2^{-s} + 2^{-(s+1)} \\
        &< 2^{-r}.
    \end{align*}
    Thus, we have shown that for every $r\geq 0$ and for each $s > r$, $A_{r, s}$ is an $(m, \ov\eps_r)$ matrix, with $\ov\eps_r = 2^{-r}$, since $2^{-r} \to 0$ as $r\to \infty$, we have that Alice has a winning strategy; our choice of $\{\ov{\eps}_r\}_{r \ge 0}$ satisfies the conditions in \Cref{list:Alice_win_folding}.
\end{proof}

\subsection{Constructing ergodic length measures from diagonal-dominant matrices}

In this section, we first prove a sufficient condition for when the limiting $\mathbb{R}$-tree of a folding sequence admits at least $m$ linearly independent ergodic length measures.

\begin{theorem} \label{thm:klp_then_ergodic_folding}
Let $(G_s)_{s \ge 0}$ be a reduced folding sequence with limiting tree $T$, and let $\calM =
(M^{r,s})_{0 \le r \le s}$ be the direct system of $n \times n$ transition matrices associated with the folding sequence $(G_s)_{s \ge 0}$. Let $1
\le m \le n$.
For each $r \geq 0$, suppose there exists $\bar\epsilon_r > 0$ such that for every $s > r$, $(M^{r,s})^T$ is an $(m, \bar\epsilon_r)$-matrix and $\lim_{r\to \infty}\bar\epsilon_r = 0.$
Then the limiting
tree $T$ admits at least $m$ linearly independent ergodic length measures.
\end{theorem}

Note that whenever Alice wins the Alice-Bob game described above (which \Cref{thm:TheGame_folding} guarantees will happen), the resultant folding sequence has transition matrices which satisfy the conditions of the above theorem. As such, we will frequently use the Alice-Bob game to construct a folding sequence whose transition matrices satisfy the conditions of the above theorem.

\begin{proof}
    Firstly, by \Cref{thm:NPR-foldingsequence} it suffices to show
    that the reduced folding sequence $(G_s)_{s \ge 0}$ admits at least $m$
    linearly independent measures. Indeed, $\calC((G_s)_{s \ge 0})$ is
    defined as the simplicial cone of non-negative vectors in the $n$-dimensional vector space whose vectors are sequences $(\vec{v}_s)_{s \ge 0}$ with $\vec{v}_s \in \R^n$  for every $s \ge 0$ satisfying the compatibility condition
    \[
        \vec{v}_{s} = M_{s+1}^T\vec{v}_{s+1},
    \]
    for every $s \ge 0$, where $M_{s+1}=M^{s,s+1}$. The
    projectivization is a simplex, and if we show $\dim \calC((G_s)_{s
      \ge 0})\geq m$ it will follow that the simplex has at least $m$
    vertices representing linearly independent ergodic measures.
    Denoting by $\calC(G_s)$ the space of length measures on $G_s$, the space $\mathcal{C}((G_s)_{s \ge 0})$ is isomorphic to the following space,
    \[
        \mathcal{C}_0 := \lim_{s \to \infty} \bigcap_{k=0}^s (M^{0,k})^T(\calC(G_k)).
    \]
    On the other hand, we know  $\calC(G_s) \cong \R_{\geq 0}^n$ for every $s \ge 0$, so it follows that
    \[
        \mathcal{C}_0 = \lim_{s \to \infty} \bigcap_{k = 0}^s \left(\text{the nonnegative cone of columns of $(M^{0,k})^T$}\right).
    \]

    Note that we may assume that $\bar\epsilon_0$ is as small as we
    like by truncating the sequence and reindexing. We choose this so
    that any $n\times m$ matrix of the form
    
    \[
    P:=
    \begin{bmatrix}
        1 & * & * & \ldots & * \\
        * & 1 & * & \ldots & * \\
        * & * & 1 & \ldots & * \\
        \vdots & \vdots & \vdots & \ddots & \vdots \\
        * & * & * & \cdots & 1 \\
        * & * & * & \cdots & * \\
        * & * & * & \cdots & * \\
    \end{bmatrix},
    \]
    where $*$ denotes an entry in the interval $[0,\bar\epsilon_0]$, will have $m$
    linearly independent columns.

    Now, we claim that $\mathcal{C}_0$ contains at least $m$ linearly independent
    vectors, by showing that the first $m$ columns of the matrices in the
    sequence $\{(M^{0,k})^T\}_{k=0}^{\infty}$ \emph{projectively}
    converge to columns as above (with scaling separate in each column).

    For each $1 \le i \le m$ and $k \ge 0$, consider the $i$-th column of
    $(M^{0,k})^T$ and divide it by the diagonal entry $(M^{0,k})^T_{ii}$. Denote
    by $C_i^k$ the resulting column vector. We will show that
    \[
        \left[ C_1^k\ C_2^k\ \cdots \ C_m^k \right] \quad \longrightarrow \quad P,
    \] 
    as $k \to \infty$, after passing to some subsequence.
    By construction, the $i$-th entry of $C_i^k$ is 1, for all $k \ge 0$. The other entries of $C_i^k$ are bounded above by $\ov{\eps}_0$. Hence, for each $i$ we can take a further subsequence of the folding sequence to ensure that $C_i^k$ will converge as $k\to \infty$. By passing to further subsequences for each $i=1,\ldots,m$, we may take a subsequence where each $C_1^k,\ldots,C_m^k$ converge and the limit is in the form of $P$. 
    \end{proof}

\subsection{End game: Our folding sequence}

\label{ssec:lowerboundFinale_folding}
In this section, for each $n \ge 4$, we construct a folding sequence of graphs of rank $n$ whose transition matrices satisfy the conditions of \Cref{thm:klp_then_ergodic_folding}.

\begin{theorem} \label{thm:lowerboundFinale_folding}
For each $n \ge 4$, there exists a reduced folding sequence $\calG=(G_k)_{k \ge 0}$ of
graphs of rank $n$ whose limiting tree admits at least $2n-6$ linearly independent ergodic length measures.
\end{theorem}

\begin{proof}
    Let $n \ge 4$. Consider the graph $\G'$ of rank $n-1$, defined and labeled
    as in the left graph of \Cref{fig:lowerbound}. The edges
    $b_0,\ldots,b_{n-3}$ are single-edged loops, and $a_0,\ldots, a_{n-3}$ are
    the edges connecting consecutive $b_i$'s. Now consider a loop $c$, called a
    \emph{rein}. Then for each $e \in E\Gamma'$, we denote by $\G_e$ the graph obtained from $\G'$ by adjoining $c$ to be incident to the edge $e$ such that $(c,e)$ and $(\ov{c},e)$ form the only illegal turns in $\G_e$. See the right graph of \Cref{fig:lowerbound} for an example of $\G_{a_0}$.

\begin{figure}[ht!]
    \centering
    \includegraphics{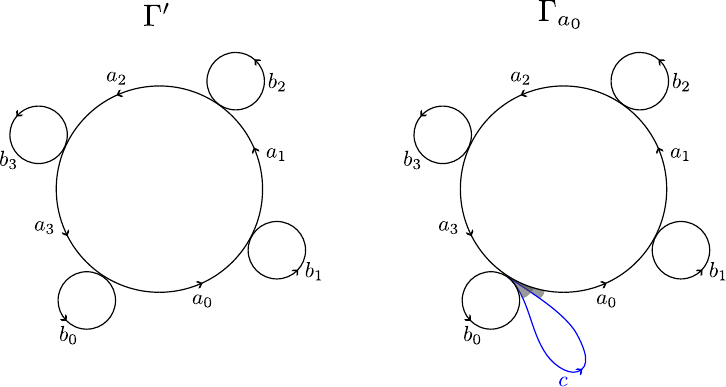}
    \caption{The graph $\G'$ on the left is a graph with $2n-4$ edges with $n-2$
      vertices, where $n=6$. The $a_1,\ldots,a_{n-3}$ edges together form a single loop in the middle, and the $b_0,\ldots,b_{n-3}$ edges are loops incident to each vertex of $\G'$. The graph $\G_{a_0}$ on the right is obtained from $\G'$ by attaching a loop $c$ that shares the same vertex as the initial vertex of $a_0$. It has a train track structure with a single gate $\{c,\ov{c},a_0\}$, which attributes to two illegal turns, shown in gray.}
    \label{fig:lowerbound}
\end{figure}

Now, we consider the following sets of maps between different $\G_{e}$'s. First, consider the $2n-6$ \emph{rein movers} $\cR_{a_0},\ldots,\cR_{a_{n-4}}$ and $\cR_{b_1},\ldots,\cR_{b_{n-3}}$
defined as:
\begin{align*}
    &\cR_{a_i}: \G_{a_i} \to \G_{b_{i+1}},
    &&\cR_{a_i} =
    \begin{cases}
    c \mapsto a_ic\ov{a_i},\\
    \text{Id} & \text{otherwise},
    \end{cases}\\
    &\cR_{b_i}: \G_{b_i} \to \G_{a_i},
    &&\cR_{b_i} =
    \begin{cases}
    c \mapsto b_i^2c\ov{b_i},\\
    \text{Id} & \text{otherwise}.
    \end{cases}
\end{align*}
Notice we have $b_i^2$ in the image of $c$ for $\cR_{b_i}$, to prevent $c$ from
being an invariant conjugacy class of the composition $\calF$ (to be defined below).

Secondly, we fix positive integers $\alpha_1,\ldots,\alpha_{n-3}$ and
$\beta_{1},\ldots,\beta_{n-3}$, and consider the $2n-6$ \emph{loopers} $\cL_{a_1},\ldots,\cL_{a_{n-3}}$ and $\cL_{b_1},\ldots,\cL_{b_{n-3}}$ defined as:
\begin{align*}
    &\cL_{a_i}: \G_{a_i} \to \G_{a_i},
    &&\cL_{a_i} =
    \begin{cases}
        a_i \mapsto c(\ov{a_{i-1}}b_{i-1}a_{i-1}b_i)^{\alpha_{i}}a_i, \\
    \text{Id} & \text{otherwise},
    \end{cases}\\
    &\cL_{b_i}: \G_{b_i} \to \G_{b_i},
    &&\cL_{b_i} =
    \begin{cases}
        b_i \mapsto c\overline{a_{i-1}}b_{i-1}^{\beta_i}a_{i-1}b_i, \\
        \text{Id} & \text{otherwise}.
    \end{cases}
\end{align*}

Finally, we consider the \emph{rotator} $\rho: \G_{a_{n-3}} \to \G_{a_0}$ defined as 
\[
    \rho =
    \begin{cases}
        a_i \mapsto a_{i+1} & \text{for $i=0,\ldots,n-3$,}\\
        b_i \mapsto b_{i+1} & \text{for $i=0,\ldots,n-3$,}\\
        c \mapsto c,
    \end{cases} 
  \]
  where for convenience of notation, we use indices modulo $n-2$, such as $a_{n-2} \equiv a_0$ and $b_{n-2} \equiv b_0$.

Then, our desired train track map $\calF$ is the composition of these maps in the following order:
    \begin{align*}
        \calF&= \rho \circ (\cL_{a_{n-3}} \circ \cR_{b_{n-3}})\circ (\cL_{b_{n-3}} \circ \cR_{a_{n-4}}) \circ \cdots \circ (\cL_{a_1} \circ \cR_{b_1}) \circ (\cL_{b_1} \circ \cR_{a_0})\\
        &= \rho \circ \prod_{i=1}^{n-3}[\cL_{a_{i}} \circ \cR_{b_i} \circ \cL_{b_i} \circ \cR_{a_{i-1}}],
    \end{align*}
    which is a map from $\G_{a_0}$ to itself.

In fact, we can compute the composition as follows. Each of $a_1,
\ldots, a_{n-3}$ or of $b_1, \ldots, b_{n-3}$ is mapped via the associated looper
$\cL_{a_i}$ or $\cL_{b_j}$ only once in the product  $\prod_{i=1}^{n-3}[\cL_{a_{i}} \circ \cR_{b_i} \circ \cL_{b_i} \circ \cR_{a_{i-1}}]$, and only $c$ is altered multiple times via the rein movers. Note that $a_0$ and $b_0$ remain the same until the rotator comes in. For brevity, we write:
\begin{align*}
  \cP_1&:= a_0b_1a_1b_2a_2\cdots b_{n-4}a_{n-4}b_{n-3},\\
  \cP_2&:= a_0b_1^2a_1b_2^2a_2\cdots b_{n-4}^2a_{n-4}b_{n-3}^2,\\
  \cQ_1&:= a_1b_2a_2b_3a_3\cdots b_{n-3}a_{n-3}b_0, \\
  \cQ_2&:= a_1b_2^2a_2b_3^2a_3\cdots b_{n-3}^2a_{n-3}b_0^2, \\
\end{align*}
and for a letter $e$ appearing in $\cP_1$, write $\cP_1[e]$ as the word obtained
from truncating the initial part of $\cP_1$ up to (but not including) $e$. For simplicity,
for $e$ not appearing in $\cP_1$, such as $b_0$ or $a_{n-3}$, define $\cP_1[e]$ as
the empty word. For example, $\cP_1[b_2] = b_2a_2b_3a_3\cdots
b_{n-4}a_{n-4}b_{n-3}$. Define $\cQ_1[e]$ similarly for a letter $e$ appearing in
$\cQ_1$, and let $\cQ_1[e]$ be the empty word if $e$ is not in $\cQ_1$. Define
similarly for $\cP_2[e]$ and $\cQ_2[e]$, for example
$\cQ_2[b_2]=b_2^2a_2b_3^2a_3\cdots b_{n-3}^2a_{n-3}b_0^2$.

Then we can describe the product $\prod_{i=1}^{n-3}[\cL_{a_{i}} \circ \cR_{b_i} \circ \cL_{b_i} \circ \cR_{a_{i-1}}]$ explicitly as follows:
\begin{align*}
    c &\longmapsto \cP_2 c \ov{\cP_1} = a_0b_1^2a_1b_2^2\cdots b_{n-4}^2a_{n-4}b_{n-3}^2c \ov{b_{n-3}a_{n-4}b_{n-4}\cdots b_2a_1b_1a_0},\\
    a_0 &\longmapsto a_0,\\
    b_0 &\longmapsto b_0,\\
    a_k &\longmapsto \cP_2[a_k]c\ov{\cP_1[a_k]}\left(\ov{a_{k-1}}b_{k-1}a_{k-1}b_k\right)^{\alpha_k}a_k,\\
    &\phantom{\!\mapsto}=
    \begin{cases}
        c\left(\ov{a_{n-4}}b_{n-4}a_{n-4}b_{n-3}\right)^{\alpha_{n-3}}a_{n-3} & \text{if $k=n-3$,}\\
        a_k\cdots b_{n-4}^2a_{n-4}b_{n-3}^2c\ov{b_{n-3}a_{n-4}b_{n-4}\cdots a_k}\left(\ov{a_{k-1}}b_{k-1}a_{k-1}b_k\right)^{\alpha_k}a_k & \text{if $k=1,\ldots,n-4,$}
    \end{cases}\\
    b_k &\longmapsto \cP_2[b_k]c\ov{\cP_1[b_k]}\left(\ov{a_{k-1}}b_{k-1}^{\beta_{k}}a_{k-1}b_k\right),\\
    &\phantom{\!\mapsto}=
    \begin{cases}
        b_{n-3}^2c\ov{b_{n-3}}\left(\ov{a_{n-4}}b_{n-4}^{\beta_{n-3}}a_{n-4}b_{n-3}\right) & \text{if $k=n-3$},\\
        b_k^2a_k\cdots b_{n-4}^2a_{n-4}b_{n-3}^2c\ov{b_{n-3}a_{n-4}b_{n-4}\cdots a_kb_k}\left(\ov{a_{k-1}}b_{k-1}^{\beta_{k}}a_{k-1}b_k\right) & \text{if $k=1,\ldots,n-4.$}
    \end{cases}
\end{align*}
Therefore, the image under $\calF$ of each edge can be obtained by applying the
rotator map $\rho$ to the image above, which has the result of increasing the subscripts of the $a_i$'s
and $b_i$'s by 1 where $a_{n-2}\equiv a_0$ and $b_{n-2}\equiv b_0$ while
fixing $c$. Also note for $j=1,2$, we have $\rho(\cP_j) =\cQ_j$, $\rho(\cP_j[a_k]) = \cQ_j[a_{k+1}],$ and $\rho(\cP_j[b_k]) = \cQ_j[b_{k+1}].$ Indeed, the $\calF$-image of each edge is:
\begin{align*}
    c &\longmapsto \cQ_2 c \ov{\cQ_1}= a_1b_2^2a_2b_3^2\cdots b_{n-3}^2a_{n-3}b_0^2c\ov{b_0a_{n-3}b_{n-3}\cdots b_3a_2b_2a_1},\\
    a_0 &\longmapsto a_1,\\
    b_0 &\longmapsto b_1,\\
    a_k &\longmapsto \cQ_2[a_{k+1}]c\ov{\cQ_1[a_{k+1}]}\left(\ov{a_{k}}b_{k}a_{k}b_{k+1}\right)^{\alpha_k}a_{k+1},\\
    &\phantom{\!\mapsto}=
    \begin{cases}
        c\left(\ov{a_{n-3}}b_{n-3}a_{n-3}b_{0}\right)^{\alpha_{n-3}}a_{0} & \text{if $k=n-3$,}\\
        a_{k+1}\cdots b_{n-3}^2a_{n-3}b_{0}^2c\ov{b_{0}a_{n-3}b_{n-3}\cdots a_{k+1}}\left(\ov{a_{k}}b_{k}a_{k}b_{k+1}\right)^{\alpha_k}a_{k+1} & \text{if $k=1,\ldots,n-4,$}
    \end{cases}\\
    b_k &\longmapsto \cQ_2[b_{k+1}]c\ov{\cQ_1[b_{k+1}]}\left(\ov{a_{k}}b_{k}^{\beta_{k}}a_{k}b_{k+1}\right),\\
    &\phantom{\!\mapsto}=
    \begin{cases}
        b_{0}^2c\ov{b_{0}}\left(\ov{a_{n-3}}b_{n-3}^{\beta_{n-3}}a_{n-3}b_{0}\right) & \text{if $k=n-3$},\\
        b_{k+1}^2a_{k+1}\cdots b_{n-3}^2a_{n-3}b_{0}^2c\ov{b_{0}a_{n-3}b_{n-3}\cdots a_{k+1}b_{k+1}}\left(\ov{a_{k}}b_{k}^{\beta_{k}}a_{k}b_{k+1}\right) & \text{if $k=1,\ldots,n-4.$}
    \end{cases}
\end{align*}

Hence, with $(e^1,\ldots,e^{2n-3})=(b_1,\ldots,b_{n-3},a_1,\ldots,a_{n-3},a_0,b_0,c)$, the transition matrix for $\calF$ is:

\begin{center}
  \scalebox{0.85}{$
    M=\begin{pmatrix}
    \setlength\arraycolsep{5pt}
    \begin{matrix}
            \beta_1    &                &                        &            &                &   \\
               4       &  \beta_{2}     &                        &            &                &   \\
               3       &      4         &   \beta_3              &            &                &   \\
            \vdots     &   \vdots       &  \vdots                & \ddots     &                &   \\
               3       &      3         &  3                     & \cdots     & \beta_{n-4}    &   \\
               \ph{.}3\ph{.}       &      \ph{.}3\ph{.}         &  \ph{.}3\ph{.}                     & \cdots     & 4              & \beta_{n-3}
    \end{matrix}
    &
    \rvline 
    &
    \setlength\arraycolsep{8pt}
    \begin{matrix}
          \alpha_{1}   &                      &                 &                &                &      \\
          \alpha_{1}   &   \alpha_{2}         &                 &                &                &      \\
              3        &   \alpha_{2}         &   \alpha_3      &                &                &      \\
           \vdots      &   \vdots             &   \vdots        & \ddots         &                &      \\
              3        &       3              &      3          & \cdots         & \alpha_{n-4}   &        \\
              \ph{1}3\ph{1}        &       \ph{1}3\ph{1}              &      \ph{1}3\ph{1}          & \cdots         & \ph{.}\alpha_{n-4}\ph{.}   &     \alpha_{n-3}
    \end{matrix}
    &
    \rvline
    &
    \begin{matrix}
      0 & 1      & 0 \\
      0 & 0      & 3 \\
      0 & 0      & 3 \\
  \vdots & \vdots & \vdots \\
      0 & 0      & 3 \\
      0 & 0      & 3
    \end{matrix}\\
    \hline 
    \setlength\arraycolsep{7pt}
    \begin{matrix}
        2             &                    &           &          &                  &  \\
        2             & 2                  &           &          &                  &  \\
        2             & 2                  &  2        &          &                  &  \\
        \ph{!}\vdots\ph{!}        & \vdots             &  \ph{!}\vdots\ph{!}   & \ph{!}\ddots\ph{!}   & \ph{\beta_{n-}} & \ph{\beta_{n}} \\
        2             & 2                  &  2        & \cdots   &    2             &  \\
        2             & 2                  &  2        & \cdots   &    2             & 2
    \end{matrix}
    &
    \rvline 
    &
    \setlength\arraycolsep{6.5pt}
    \begin{matrix}
        2\alpha_{1}    &                &                   &            &               &        \\
            3          &  2\alpha_{2}   &                   &            &               &        \\
            2          &      3         & 2\alpha_{3}       &            &               &        \\
            \vdots     &  \vdots        &   \vdots          & \ddots     &               &        \\
            2          &      2         &      2            & \cdots     & 2\alpha_{n-4} &        \\
            \ph{1.}2\ph{1.}          &      \ph{1.}2\ph{1.}         &      \ph{1\!.}2\ph{1\!.}            & \cdots     &      3        & \ph{1}2\alpha_{n-3}\ph{1}
    \end{matrix}
    &
    \rvline
    &
    \begin{matrix}
      1 & 0      & 2 \\
      0 & 0      & 2 \\
      0 & 0      & 2 \\
  \vdots & \vdots & \vdots \\
      0 & 0      & 2 \\
      0 & 0      & 2
    \end{matrix}\\
    \hline
    \setlength\arraycolsep{9pt}
    \begin{matrix}
       0 & 0 & 0 & \cdots & \ph{|}0\ph{|} & \ph{|}0\ph{|}\\
       3 & 3 & 3 & \cdots & \ph{|}3\ph{|} & \ph{|}4\ph{|}\\
       1 & 1 & 1 & \cdots & \ph{|}1\ph{|} & \ph{|}1\ph{|}
    \end{matrix}
    &
    \rvline
    &
    \setlength\arraycolsep{8.3pt}
    \begin{matrix}
        \ph{!}0\ph{!} & \ph{!}0\ph{!} & \ph{!}0\ph{!} & \cdots & \ph{\alpha}1\ph{\alpha} & \ph{\alpha}1\ph{\alpha}           \\
        \ph{1}3\ph{1} & \ph{1}3\ph{1} & \ph{1}3\ph{1} & \cdots & \ph{\alpha}1\ph{\alpha} & \ph{1}\alpha_{n-3}\ph{1}\\
        \ph{!}1\ph{!} & \ph{!}1\ph{!} & \ph{!}1\ph{!} & \cdots & 1 & 1            
    \end{matrix}
    &
    \rvline
    &
    \begin{matrix}
        0 & 0 & 0\\
        0 & 0 & 3\\
        0 & 0 & 1
    \end{matrix}
    \end{pmatrix}.
    $}
  \end{center}

Hence, the transpose of $M$ is of the following form:

\begin{center}
  \scalebox{0.8}{$
    M^T=\begin{pmatrix}
    \setlength\arraycolsep{5pt}
    \begin{matrix}
       \ph{\ 1}\beta_{1}\ph{\ } &  \ph{\ 1}4\ph{\ 1}             &   \ph{\ }3\ph{\ }                    & \cdots     &3              & 3 \\
                       &  \beta_{2}     &   4                    & \cdots     & 3              & 3 \\
                       &                &   \beta_{3}            & \cdots     & 3              & 3 \\
                       &                &                        & \ddots     & \vdots         & \vdots \\
                       &                &                        &            & \beta_{n-4}    & 4 \\
                       &                &                        &            &                & \beta_{n-3}
    \end{matrix}
    &
    \rvline 
    &
    \setlength\arraycolsep{8pt}
    \begin{matrix}
        \ph{2}2\ph{2}             &  \ph{2}2\ph{2}                 &   \ph{1}2\ph{1}       &   \cdots        &   \ph{11}2\ph{11} & \ph{11}2\ph{11} \\
                                  &  2                 &   2       &   \cdots   &   2  & 2 \\
                                  &                    &   2       &   \cdots   &    2      & 2 \\
                                  &                    &           &   \ddots   &   \vdots  & \vdots \\
                                  &                    &           &            &    2      & 2 \\
                                  &                    &           &            &           & 2
    \end{matrix}
    &
    \rvline
    &
    \begin{matrix}
      0 & \ph{\alpha\ \!}3\ph{\alpha\ \!}      & 1 \\
      0 & 3      & 1 \\
      0 & 3      & 1 \\
  \vdots & \vdots & \vdots \\
      0 & 3      & 1 \\
      0 & 4      & 1
    \end{matrix}\\
    \hline 
    \setlength\arraycolsep{7pt}
    \begin{matrix}
          \alpha_{1}   &   \alpha_{2}         &   3             & \cdots         & 3         &     3 \\
                       &   \alpha_{2}         &   \alpha_{3}    & \cdots         & 3         &     3 \\
                       &                      &   \alpha_{3}    & \cdots         & 3         &     3 \\
                       &                      &                 & \ddots         & \vdots    &     \vdots \\
                       &                      &                 &                & \ph{1}\alpha_2\ph{1}  &  \alpha_1 \\
                       &                      &                 &                &           &     \alpha_1
    \end{matrix}
    &
    \rvline 
    &
    \setlength\arraycolsep{6.5pt}
    \begin{matrix}
        2\alpha_{1}    &  3             &  2                 & 2          & \cdots    & 2 \\
                       &  2\alpha_{2}   &  3                 & 2          & \cdots    & 2 \\
                       &                &  2\alpha_3         & 3          & \cdots    & 2 \\
                       &                &  \ddots            & \ddots     & \ddots    & \vdots \\
                       &                &  \ph{\alpha_{n-5}} &            & 2\alpha_{n-4} & 3 \\
                       &                &                    &            &           & 2\alpha_{n-3}
    \end{matrix}
    &
    \rvline
    &
    \begin{matrix}
      0 & 3      & 1 \\
      0 & 3      & 1 \\
      0 & 3      & 1 \\
  \vdots & \vdots & \vdots \\
      0 & 3      & 1 \\
      1 & \alpha_{n-3}      & 1
    \end{matrix}\\
    \hline
    \setlength\arraycolsep{9pt}
    \begin{matrix}
       \ph{\ }0\ph{\ }            & 0 & 0 & \cdots & 0 & \ph{\ }0\ph{\ }\\
        1            & 0 & 0 & \cdots & 0 & 0\\
        0            & 3 & 3 & \cdots & 3 & 3
    \end{matrix}
    &
    \rvline
    &
    \setlength\arraycolsep{8.3pt}
    \begin{matrix}
       1 & 0 & 0 & \cdots & 0 & 0\\
       0 & 0 & 0 & \cdots & 0 & 0\\
       \ph{1\!}2\ph{1\!} & \ph{1}2\ph{1} & \ph{1}2\ph{1} & \cdots & \ph{\alpha\ \!}2\ph{\alpha\ \!} & \ph{\alpha\ \!}2\ph{\alpha\ \!}
    \end{matrix}
    &
    \rvline
    &
    \begin{matrix}
        0 & 0 & 0\\
        0 & 0 & 0\\
        0 & \ph{\alpha\ \!}3\ph{\alpha\ \!} & 1
    \end{matrix}
    \end{pmatrix}.
    $}
  \end{center}

Finally, we form a folding sequence $\calG = (G_k)_{k \ge 0}=
\{\phi_{k}:G_{k} \to G_{k+1}\}_{k \ge 0}$ of graphs of rank $n$ with
$G_k = \G_{a_0}$ and for $k \ge 0$, let $\phi_{k} = \calF^{(k)}$,
which is the map $\calF$ with replacing
$(\beta_1,\ldots,\beta_{n-3},\alpha_1,\ldots,\alpha_{n-3}) =
\left(\beta_1^{(k)},\ldots,\beta_{n-3}^{(k)},\alpha_1^{(k)},\ldots,\alpha_{n-3}^{(k)}\right)$,
where the $\alpha_i^{(k)}$ and $\beta_j^{(k)}$ are chosen following Bob's role in
\Cref{thm:TheGame_folding} with $m=2n-6$. 
We note that by construction, in particular by the rotator $\rho$, the folding sequence does not have a proper invariant subgraph, therefore it is reduced. 
Now, note that for any $\epsilon_k>0$ that Alice chooses, Bob can
choose large enough
$\alpha_i^{(k)}$ and $\beta_j^{(k)}$ so that the transposed matrix is
a $(2n-6,\epsilon_k)$-matrix. 
So, \Cref{thm:TheGame_folding} guarantees the
existence of $\{\ov{\eps_r}\}_{r=0}^\infty$ for the assumptions in
\Cref{thm:klp_then_ergodic_folding}.
Therefore, we conclude that the limiting tree admits at least $m=2n-6$ linearly independent
ergodic length measures. \end{proof}

\section{Lower Bound $2n-6$: Unfolding Sequence}\label{sec:lowerbound_unfolding}
In this section, we estimate the lower bound for the number of linearly independent ergodic measures supported on the legal lamination of a given unfolding sequence of graphs of rank $n$.
\Cref{thm:NPR-unfoldingsequence} reduces our problem to finding the number of independent ergodic measures supported on the unfolding sequence.

\subsection{Diagonal-dominant matrices}
\label{ssec:ddmatrix}
Let $n$ be a positive integer. Similar to \Cref{ssec:ddmatrix_folding}, here we consider a collection of diagonal-dominant $n \times n$ matrices
whose product remains diagonal-dominant. This time however, the
definition is a bit more involved. The diagonal entries will be
ordered from largest to smallest and we will allow entries below the
diagonal to be comparable to the diagonal entry in the same column.

\begin{definition} \label{def:klp}
  Let $p > 1$ and $\eps, \delta>0$, and let $m$ be a positive integer such that $m \le n$. An $n \times n$ nonnegative matrix $A =
  (a_{ij})$ with positive diagonal entries is called an
  \textbf{$(m,p,\eps,\delta)$-matrix} if the following holds:
  \begin{enumerate}[label=(\alph*)]
  \item \textbf{(Diagonal entries)} for $i < j \le m$, we have
    \[
      \frac{a_{jj}}{a_{ii}} < \delta.
    \]
  \item \textbf{(Above diagonal entries)} for $i < j \le m$, we have
    \[
      \frac{a_{ij}}{a_{jj}} < \eps.
    \]
  \item \textbf{(Below diagonal entries)} for $i < j \le m$, we have
    \[
      \frac{a_{ji}}{a_{ii}} < p.
    \]
  \item \textbf{(Tier 2 entries)} for $k \le m < j$ and any $i$, we have
    \[
      \frac{a_{ij}}{a_{kk}} < p\delta.
    \]
  \end{enumerate}
\end{definition}
Recall from \Cref{ssec:ddmatrix_folding} that we refer to the first $m$ columns of the
matrix $A$ as \textbf{Tier 1} columns, and the other ones as \textbf{Tier 2}
columns. 
We will be primarily considering $(m, p, \eps, \delta)$-matrices with
small $\delta$ and $\eps$ and with $\delta$ much smaller than $\eps$. We will
refer to these matrices as \emph{diagonal-dominant} matrices. 
 Also for an $n \times n$ matrix, denote by $K_A$ the upper
bound of the ratios of any entry to any diagonal entry in a Tier 1 column. Namely:
\[
  K_A := \max_{\substack{i,j\le n\\k\le m}} \frac{a_{ij}}{a_{kk}}.
\]

Given a diagonal entry in a Tier 1 column, the ratio of the other entries of $A$ by this entry is
bounded according to the following schematic in \Cref{fig:ratios}.
\begin{figure}[ht!]
  \centering
  \includegraphics[width=.3\textwidth]{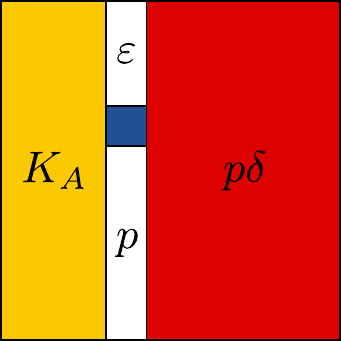}
  \caption{The blue cell represents a given diagonal entry in a Tier 1 column.
    Then each value in the other cells shows an upper bound of the ratio of the entry in
  that cell to the given diagonal entry. Of course, $K_A$ bounds all such
  ratios, but $p$ and $p\delta$ are typically much better bounds.}
  \label{fig:ratios}
\end{figure}

Indeed, the entries above and below the diagonal are controlled by (b) and
(c). An entry to the left is controlled by $K_A$ by definition. For
entries to the right which are in a Tier 1 column, the ratio bounded by $p$ times the
diagonal entry in the same column, which is bounded by $\delta$ times the
original entry by (a). For the entries to the right which are in a Tier 2 column, the ratio is
also bounded by $p\delta$ by (d).  

\begin{lemma}  \label{lem:matrixproduct}
  Let $A$ be an $n \times n$, $(m,p_A,\eps_A,\delta_A)$-matrix with
  $K_A$ defined 
  as above. If $A'$ is an $n \times n$ $(m, p, \epsilon, \delta)$-matrix, then there exist $p_{AA'}>1$, and $\eps_{AA'}, \delta_{AA'}>0$ such that $AA'$ is an
  $(m,p_{AA'}, \eps_{AA'},\delta_{AA'})$-matrix with
  \begin{enumerate}[label=(\arabic*)]
  \item $p_{AA'} = p_A + nK_A\eps + np_A\delta_Ap$,
  \item $\eps_{AA'} = \eps_A + nK_A\eps + np_A\delta_Ap$, and
  \item $\delta_{AA'} = nK_Ap\delta.$
  \end{enumerate}
\end{lemma}

\begin{proof}
  Let $A=(a_{ij})$ and $A'=(a'_{ij})$ be as given in the statement. Then $AA'$
  is a nonnegative matrix with positive diagonal entries. 
  We now check that conditions (a--d) in \Cref{def:klp} are satisfied for the matrix $AA'$, with $p_{AA'}$, $\eps_{AA'}$, and $\delta_{AA'}$ as desired.

  We first establish (a). Let $i < j\leq m$. Note that
  \begin{align}
    \frac{(AA')_{jj}}{(AA')_{ii}} \le \frac{\sum_{k}a_{jk}a'_{kj}}{a_{ii}a_{ii}'}
    &= \sum_{k} \frac{a_{jk}}{a_{ii}} \frac{a_{kj}'}{a_{ii}'}\notag\\
    &< \sum_{k} (K_A) \cdot (p\delta) =nK_Ap\delta,\tag{$\ast$}
  \end{align}
  so we can take $\delta_{AA'} = nK_Ap\delta$.

  We now consider (b). Let $i < j\leq m$. Similar to above, we see that
  \begin{align*}
    \frac{(AA')_{ij}}{(AA')_{jj}} \le \frac{\sum_{k}a_{ik}a'_{kj}}{a_{jj}a_{jj}'}
    &= \sum_{k} \frac{a_{ik}}{a_{jj}} \frac{a_{kj}'}{a_{jj}'}.
  \end{align*}
  This time, we consider the value of each summand depending on the value of $k$.  
  \begin{itemize}
  \item $\mathbf{(k<j)}$ We have 
    $\frac{a_{ik}}{a_{jj}} \le  K_A$ and $\frac{a_{kj}'}{a_{jj}'} < \eps$, so
    each such summand is bounded above by $K_A \eps$.
  \item $\mathbf{(k>j)}$ We have 
    $\frac{a_{ik}}{a_{jj}} \le p_A\delta_A$ and $\frac{a_{kj}'}{a_{jj}'} < p$, so
    each such summand is less than $p_A\delta_Ap$.
  \item $\mathbf{(k=j)}$ Here the summand is equal to $\frac{a_{ij}}{a_{jj}}$,
    which is bounded above by $\eps_A$.
  \end{itemize}
  All in all, we have that 
  \[
    \frac{(AA')_{ij}}{(AA')_{jj}} \le \sum_{k} \frac{a_{ik}}{a_{jj}}
    \frac{a_{kj}'}{a_{jj}'} < \eps_A + nK_A\eps + np_A\delta_Ap,
  \]
  so we take $\eps_{AA'} = \eps_A + nK_A\eps + np_A\delta_Ap$.

  In a similar way, we can prove (c). Let $i< j\leq m$, and then
  \begin{align*}
    \frac{(AA')_{ji}}{(AA')_{ii}} &\le \frac{\sum_{k}a_{jk}a'_{ki}}{a_{ii}a_{ii}'}
    = \sum_{k} \frac{a_{jk}}{a_{ii}} \frac{a_{ki}'}{a_{ii}'} \\
    &= \left(  \sum_{k<i} \frac{a_{jk}}{a_{ii}} \frac{a_{ki}'}{a_{ii}'} \right) + \left( \sum_{k>i} \frac{a_{jk}}{a_{ii}} \frac{a_{ki}'}{a_{ii}'}\right) + \frac{a_{ji}}{a_{ii}}\\
    &< nK_A\eps + np_A\delta_Ap + p_A,
  \end{align*}
  so we take $p_{AA'} = p_A + nK_A\eps + np_A\delta_Ap$.

  Finally, we verify (d). Let $t \le m <j$.
  Then for any $i$,
  \[
    \frac{(AA')_{ij}}{(AA')_{tt}} \le \sum_{k} \frac{a_{ik}}{a_{tt}}
    \frac{a'_{kj}}{a'_{tt}} < nK_Ap\delta = \delta_{AA'}< p_{AA'}\delta_{AA'}
  \]
  since $p_{AA'}>p_A>1$. 
\end{proof}

\subsection{The game: Constructing matrices with increasingly dominant
  diagonal entries}\label{ssec:thegame}

As in \Cref{ssec:thegame_folding}, we consider another game by Alice and Bob. Due to the more involved definition of the diagonal dominant matrix used in this section, they engage in a game with more parameters as well.
Alice and Bob play the following game. Fix $n\ge 2$, $m\le n$, and $p>1$. 
First, Alice chooses $(\eps_0,\delta_0)$. Then Bob chooses an $n \times n$,
$(m,p,\eps_0,\delta_0)$-matrix $A_0$. Then Alice chooses $(\eps_1,\delta_1)$ and
Bob chooses an $n \times n$, $(m,p,\eps_1,\delta_1)$-matrix $A_1$. Continuing
inductively, they construct a sequence of $n \times n$ matrices
$\{A_j\}_{j \ge 0}$. For $r<s$, put $A_{r,s} = A_rA_{r+1}\cdots A_s$. Write $A_{r,r}:=A_r$.
Alice wins the game if there exist triples $\{(\ov{p}_r,
\ov{\eps}_r,\ov{\delta}_r)\}_{r \ge 0}$ such that
\begin{winninglist}
\begin{enumerate}
\item for each $s>r$, $A_{r,s}$ is an $(m,\ov{p}_r,\ov{\eps}_r,\ov{\delta}_r)$-matrix,
\item $\{\ov{p}_r\}_{r \ge 0}$ is uniformly bounded,
\item $\ov{\eps}_r \to 0$ as $r \to \infty$, and
\item $\ov{\delta}_r \to 0$ as $r \to \infty$.
\end{enumerate}
\caption{Conditions on $\{(\ov{p}_r,
\ov{\eps}_r,\ov{\delta}_r)\}_{r \ge 0}$ for Alice to win in the Alice-Bob game.}
\label{list:Alice_win_unfolding}
\end{winninglist}
Otherwise, Bob wins the game.

\begin{theorem}\label{thm:TheGame}
  Alice has a winning strategy. 
\end{theorem}

\begin{proof}
    We will show that Alice can make choices such that conditions
    (1--4) in \Cref{list:Alice_win_unfolding} hold. It will be convenient to simplify notation in this
    proof and denote by $K_{r,s},p_{r,s},\eps_{r,s},\delta_{r,s}$ the corresponding
    constants for the matrix $A_{r,s}$. In particular,
    $\epsilon_{r,r}=\epsilon_r$ and $\delta_{r,r}=\delta_r$.
    We also
    have $p>1$ without the subscript, chosen before the game starts.
  For Alice to choose $(\eps_s,\delta_s)$, it will be important to see all
  matrices $A_r$ for $r<s$. First, we claim the following:

  \begin{claim} \label{claim:smalldelta}
    By choosing $\delta_i$ sufficiently
    small at each step, Alice can ensure that 
    \[
      \delta_{r,s} < \frac{1}{2^s},
    \]
    for all $r\geq 0$ and $s > r$.
  \end{claim}

  \begin{proof}[Proof of the claim]
      \renewcommand{\qedsymbol}{$\triangle$}
      Let Alice choose $\delta_0 = 1$, and for each subsequent $s > 0$, choose $\delta_s$ so that \[nP_sp\delta_s < 2^{-s},\] where $P_s = \max\{K_{r, t} \mid r \le t < s \}$.
      Then by \Cref{lem:matrixproduct} for $A_{r,s} = A_{r,s-1}A_s$, we have
      \[
        \delta_{r, s} \le nK_{r, s-1}p\delta_s \leq nP_sp\delta_s < 2^{-s} <
      2^{-r}
    \]
    holds for each $r \ge 0$ and $s>r$, concluding the claim.
  \end{proof}
In light of \Cref{claim:smalldelta}, set $\ov{\delta}_r:= 2^{-r}$ for $r \ge 0$.
  
  Next, note that the increase in the $p$-constant and the $\eps$-constant of
  $AA'$ is bounded by the same amount
    $nK_A\eps + np_A\delta_Ap.$
  We apply this to $A_{r,s}A_{s+1}=A_{r,s+1}$.
  Then the corresponding increment to $\eps_{r,s}$ and $p_{r,s}$ is:
  \[
    nK_{r,s}\eps_{s+1} + np_{r,s}\delta_{r,s}p.
  \]
  By choosing small $\eps_{s+1}$, Alice can ensure
  that the first term is arbitrarily small. \Cref{claim:smalldelta} ensures the second 
  term to be smaller than $\frac{npp_{r,s}}{2^s}$. Hence, we can take the sum
  of the two terms to be no larger than $\frac{npp_{r,s}}{2^s}$. In other words,
  by choosing small $(\eps_{s+1},\delta_{s+1})$, we have for $r \le s$:
  \begin{align}
    p_{r,s+1} &\le p_{r,s} + \frac{npp_{r,s}}{2^s} = p_{r,s}\left( 1 + \frac{np}{2^s} \right), \label{ineq:p_bound}\tag{$\ast_1$} \\ 
  \eps_{r,s+1} &\le \eps_{r,s} + \frac{npp_{r,s}}{2^s}. \label{ineq:eps_bound}\tag{$\ast_2$}
  \end{align}
  Multiplicatively telescoping Inequality \eqref{ineq:p_bound}, and noting $p_{r,r}=p$,
  \[
    p_{r,s+1} \le p \left( 1 + \frac{np}{2^{r}} \right)\cdots\left( 1 +
      \frac{np}{2^{s-1}} \right)\left( 1 + \frac{np}{2^{s}} \right).
  \]
  Note the infinite product
  \begin{align}
    \label{ineq:infiniteProduct_p}
    \prod_{k = r}^\infty \left( 1+ \frac{np}{2^k} \right) \tag{$\ast\!\ast$}
  \end{align}
  converges, since the infinite sum
    $\sum_{k = r}^\infty \frac{np}{2^k}$
  converges. Take $\ov{p}_r$ to be the infinite
  product \eqref{ineq:infiniteProduct_p} multiplied by $p$. Since
  $p=p_{r,r}<\ov{p}_r$, it follows that $p_{r,s}\le \ov{p}_r$ for all
  $r \ge 0$ and $s \ge r$, and $\{\ov{p}_r\}_{r \ge 0}$ is uniformly bounded
  above by $
    \prod_{k = 0}^\infty \left( 1+ \frac{np}{2^k} \right)$

  Similarly, picking up from Inequality \eqref{ineq:eps_bound}, we have
  \[
  \eps_{r,s+1} \le \eps_{r,s} + \frac{npp_{r,s}}{2^s} \le
  \eps_{r,s} + \frac{np\ov{p}_r}{2^s},
\]
for each $r \ge 0$ and $s \ge r$.
By telescoping with a fixed $r$, and varying $s$ to $r$,
\[
  \eps_{r,s+1} \le \eps_r + \sum_{k=r}^s \frac{np\ov{p}_r}{2^k}.
\]
Therefore, by letting $\ov{\eps}_r$ be 
  $\ov{\eps}_r := \epsilon_r+\sum_{k=r}^\infty \frac{np\ov{p}_r}{2^k},$
we have that $\eps_{r,s} \le \ov{\eps}_r$ for every $r \ge 0$ and $s \ge r$, and that
$\ov{\eps}_r \to 0$ as $r \to \infty$ (since the infinite sum converges to $0$
as $r \to 0$ and Alice can ensure that
$\epsilon_r\to 0$), concluding the proof.
\end{proof}

\subsection{Constructing ergodic currents from diagonal-dominant matrices}

\begin{theorem} \label{thm:klp_then_ergodic}
Let $(G_s)_{s \le 0}$ be a reduced unfolding sequence with legal lamination $\Lambda \subset G_0$, and $\calM = (M^{r,s})_{r \le s \le 0}$ be the inverse system of $n \times n$ transition matrices associated with the unfolding sequence $(G_s)_{s \le 0}$. Let $1
\le m \le n$. Suppose that for each $s \le 0$, there exists a positive triple $(\ov{p}_s, \ov{\eps}_s,
\ov{\delta}_s)$ such that for every $r < s$, $M^{r,s} \in \calM$ is an $(m, \ov{p}_s,
\ov{\eps}_s,\ov{\delta}_s)$ matrix such that 
\begin{enumerate}[label=(\roman*)]
\item $\{\ov{p}_s\}_{s \le 0}$ is uniformly bounded,
\item $\lim_{s \to -\infty} \ov{\eps}_s =0$, and
\item $\lim_{s \to -\infty} \ov{\delta}_s =0$.
\end{enumerate}
Then the legal
lamination $\Lambda$ admits at least $m$ linearly independent ergodic measures.
\end{theorem}

\begin{proof}
    The proof of this result mirrors the proof of \Cref{thm:klp_then_ergodic_folding}.
    Firstly, by \Cref{thm:NPR-unfoldingsequence} it suffices to show the reduced unfolding sequence $(G_s)_{s \le 0}$ admits at least $m$ linearly independent ergodic measures. Recall $\calC((G_s)_{s\le 0})$ is defined as the cone of non-negative vectors in the $n$-dimensional vector space whose vectors are sequences $(\vec{v}_s)_{s \le 0}$ with $\vec{v}_s \in \R^n$ for every $s \le 0$ satisfying the compatibility condition
    \[
        \vec{v}_s = M_{s-1}\vec{v}_{s-1},
    \]
    for every $s \le 0$, where $M_{s-1}=M^{s-1,s}$. Denoting by $\calC(G_s)$ the space of currents on $G_s$,
    the space $\mathcal{C}((G_s)_{s \le 0})$ is isomorphic to the following
    space,
    
    \[
        \mathcal{C}_0 := \lim_{r \to -\infty} \bigcap_{k=r}^0 M^{k,0}(\calC(G_k)).
    \]
    On the other hand, we know  $\calC(G_s) \cong \R^n_{\ge 0}$ for every $s \le 0$, so it follows that
    \[
        \mathcal{C}_0 = \lim_{r \to -\infty} \bigcap_{k = r}^0 \left(\text{the nonnegative cone of columns of $M^{k,0}$}\right).
        \]

   After truncating the sequence and reindexing, we may assume that $\bar\eps_0$ is
   so small that any $n\times m$ matrix of the form

    \[
    P:=
    \begin{bmatrix}
        1 & * & * & \ldots & * \\
        * & 1 & * & \ldots & * \\
        * & * & 1 & \ldots & * \\
        \vdots & \vdots & \vdots & \ddots & \vdots \\
        * & * & * & \cdots & 1 \\
        * & * & * & \cdots & * \\
        * & * & * & \cdots & * \\
    \end{bmatrix},
    \] 
    with entries above the diagonal bounded by $\bar\eps_0$ and
    entries below the diagonal bounded by $\sup\bar p_s$ will have $m$ linearly independent columns.

    Now we claim that $\mathcal{C}_0$ contains at least $m$ linearly independent
    vectors, by showing that the first $m$ columns of the matrices in the
    sequence $\{M^{k,0}\}_{k=-\infty}^{0}$ \emph{projectively} (with
    separate scaling in each column) converge to $m$
    vectors as above. The details are analogous to the proof of
    \Cref{thm:klp_then_ergodic_folding} and are thus omitted.
\end{proof}

\subsection{End game: Our unfolding sequence}
\label{ssec:lowerboundFinale}
In this section, for each $n\geq 4$, we construct an unfolding sequence of graphs of rank $n$ whose transition matrices satisfy the conditions of \Cref{thm:klp_then_ergodic}.

\begin{theorem} \label{thm:lowerboundFinale_unfolding} 
For each $n \ge 4$, there exists a reduced unfolding sequence
$\calG=(G_k)_{k \le 0}$ of graphs
of rank $n$ whose legal lamination admits at least $2n-6$ linearly independent ergodic measures.
\end{theorem}

\begin{proof}

Following the construction in \Cref{thm:lowerboundFinale_folding}, build the
same graph $\G_{a_0}$ with the maps $\calF^{(k)}$, to form a reduced unfolding sequence
$\calG = (G_k)_{k\le 0}=\{\phi_{k-1} : G_{k-1} \to G_{k}\}_{k \le 0}$ of graphs
of rank $n$ with $G_k = \G_{a_0}$ and $\phi_k = \calF^{(k)}$. The only
difference we will make for our current setting of an unfolding sequence is that
we will label the edges of $\Gamma_{a_0}$ so that
\[
(e^1,\ldots,e^{2n-3})=(a_{n-3},a_{n-2},\ldots, a_1, b_{n-3},b_{n-2},\ldots, b_1,a_0,b_0,c).
\] 
Then, our
transition matrix $M$ for $\calF$ will be of the form:
\begin{center}
  \scalebox{0.9}{$
    \begin{pmatrix}
    \setlength\arraycolsep{5pt}
    \begin{matrix}
        2\alpha_{n-3}  &  3             &  2                 & 2          & \cdots    & 2 \\
                       &  2\alpha_{n-4} &  3                 & 2          & \cdots    & 2 \\
                       &                &  \ddots            & \ddots     & \ddots    & \vdots \\
                       &                &                    & 2\alpha_3  & 3         & 2 \\
                       &                &  \ph{\alpha_{n-5}}      &            & 2\alpha_2 & 3 \\
                       &                &                    &            &           & 2\alpha_1
    \end{matrix}
    &
    \rvline 
    &
    \setlength\arraycolsep{8pt}
    \begin{matrix}
        2                         &  2                 &   2       &   2        &   \cdots  & 2 \\
        \ph{\beta_{n-3}}          &  2                 &   2       &   2       &   \cdots  & 2 \\
                                  & \ph{\beta_{n-4}}   &   \ddots  &   \ddots   &   \ddots  & \vdots \\
                                  &                    &           &   2        &    2      & 2 \\
                                  &                    &           &            &    2      & 2 \\
                                  &                    &           &            &           & 2
    \end{matrix}
    &
    \rvline
    &
    \begin{matrix}
      0 & 0      & 2 \\
      0 & 0      & 2 \\
  \vdots & \vdots & \vdots \\
      0 & 0      & 2 \\
      0 & 0      & 2 \\
      1 & 0      & 2
    \end{matrix}\\
    \hline 
    \setlength\arraycolsep{7pt}
    \begin{matrix}
          \alpha_{n-3} &   \alpha_{n-4}       &   3             & 3              & \cdots         &     3 \\
                       &   \alpha_{n-4}       &   \alpha_{n-5}  & 3              & \cdots         &     3 \\
                       &                      &   \ddots        & \ddots         & \ddots         &     \vdots \\
                       &                      &                 & \alpha_3       & \alpha_2       &     3 \\
                       &                      &                 &                & \alpha_2       &     \alpha_1 \\
                       &                      &                 &                &                &     \alpha_1
    \end{matrix}
    &
    \rvline 
    &
    \setlength\arraycolsep{6.5pt}
    \begin{matrix}
       \ph{2}\beta_{n-3}    &  4             &   3                    & 3          & \cdots         & 3 \\
                       &  \beta_{n-4}   &   4                    & 3          & \cdots         & 3 \\
                       &                &   \ddots               & \ddots     & \ddots         & \vdots \\
                       &                & \ph{\beta_{n-3}}       & \beta_3    & 4              & 3 \\
                       &                &                        &            & \beta_2        & 4 \\
                       &                &                        &            &                & \beta_1
    \end{matrix}
    &
    \rvline
    &
    \begin{matrix}
      0 & 0      & 3 \\
      0 & 0      & 3 \\
  \vdots & \vdots & \vdots \\
      0 & 0      & 3 \\
      0 & 0      & 3 \\
      0 & 1      & 0
    \end{matrix}\\
    \hline
    \setlength\arraycolsep{9pt}
    \begin{matrix}
        1            & \ph{\alpha}0\ph{\alpha_n} & 0 & \ph{\alpha}0 & \cdots & 0\\
        \alpha_{n-3} & \ph{\alpha}3\ph{\alpha_n} & 3 & \ph{\alpha}3 & \cdots & 3\\
        1            & \ph{\alpha}1\ph{\alpha_n} & 1 & \ph{\alpha}1 & \cdots & 1
    \end{matrix}
    &
    \rvline
    &
    \setlength\arraycolsep{8.3pt}
    \begin{matrix}
       \ph{\beta}0\ph{\beta} & \ph{\beta_n}0\ph{\beta} & \ph{1}0 & \ph{\beta}0 & \cdots & 0\\
       \ph{\beta}4\ph{\beta} & \ph{\beta_n}3\ph{\beta} & \ph{1}3 & \ph{\beta}3 & \cdots & 3\\
       \ph{\beta}1\ph{\beta} & \ph{\beta_n}1\ph{\beta} & \ph{1}1 & \ph{\beta}1 & \cdots & 1
    \end{matrix}
    &
    \rvline
    &
    \begin{matrix}
        0 & 0 & 0\\
        0 & 0 & 3\\
        0 & 0 & 1
    \end{matrix}
    \end{pmatrix}
    $}
  \end{center}
 which can be seen as Bob's $(2n-6, p=2, \eps_r, \delta_r)$-matrix with appropriate choice of $(\alpha_1,\ldots,\allowbreak\alpha_{n-3},\beta_1,\ldots,\beta_{n-3})$ for each of Alice's choices of $\eps_r >0$ and $\delta_r>0$. Hence, Alice's winning sequence $\{(\eps_r, \delta_r)\}_{r \le 0}$ in \Cref{thm:TheGame} satisfying the conditions (i--iii) of \Cref{thm:klp_then_ergodic} provides at least $m=2n-6$ linearly independent ergodic measures on the legal lamination of our reduced unfolding sequence $\calG$.
\end{proof}

\section{Upper Bound $2n-1$}\label{sec:upperbound}
In \Cref{thm:arational_limiting} we have seen that for every arational limiting
tree $T$, there exists a reduced folding sequence limiting to $T$. Likewise, by \Cref{thm:arational_limiting_unfolding}, for every
arational legal lamination $\Lambda$, there exists an unfolding sequence
whose legal lamination has diagonal closure equal to $\Lambda$. Hence, to establish the number of ergodic measures on an
arational tree or arational legal lamination, it suffices to prove the following.

\begin{theorem}
\label{thm:Main_Upperbound}

Let $( G_{s} )_{s \ge 0}$ be a reduced folding sequence of graphs of rank $n$. Then, the
number of linearly independent ergodic measures on the limiting tree of the
folding sequence is bounded above by $2n-1$.
Similarly, given a reduced unfolding sequence $( G_s )_{s \le 0}$ of graphs of rank $n$, the
number of linearly independent ergodic measures on the legal lamination of the
unfolding sequence is also bounded above by $2n-1$.
\end{theorem}

We will order the edges of $G_0 = G$ so that if we build the graph $G$ by adding one edge at a time in the increasing
order, the `complexity' (\Cref{def:complexity}) of the resultant graph increases for
each addition. We will then use that the total increase of complexity is bounded above
by $-\chi(G)=n-1$.

Intuitively, we give an ordering so that smaller edges are `easier' to be
traversed by the images of other edges. For folding sequences, smaller edges correspond to
\emph{wider} edges, and for unfolding
sequences, they correspond to \emph{shorter} edges. This correspondence is analogous to the idea that in an electrical current, a wider and shorter resistor
has a smaller resistance than a thinner and longer resistor.

To obtain such an ordering from a \emph{folding sequence} $(G_s)_{s \ge 0}$,
we consider the \textbf{frequency current} $\mu$ on $G_0$,
namely $\mu(e^i) = 1$ for $i=1,\ldots,n$. Then, $\mu$ pushes forward to $\mu_s$
on $G_s$ for
$s \ge 0$.
We then pass to a subsequence to ensure that for all pairs of edges $e, e'$ in $G_0$, $\lim_{s
  \to \infty} \frac{\mu_s(e_s')}{\mu_s(e_s)}$  
exists in $[0,\infty]$. Next, for edges $e, e'$ of $G_0$, we define:
\begin{align*}
  e < e' \quad & \text{ if } \quad \lim_{s \to \infty} \frac{\mu_s(e_s')}{\mu_s(e_s)} = 0, \text{ and} \\
  e \sim e' \quad & \text{ if } \quad \lim_{s \to \infty} \frac{\mu_s(e_s')}{\mu_s(e_s)} \in (0,\infty),
\end{align*}
and write $e \lesssim e'$ if $e < e'$ or $e \sim e'$. Here, we defined the order in a way that
$e<e'$ if the width of $e$ grows \emph{faster} than that of $e'$.

Similarly, for an \emph{unfolding sequence} $(G_s)_{s\leq 0}$, let $\ell$ be the \textbf{simplicial length function} on
$G_0$, namely $\ell(e^i)=1$ for $i=1,\ldots, n$. Then, $\ell$ pulls back to
$\ell_s$ on $G_s$
for $s \le 0$.
We then pass to a subsequence
such that for each pair of edges $e, e'$ of $G_0$, either $\lim_{s \to
  -\infty}\frac{\ell_s(e_s)}{\ell_s(e_s')}$ or $\lim_{s \to -\infty}\frac{\ell_s(e_s')}{\ell_s(e_s)}$ exists.
Finally, for edges $e, e'$ in $G_0$, we define:
\begin{align*}
  e < e' \quad & \text{ if } \quad \lim_{s \to -\infty} \frac{\ell_s(e_s)}{\ell_s(e_s')} = 0, \\
  e \sim e' \quad & \text{ if } \quad \lim_{s \to -\infty} \frac{\ell_s(e_s)}{\ell_s(e_s')} \in (0,\infty),
\end{align*}
and denote by $e \lesssim e'$ if $e \sim e'$ or $e < e'$.
Thus $e<e'$ if the length of $e$ grows \emph{slower} than that of $e'$.

In \Cref{ssec:UBsetup_unfolding,ssec:UBsetup_folding}, we prepare the preliminary lemmas necessary for the proof of \Cref{thm:Main_Upperbound}. In \Cref{ssec:upperbound_proof}, we utilize those lemmas along with a common combinatorial trick to establish the
$2n-1$ upper bound for both unfolding and folding sequences at the same time.

\subsection{Setup for unfolding sequence}
\label{ssec:UBsetup_unfolding}
The first step is to exploit the reduced hypothesis on the unfolding sequence. Recall that given finite directed labeled paths $\gamma$ and $\omega$, we use the notation $\langle \gamma, \omega \rangle$ to denote the number of times $\gamma$ appears as a (directed) subpath of $\omega$.

\begin{lemma}[Reduced Unfolding Sequence]\label{lem:reduced_unfolding}
  Let $(G_s)_{s \le 0}$ be an unfolding sequence with every graph
  isomorphic to $G_0 = G$. If $(G_s)_{s \le 0}$ is reduced, then for edges $e,e'$ of $G$ and any number
  $K>0$, there exists $N<0$ such that for each $s < N$, we have that
  \[
    \langle e'_0, \phi_{s,0}(e_s) \rangle > K.
  \] 
\end{lemma}

\begin{proof} 
  Let $e_0 \in EG_0$ and $K>0$. For each $s \le 0$, consider the subgraph $H_s$ of $G_s$
  whose edge set is given by
  \[
    EH_s = \{f_s \in EG_s\ |\ \langle e_0, \phi_{s,0}(f_s) \rangle \le K \}.
  \]
  For each $s < 0$, we claim $\phi_{s,s+1}(H_s) \subset H_{s+1}$. Indeed, pick
  $f_s \in EH_s$. Because $\phi_{s,0} =
  \phi_{s+1,0} \circ \phi_{s,s+1}$, 
  \[
    \langle e_0, \phi_{s+1,0}(\phi_{s,s+1}(f_s)) \rangle \le K.
  \]
  Now for any $g_{s+1}$ contained in $\phi_{s,s+1}(f_s)$
  we have
  \[
    \langle e_0, \phi_{s+1,0}(g_{s+1}) \rangle \le K,
  \]
  showing $g_{s+1} \in EH_{s+1}$. This establishes the claim and
  therefore $(H_s)_{s \le 0}$ is an invariant sequence of subgraphs.

  As $(G_s)_{s \le 0}$ is reduced, it follows that $(H_s)_{s \le 0}$ is eventually either
  the sequence of whole graphs or the sequence of empty sets. In the
  latter case we are done so
  assume that $H_s = G_s$ for every $s<N$ for some $N<0$.
  As the transition maps of the unfolding sequence are surjective, the sequence $(\langle e_0, \phi_{s, 0}(G_s)
  \rangle)_{s<N}$ is a nondecreasing sequence of positive integers bounded above by $K|EG|$. Hence,
  we may assume the sequence $(\langle e_0, \phi_{s, 0}(G_s) \rangle)_{s<N}$ is
  constant. This implies that if $e'_s\subset G_s$ is any edge such
  that $\phi_{s,0}(e_s')$ crosses $e_0$ and if $x_s$ is a point in the
  interior of $e_s'$ then the preimages of $x_s$
  in $G_t$ for $t<s$ all consist of a single point $x_t$. But then the
  sequence $K_t$ of subgraphs of $G_t$ consisting of all edges except
  the one containing $x_t$ is invariant, contradicting
  the assumption that the sequence $G_t$ is reduced ($G$ has rank
  more than one). 
\end{proof}

\begin{lemma}[Witness Lemma for Unfolding Sequence] \label{lem:witness-unfolding}
  Let $(G_s)_{s \le 0}$ be a reduced unfolding sequence with every graph isomorphic to
  $G$. For $e \in EG$ such that $e \in H^i$ for some $i>0$, there exists an immersed loop
  $\alpha$ in $G$ such that:
  \begin{enumerate}[label=(\arabic*)]
  \item $\alpha$ contains $e$,
  \item $\alpha$ does not go through $e' \in EG$ with
    $e' > e$, and
  \item $\alpha$ does not go through $e' \in EG$ with $e' \sim e$ and $e' \in
    H^j$ for some $j\neq i$.
  \end{enumerate}
\end{lemma}

\begin{proof}
  Let $e \in H^i$ for some $i>0$ and $e' \in EG$ be another edge. 
   For $r<s$, denote by $M:=M^{r,s}$ the
  transition matrix for the transition map $G_r \to G_s$, and denote by
  $\wt{M}:=\wt{M}^{r,s}$ the transition matrix for the linear map $V_r \to V_s$
  induced by the map $G_r \to G_s$, with normalized bases.
  Up to taking a subsequence, the normalized transition matrices $\wt{M}^{r,s}$
  converge to the retraction $S$.
  Hence, by \Cref{cor:normalized_diagonal}, we have that for large enough $|r|,|s|$ such that
  $r<s<0$, the diagonal entry $\wt{M}_{ee}:=\wt{M}_{ee}^{r,s}$ of the normalized
  matrix $\wt{M}^{r,s}$ is
  bounded away from 0. Moreover,
  according to \Cref{lem:transition_relation}, we have that
  \begin{align}\label{eqn:comparison}
    \frac{\wt{M}_{e'e}}{\wt{M}_{ee}} =
    \frac{M_{e'e}\frac{\ell_r(e')}{\ell_s(e)}}{M_{ee}\frac{\ell_r(e)}{\ell_s(e)}}
    = \frac{M_{e'e}}{M_{ee}} \cdot \frac{\ell_r(e')}{\ell_r(e)}  \tag{$\star$}.
  \end{align}
  Since $\wt{M}_{ee}>\epsilon>0$ and $\wt{M}_{e'e}$ is bounded (as $\wt{M}^{r,s}$
  converges to $S$), $\frac{\wt{M}_{e'e}}{\wt{M}_{ee}}$ is bounded above.
  We divide into two cases with a goal to achieve Property (2) and (3) for a
  witness loop.

  \textbf{Case 1. $e'>e$}.
  For this case, $e'>e$ implies $\frac{\ell_r(e')}{\ell_r(e)}\to \infty$ as $r \to
  -\infty$. Therefore, it follows that $\frac{M_{e'e}}{M_{ee}} \to 0$ as $r,s \to -\infty$.

  \textbf{Case 2. $e' \sim e$ but $e' \in H^j$ for some $j \neq i$}.
   Again by
  \Cref{cor:normalized_diagonal}, we have
  \[
    \lim_{r \to -\infty}\lim_{\substack{s<r\\ s \to
        -\infty}}\frac{\wt{M}_{e',e}}{\wt{M}_{e,e}} = 0.
  \]
  However, as $e \sim e'$, the lengths $\ell_r(e)$ and $\ell_r(e')$ are
  comparable as $r \to -\infty$, so by Equation \eqref{eqn:comparison}, it
  follows that
  \[
    \lim_{r \to -\infty}\lim_{\substack{s<r\\ s \to
        -\infty}}\frac{M_{e',e}}{M_{e,e}} = 0.
  \]
  
  All in all, in both cases, the number of $e_s'$ in $\phi_{r,s}(e_r)$ is sublinear in the number of $e_s$
  in $\phi_{r,s}(e_r)$ as $r<s$ and $r,s \to -\infty$. 
  Thus, in the immersed path $\phi_{r,s}(e_r)$ in $G_s$ there
  exist two occurrences of $e_s$'s with the same
  orientation, say $e_{s,1}Ae_{s,2} \subset \phi_{r,s}(e_r)$, such that in 
  the gap $A$ between them
  there is no $e'$ such that $e'>e$ or $e' \sim e$ with
  $e' \in H^j$ for some $j \neq i$. Then $\alpha=e_sA$ is the desired witness loop.
\end{proof}

\subsection{Setup for folding sequence}
\label{ssec:UBsetup_folding}
  Mirroring \Cref{ssec:UBsetup_unfolding}, we set up lemmas for folding sequences.
\begin{lemma}[Reduced Folding Sequence]\label{lem:reduced_folding}
  Let $(G_s)_{s \ge 0}$ be a folding sequence with every graph
  isomorphic to $G$. If $(G_s)_{s \ge 0}$ is reduced, then for edges $e,e'$ of $G$ and any number
  $K>0$, there exists $N>0$ such that
  \[
    \langle e'_s, \phi_{0,s}(e_0) \rangle > K, \]  for each $s > N$.
\end{lemma}

\begin{proof}
  Fix $K>0$ and $s>0$. Let $e_s \in EG_s$. For $0 \le i \le s$, consider the subgraph $H_s^i
  \subset G_i$ whose edge set is given by:
  \[
    EH_s^i = \{f_i \in EG_i\ |\ \langle e_s, \phi_{i,s}(f_i) \rangle \le K\}.
  \]
  We note here $H_s^s =G_s$. 
  Using the same argument as in the proof of \Cref{lem:reduced_unfolding}, we
  have $\phi_{i,i+1}(H_s^i) \subset H_{s}^{i+1}$ for $0 \le i < s$.

  Now we want to show that all but finitely many $s>0$,
  $H_s^0$ is the empty set, for every choice $e_s \in EG_s$.
  Suppose for the sake of contradiction that there exist infinitely many $s>0$
  such that $H_s^0$ is not empty for some choice of $e_s \in EG_s$.
  As $G_0$ is a finite graph, there could be only finitely many different
  subgraphs of $G_0$. Hence, we may take an infinite subsequence of such $s$'s
  such that $H_s^0$'s are the same nonempty subgraph of $G_0$. By taking a
  further subsequence, we may assume the $H_s^0$'s are constructed with $e_s$ sharing the same
  $e \in EG$. Repeating the same argument, we may assume there exists a sequence
  $(s_j)_{j>0}$ of positive integers such that for each $i\ge 0$, whenever
  defined, $H_{s_j}^i$'s
  are all identical for $j>0$. Hence, for each $i \ge 0$, set $H_i = H_{s_j}^i$
  for any $j>0$ such that $s_j \ge i$. Then as $\phi_{i,i+1}(H_s^i) \subset
  H_s^{i+1}$, it follows that $(H_i)_{i \ge 0}$ is an invariant sequence of
  subgraphs. By the reducibility of $(G_s)_{s>0}$, it follows that $(H_i)_{i\ge
    0}$ is eventually either the sequence of empty sets or the sequence of the whole
  graphs.

  Suppose that $(H_i)_{i \ge 0}$ is eventually the sequence of the whole graphs.
  Namely, say there is $N>0$ such that $H_i = G_i$ for each $i \ge N$. Then
  by the identification above, there exists an increasing sequence $(t_j)_{j \ge
    1}$ such that $t_1 > N$ and $e_{t_j} \in EG_{t_j}$ such that
  \[
    H_N = \{f_N \in EG_N\ |\ \langle e_{t_j}, \phi_{N, t_j}(f_N)  \rangle \le K\}
  \]
  for each $j \ge 1$. Since $H_N = G_N$, we have
  \[
    \langle e_{t_j}, \phi_{N, t_j}(G_N) \rangle \le |EG| \cdot K
  \]
  for each $j \ge 1$. Recall we have reduced to the case where $H_N$'s are
  constructed with the same $e \in EG$. 
  Then as $\phi_{N,t_j}$'s are
  surjective, it follows that the sequence $(\langle e_{t_j}, \phi_{N,t_j}(G_N)
  \rangle)_{j \ge 1}$ is a nondecreasing sequence, bounded above by $K|EG|$.
  Hence it is eventually constant. Say it is constant for $j \ge M$.
  This means that for $i \ge t_M$, no folding is happening on the preimages of
  $e_{t_j}$ in $G_i$, where $j>0$ is any integer such that $t_j>i$. Now for $i \ge t_M$, take the subgraph $H'_i$
  of $G_i$ whose edge set consists of the edges not mapped over $e_{t_j}$ for
  every $j$ such that $t_j > i$. Then $(H'_i)_{i \ge t_M}$ is a sequence of
  invariant subgraphs, which cannot be the whole graph as $\phi_{i, t_j}$ is
  surjective. Hence by the reducibility of $(G_s)_{s \ge 0}$, $(H_i')_{i \ge
    t_M}$ is the sequence of empty graphs. Running the same argument with
  different choice of $e_{t_j}$ that appears in the identifications of $H_N$, we
  conclude that every sequence $(H_i')$ associated to each choice of $e_{t_j}$
  is the sequence of empty sets. However, this implies that eventually there is no
  folding happening on $(G_s)_{s \ge 0}$, contradiction.

  All in all, we conclude that $(H_i)_{i \ge 0}$ is eventually a sequence of
  empty sets. By taking larger $K >0$ and running the same argument, we have
  $\lim_{s \to \infty}\langle e_s, \phi_{0,s}(f_0) \rangle = \infty$ for every
  choice of $e,f \in EG$, as desired.
\end{proof}

We have very similar witness lemma for the folding sequence. 
\begin{lemma}[Witness Lemma for Folding Sequence]\label{lem:witness-folding}
  Let $(G_s)_{s \ge 0}$ be a reduced folding sequence with every graph isomorphic to
  $G$. For $e \in EG$ with $e \in H^i$ for some $i > 0$, there exists an immersed loop
  $\alpha$ in $G$ such that:
  \begin{enumerate}[label=(\arabic*)]
  \item $\alpha$ contains $e$,
  \item $\alpha$ does not go through $e' \in EG$ with
    $e' > e$, and
  \item $\alpha$ does not go through $e' \in EG$ with $e' \sim e$ and $e' \in
    H^j$ for some $j\neq i$ such that $j>0$.
  \end{enumerate}
\end{lemma}

\begin{proof}
  This is the same as \Cref{lem:witness-unfolding} except we use 
  $\wt{M}:=\wt{M}^{r,s}$ the normalized transition matrix for the linear map
  $V_s \to V_r$ induced by
  the map $G_r \to G_s$.
\end{proof}

  \subsection{Upper bound $2n-1$}
  \label{ssec:upperbound_proof}

Suppose there are $k$ linearly independent ergodic probability measures on the
limiting tree of a reduced folding sequence, $(G_s)_{s\geq 0}$, or on the legal lamination of a reduced unfolding sequence, $(G_s)_{s\leq 0}$.
Then applying \Cref{thm:NPRdecomp-folding} or
\Cref{thm:NPRdecomp-unfolding}, respectively, we obtain the transverse
decomposition $G=H^0\sqcup H^1 \sqcup \ldots \sqcup H^k$.
Now we bound $k$ by carefully measuring the \emph{complexity} of a graph in the following way.

\begin{definition}\label{def:complexity}
Let $G$ be a (possibly disconnected) non-contractible graph and $N_1,\ldots,N_k$
be the \textit{non-contractible} connected components of $G$. Define the
\textbf{complexity}, $\chi_-(G)$, of $G$ as:
\[
    \chi_-(G) := \sum_{i=1}^k -\chi(N_i)=
\sum_{i=1}^k \left(\rk(N_i)-1\right),
\]
where $\chi(N)$ is the Euler characteristic of $N$, which is equal to $1-\rk(N)$
for a non-contractible graph $N$.
\end{definition}
  
  Now we prove \Cref{thm:Main_Upperbound} for both folding and unfolding sequences at once.
  \begin{proof}[Proof of \Cref{thm:Main_Upperbound}.]
    Let $A_1,\ldots,A_{k'}$ be the equivalence classes of edges in $G \setminus
    H^0$ such that if $e \in A_i$ and $e' \in A_j$ with $i<j$, then $e < e'$.
    Subdivide each $A_j$ further, by letting $B_j^i:= A_j \cap H^i$ for each
    $i=1,\ldots,k$. It could be that some of $B_j^i$'s are empty.
    Then take the initial collection $I$ from $G$ as
    \begin{itemize}
    \item All edges in $H^0$.
    \item All single-edged loops. Say there are $s$ such loops.
    \item All rank 1 components of $B_j^i$'s that are not attached to any of $H^0$, the $s$
  single-edged loops, or $B_{j'}^{i'}$ for either 1) $j'<j$ or 2) $j'=j$ and $
  i'<i$. Say there are $m$ such components.
\end{itemize}

Then we attach $B_j^i$'s (those that are not part of $I$) to $I$ one by one in this order:
\[
  B_{1}^1,B_1^2,\ldots,B_1^{k},B_2^1,B_2^2,\ldots,B_2^k,\ldots, B_{k'}^1,B_{k'}^2\ldots,B_{k'}^k.
  \]

  We observe that by our choice of the initial collection, and
  \Cref{lem:witness-folding,lem:witness-unfolding}, each addition of
  $B_j^i$ contributes to the increase of complexity of
  $\chi_-$ at least by 1. Since $\chi_-(G)=n-1$ and $\chi_-(I) \ge 0$, it
  follows that there can be at most $n-1$ addition of $B_j^i$'s, which means
  that at most $n-1$ new ergodic measures can be added, as each $B_j^i = A_j\cap
  H^i$ supports at most one ergodic measure (as it could be empty).

  On the other hand, the initial collection $I$ of edges can support at most $s+m$
  ergodic measures. This is because $H^0$ supports none, each single loop can
  admit at most one ergodic measure, and each $B_j^i$ can support at most one
  ergodic measure. However, as $I$ has total rank at least $s+m$, it follows that $s+m
  \le n$.

  All in all, the number $k$ of ergodic measures supported on $G$ is:
  \[
    k \le (s+m) + (n-1) \le 2n-1,
  \]
  which proves the upper bound of $2n-1$. The equality is possible only
  if $s+m = n$ and if
  each addition of $B_j^i$ exactly increases the $\chi_-$ term by 1. \end{proof}

\section{Arationality and Nongeometricity} \label{sec:arationality}
\subsection{Folding sequence}
Recall that an $\R$-tree $T$ is \textbf{arational} if for every proper free
factor $H \le \bbF$, the action $\bbF \curvearrowright T$ restricted to $H$ is
free and discrete. In this section, we prove that our construction in
\Cref{ssec:lowerboundFinale_folding} can be improved to have the limiting tree be both non-geometric and arational.
\begin{theorem} \label{thm:arational_folding}
    There exists a folding sequence whose limiting tree is non-geometric, arational and admits $2n-6$ linearly independent ergodic length measures.
\end{theorem}
We argue that the limiting tree of our folding sequence can be arranged to be arational and free, after perhaps slightly modifying the folding sequence, but without affecting the winning strategy for Alice in the Alice-Bob game, as described in \Cref{thm:TheGame_folding}. 

We endow the folding sequence $(G_i)_{i \ge 0}$ given in
\Cref{ssec:lowerboundFinale_folding} with a metric by assigning arbitrary
positive lengths to the edges of $G_i$'s such that they are consistent under
pulling back. We then scale these lengths so that the volume of $G_0$ is 1. Under this
metric, all folding maps become isometric immersions on each edge, and such a
sequence always converges to an $\R$-tree in $\ov{CV_n}$ by
  \Cref{lem:folding_limiting_tree}. 
  
Take $G_t$ as an induced folding \emph{path} of the folding sequence
$(G_i)_{i\ge 0}$, and reparametrize the folding path $G_t$ so that the
integer values of $t$ correspond to the times when illegal turns arrive at a
valence $\ge 4$ vertex. Thus, $t$ is not an integer if and only if $G_t$ has a
vertex with two gates.

We start with the following observation.
\begin{lemma}
    In this pullback metric on the folding sequence $(G_i)_{i\ge 0}$, every edge
    of $G_i$ has positive length.
\end{lemma}

\begin{proof}
    Suppose that $a_1$ in some $G_i$ has length zero. Then all edges in
    $G_{i+1}$ that $a_1$ maps over also have length 0. From the first matrix in
    the proof of \Cref{thm:lowerboundFinale_folding}, we see that the edges
    $a_1$ maps over include $a_1,a_2,b_1,b_2$. By the same argument, edges $a_1,
    \ldots, a_{n-3}, b_1, \ldots, b_{n-3}$ have length 0 in $G_{i+n-4}$, and
    then all edges have length 0 in $G_{i+n-3}$, contradicting that the length
    function on this graph pulls back to $G_0$ with volume 1. A similar argument works if
    any other edge in $G_i$ in place of $a_1$ has length 0.
\end{proof}

To prove \Cref{thm:arational_folding}, we claim a slightly stronger assertion, but we recall first how we count the number of illegal turns. A gate containing $k$ oriented edges contributes $k-1$ illegal turns. To find the number of illegal turns in a graph, we add those up over all gates. 
\begin{lemma}\label{lem:arational_main}
    Let $H \le F_n \cong \pi_1(G_0)$ be a proper free factor, or a maximal
    cyclic subgroup. Then the folding sequence $(G_i)_{i \ge 0}$ can be arranged
    so that the Stallings core $SH_i$ of the $H$-cover of $G_i$ has no illegal turns for large $i$.
\end{lemma}

The reason why \Cref{lem:arational_main} implies \Cref{thm:arational_folding} comes from definitions. 
Indeed, eventually having no illegal turns in the Stallings cores implies that the sequence $(SH_i)_{i \ge 0}$ becomes eventually constant. Since the minimal $H$-subtree is the limit of the universal covers of the Stallings cores $SH_i$, it follows that the limiting $H$-subtree on which $H$ acts has to be a simplicial tree; which is free and discrete. As the same holds for every maximal cyclic subgroup, this implies that there cannot be an elliptic element, showing that the limiting tree is non-geometric.

\begin{proof}[Proof of \Cref{lem:arational_main}]
    Firstly, we observe the number of illegal turns may change only at discrete
    times. We further claim that in $SH_t$ the number of illegal turns never
    increases as $t \to \infty$. This is because $SH_t$ immerses into $G_t$, so
    a neighborhood of a vertex in $SH_t$ is embedded in that of the image vertex
    in $G_t$. If furthermore the immersion $SH_t \to G_t$ is a homeomorphism of
    the neighborhoods, then the number of illegal turns does not change in this
    neighborhood at time $t$. Otherwise, if the map is just an embedding of a
    neighborhood to a proper subset of the neighborhood of the image vertex,
    then the number of illegal turns may decrease at time $t$.
    Therefore, it suffices to argue that it can be arranged that the number of illegal turns in $SH_t$ \emph{eventually drops}.
    We remark here that the only ways to increase the number of illegal turns in
    the folding path are characterized in \cite{bestvina2014hyperbolicity}, and
    a typical scenario is in \Cref{fig:increasing_illegals}.

    \begin{figure}[ht!]
        \centering
        \includegraphics[width=.6\textwidth]{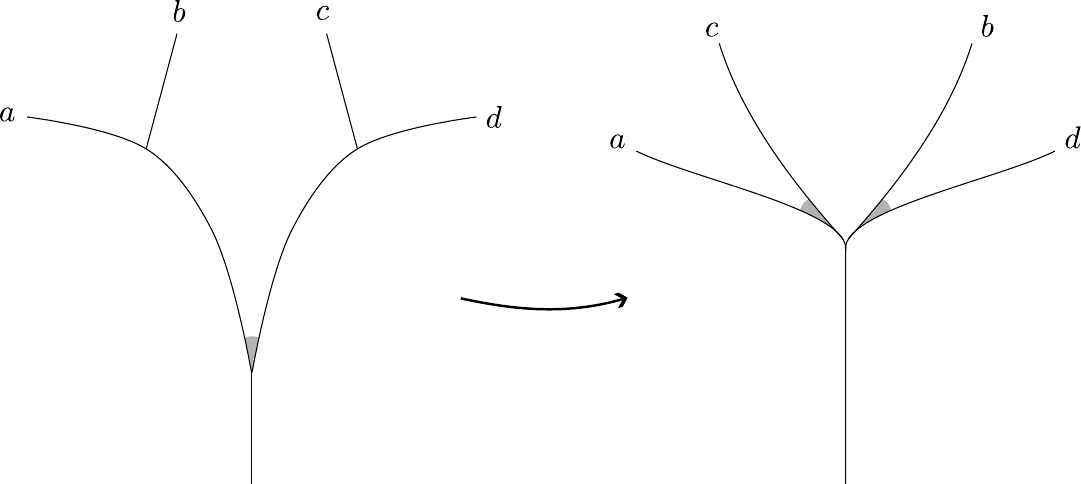}
        \caption{(cf. \cite[Figure 5]{bestvina2014hyperbolicity}.) A typical
          scenario for the number of illegal turns (colored in gray) to increase
          after folding. This situation does not happen in our folding sequence;
          all of our valence $\ge 3$ vertices never merge to valence $\ge 5$
          vertices during the folding.}
        \label{fig:increasing_illegals}
    \end{figure}

    Now, let $v$ be a vertex in $SH_t$ at which an illegal turn exists. We will
    insert folding maps in the folding sequence $(G_s)_{s \ge 0}$ such that the number of illegal turns at $v$ in $SH_t$ drops, while ensuring that Alice still has a winning strategy. We divide the cases based on the edges incoming to the vertex $v$.

    \textbf{Case 1. An illegal turn at $v$ in $SH_t$ enters $v$ through some
      lift of $b_i$.} We examine what other edges in $SH_t$ are attached at $v$.
    If the loop $b_i$ lifts to a \emph{loop} in $SH_t$ at $v$, then the illegal turn would have entered $v$ on the previous step $SH_{t-1}$, and we would rewind to that time and move on to the appropriate subcase of Case 
    2. So, we may assume that $b_i$ does not lift to a loop in $SH_t$ based at $v$. 
    Then, as $H$ is \emph{root-closed} (i.e., $h^k \in H$ for some $k \neq 0$
    implies $h \in H$), no power of $b_i$ lifts to a loop based at $v$. Let $k
    \ge 0$ be the maximum $k$ such that there is a path $b_i^k$ in $SH_t$ that
    starts at $v$. In $SH_t$, the edge $b_i$ could form an illegal turn with
    $a_{i}$ or $\ov{a_{i-1}}$, as well as with $c$ or $\ov{c}$. If either $\{a_i,b_i\}$ or $\{\ov{a_{i-1}},b_i\}$ are illegal, 
    we
    wait; namely, for larger $t'$, we can guarantee that the image of $v$ in
    $SH_{t'}$ has no illegal turn $\{a_i,b_i\}$ or $\{\ov{a_{i-1}},b_i\}$.
    Hence, without loss of generality, it suffices to deal with the case when the turn $\{\ov{c},b_i\}$
    forms an illegal turn at $v$. Recall that our folding sequence must be constructed in such a way that the resultant transition matrices satisfy the conditions required to preserve Alice's winning strategy, as described in
    \Cref{thm:TheGame_folding}. We first try inserting the fold $c
    \mapsto b_i^k c \ov{b_i}^{k+1}$ in the folding sequence. Then, the number of
    illegal turns decreases in $SH_t$, as illustrated in
    \Cref{fig:decreasing_illegals_1}.
    \begin{figure}
        \centering
        \includegraphics[width=.7\textwidth]{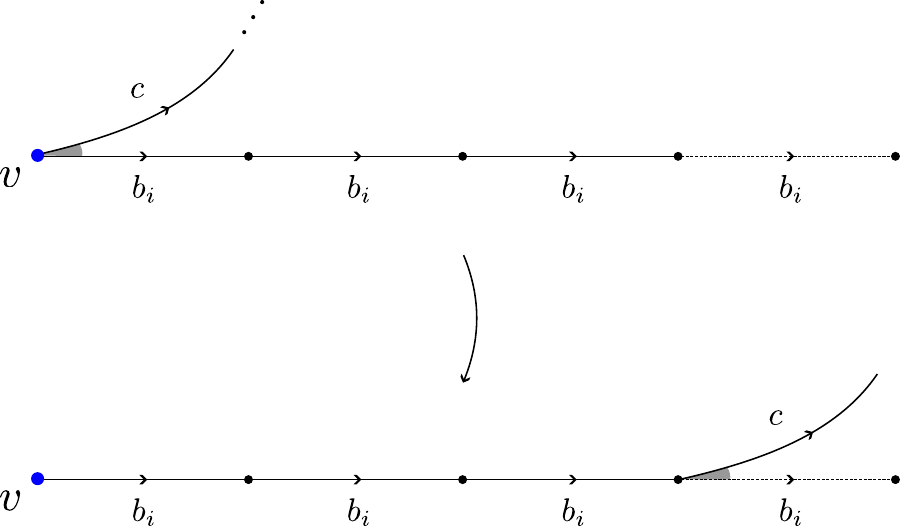}
        \caption{Illustration of folding $c$ over $b_i$ by $k=3$ times as a result of inserting $c \mapsto b_i^3c\ov{b_i}^4$. There is an illegal turn (in gray) between $c$ and $b_i$ in $G_t$, but not in the latter \emph{core} since the edge $b_i$ incident to $c$ is not in the core anymore, as drawn in dashed edge.}
        \label{fig:decreasing_illegals_1}
    \end{figure}
    
    If $k$ is uniformly bounded throughout the folding sequence, 
    then the proof of \Cref{thm:TheGame_folding}
    shows that Alice still has a winning strategy for 
    this modified folding sequence.
    Otherwise, 
    we observe that the 
    bounded rank condition on $H$ implies that the number of vertices of valence
    $\ge 3$ in $SH_t$ is bounded.
    Therefore, on the path $b_i^k$, we can find a path $b_i^s$ of bounded length
    $s \ge 0$ from $v$ such that this subpath ends at a valence $2$ vertex. In this case, we insert the fold 
    $c \mapsto b_i^s \ov{a_{i-1}} b_{i-1} a_{i-1} \ov{b_i}^{s} c$. With this fold, the
    number of illegal turns in $SH_t$ also decreases, as shown in
    \Cref{fig:decreasing_illegals_2}. As $s$ is uniformly bounded by $\rk H$, 
    Alice's winning strategy described in the proof of \Cref{thm:TheGame_folding} is still valid, and so the resultant folding sequence still satisfies the conditions in \Cref{list:Alice_win_folding}.
    This concludes Case 1.
    \begin{figure}[ht!]
        \centering
        \includegraphics[width=.7\textwidth]{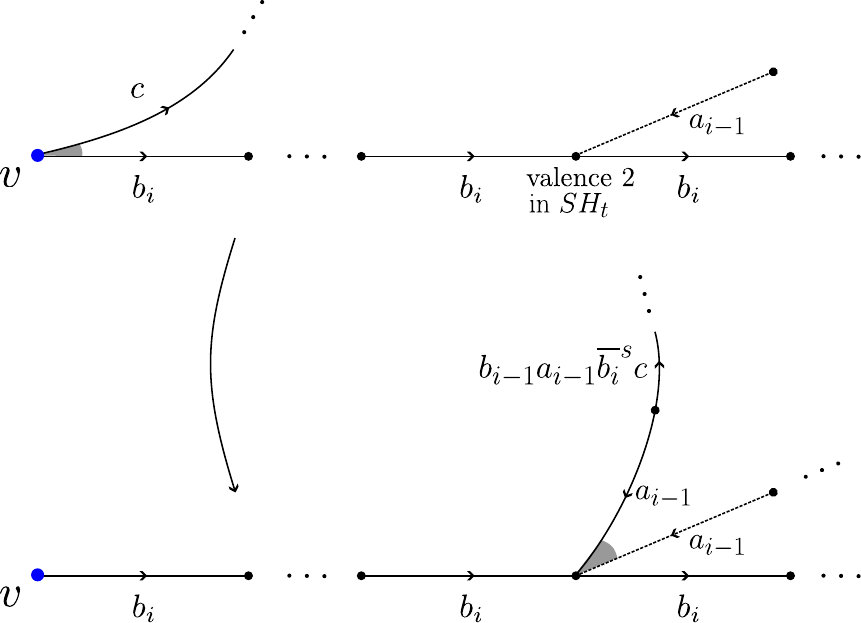}
        \caption{Illustration of folding $c$ over $b_i$ by $s$ times as a result of inserting $c \mapsto b_i^s \ov{a_{i-1}} b_{i-1} a_{i-1} \ov{b_i}^{s} c$. The picture below illustrates that the illegal turn between $c$ and $b_i$ in $G_t$ has already been moved to outside the core, where the dashed edge is an edge not in the core. Hence, the resulting graph after the folding will have less illegal turns in the core.}
        \label{fig:decreasing_illegals_2}
    \end{figure}

    \textbf{Case 2. An illegal turn at $v$ in $SH_t$ enters $v$ through some
      lift of $a_i$.} As in the discussion of Case 1, we can wait if there is an
    illegal turn $\{a_i,b_{i}\}$. Hence, it remains to check when the turn
    $\{\ov{c},a_i\}$ forms an illegal turn at $v$. If $b_i$ is not attached as a
    loop in $SH_t$ at $v$, then we can convert the illegal turn to
    $\{\ov{c},b_i\}$ by inserting $c \mapsto b_i^kc\ov{b_i}^{k+1}$ or $c \mapsto
    b_i^s\ov{a_{i-1}}b_{i-1}a_{i-1}\ov{b_i}^s c$ as in Case 1. In this case, if
    $b_i$ is not attached to $v$ at all in $SH_t$, this conversion results in
    a decrease of the number of illegal turns in $SH_t$. Even if $b_i$ is
    attached to $v$ (but not as a loop), following Case 1, the insertion of such
    maps will still move the illegal turns to outside the core.

    On the other hand, if $a_{i-1}$ is not attached to $v$ in $SH_t$, then we can insert $c \mapsto \ov{a_{i-1}}b_{i-1}a_{i-1} c$ to first convert the illegal turn to $\{c,\ov{a_{i-1}}\}$. Then, having no $a_{i-1}$ attached to $v$ prevents $\{c,\ov{a_{i-1}}\}$ from being illegal in the core, so we decrease the number of illegal turns in this case as well.

    All in all, we may assume that the core $SH_t$ has both (1) the loop $b_{i}$
    at $v$ and (2) $a_{i-1}$ at $v$, as well as $a_i$ and $c$. In this case, we notice that a neighborhood of the vertex $v$ in $SH_t$ is locally homeomorphic to a neighborhood of the image vertex in $G_t$
    (after possibly deleting the $c$ loop). Hence, continuing the folding path,
    $c$ eventually moves along $a_i$ to the next vertex. At this stage,
    similarly we can either reduce the number of illegal turns, or we see this
    new vertex is incident with the loop $b_{i+1}$ and $a_{i+1}$. If the latter
    case keeps happening, the core $SH_{t}$ contains a locally homeomorphic copy
    of $\Delta:=\ov{G_t \setminus c}$ (which is isomorphic to $\G'$ in
    \Cref{fig:lowerbound}), namely the subgraph of $G_t$ minus the
    loop $c$. As the core is compact, it follows that $SH_t$ contains a finite
    cover of $\Delta$ as a subgraph, which implies that $H$ contains a finite-index subgroup of the free factor $\pi_1(\Delta)$. By the root-closed condition on
    $H$, it follows that $H$ contains the full factor $\pi_1(\Delta)$. If
    $H$ is a maximal cyclic subgroup of $F_n$, then we already have a contradiction by
    comparing the rank, so consider the case when $H$ is a proper free factor.
    Here, $\pi_1(\Delta)$ being a maximal proper free factor in $\pi_1(G_t)$
    implies that $H= \pi_1(\Delta)$. However, $\pi_1(\Delta)$ keeps changing with
    folding, as $\Delta$ is already legal. This is a contradiction. Therefore, we decrease the number
    of illegal turns, rather than seeing a copy of $G_t$ minus $c$ inside the
    core.

    \textbf{Case 3. A 1-gate illegal turn at $v$ in $SH_t$ consists of lifts of $c$ and $\ov{c}$.}
    We note that excluding Case 1 and 2, it remains to consider this case, since every illegal turn either involves $c$ or $\ov{c}$ (or both). 
    For this case, we insert $c \mapsto b_ic\ov{b_i}^2$ in the folding path. In doing so, the number of illegal turns decreases in $SH_t$ so that all turns in $SH_t$ are now legal. Indeed, inserting such a map splits the $c$ and $\overline{c}$ edges so they are now incident to different vertices, as in \Cref{fig:decreasing_illegals_3}.
    \begin{figure}[ht!]
        \centering
        \includegraphics[width=.5\textwidth]{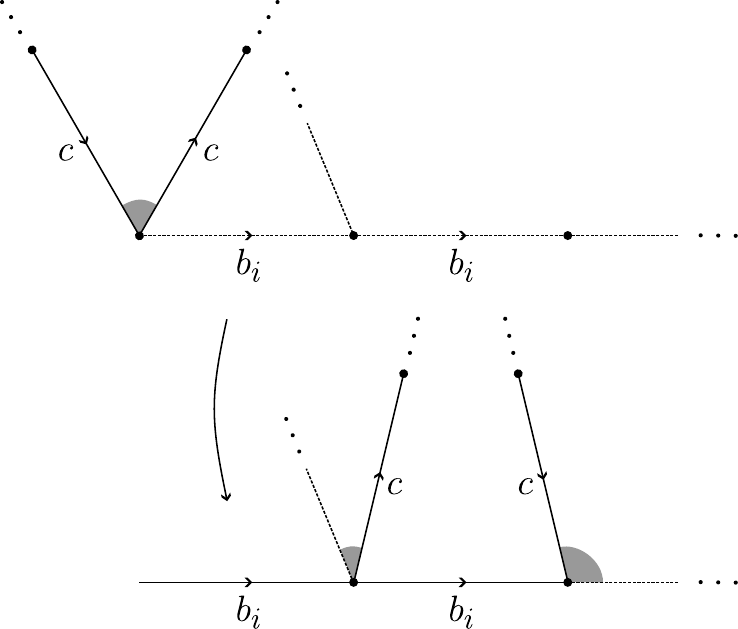}
        \caption{Illustration of Case 3. Inserting $c \mapsto b_ic\ov{b_i}^2$ splits the edges $c,\ov{c}$ so that the number of illegal turns in the core decreases.}
        \label{fig:decreasing_illegals_3}
    \end{figure}
    
    To conclude, we note that each time we insert additional folds, the new folding sequence still satisfies the conditions in \Cref{list:Alice_win_folding} required for Alice to have a winning strategy described in \Cref{thm:TheGame_folding}.
    Therefore, 
    the resultant folding sequence still has transition matrices which satisfy the hypotheses of \Cref{thm:klp_then_ergodic_folding}. Furthermore, each time a new fold is inserted for a particular free factor, it does not affect any of the folds or free factors that have been adjusted previously. This is because each time we insert a new fold into the folding sequence, it occurs later in time from any previous fold that has been inserted. Once an illegal turn becomes legal, it stays legal regardless of any additional foldings.
\end{proof}

\subsection{Unfolding sequence} In this section, we consider the unfolding sequence given in \Cref{ssec:lowerboundFinale}. We show that our construction can be modified to ensure that the legal lamination of our unfolding sequence is arational and non-geometric. In particular,
we will consider the inverse of the unfolding sequence given in \Cref{ssec:lowerboundFinale}, and equip it with an
appropriate train track structure to make it a folding sequence.
Then, the same proof as in \Cref{lem:arational_main} shows that we can arrange
the unfolding sequence such that its inverse folding sequence has arational
limiting tree. This shows that an arational tree is in the accumulation set of
the unfolding sequence. Hence, by \cite[Theorem 1.1]{bestvina2015boundary}, the
unfolding sequence descends into the free factor complex as a quasi-geodesic
limiting to a point on the boundary. Therefore, the accumulation points of the
unfolding sequence lie in a single simplex consisting of arational trees, concluding that
the legal lamination of our unfolding sequence is arational.
Thus, it suffices to investigate the inverse of our unfolding sequence and
an appropriate train track structure on it.

We first compute the inverses of the rein movers, loopers, and the rotator given
in the proof of \Cref{thm:lowerboundFinale_folding}. The inverses of the rein movers $\cR_{a_0},\ldots,\cR_{a_{n-4}}$ and $\cR_{b_{1}},\ldots,\cR_{b_{n-3}}$ are:
\begin{align*}
    &\cR_{a_i}^{-1}: \G_{b_{i+1}} \to \G_{a_i},
    &&\cR_{a_i}^{-1} =
    \begin{cases}
    c \mapsto \ov{a_i}ca_i,\\
    \text{Id} & \text{otherwise},
    \end{cases}\\
    &\cR_{b_i}^{-1}: \G_{a_i} \to \G_{b_i},
    &&\cR_{b_i}^{-1} =
    \begin{cases}
    c \mapsto \ov{b_i}^2cb_i,\\
    \text{Id} & \text{otherwise}.
    \end{cases}
\end{align*}

Secondly, the inverses of the loopers $\cL_{a_1},\ldots,\cL_{a_{n-3}}$ and $\cL_{b_1},\ldots,\cL_{b_{n-3}}$ are:
\begin{align*}
    &\cL_{a_i}^{-1}: \G_{a_i} \to \G_{a_i},
    &&\cL_{a_i}^{-1} =
    \begin{cases}
        a_i \mapsto (\ov{b}_i\ov{a_{i-1}}\ov{b_{i-1}}a_{i-1})^{\alpha_i}\ov{c}a_i \\
    \text{Id} & \text{otherwise},
    \end{cases}\\
    &\cL_{b_i}^{-1}: \G_{b_i} \to \G_{b_i},
    &&\cL_{b_i}^{-1} =
    \begin{cases}
        b_i \mapsto \ov{a_{i-1}}\ov{b_{i-1}}^{\beta_i}a_{i-1}\ov{c}b_i \\
        \text{Id} & \text{otherwise}.
    \end{cases}
\end{align*}

Finally, we compute the inverse of the \emph{rotator} $\rho$:
\[
    \rho^{-1}: \G_{a_0} \to \G_{a_{n-3}},
    \qquad \qquad
    \rho =
    \begin{cases}
        a_i \mapsto a_{i-1} & \text{for $i=0,\ldots,n-3$,}\\
        b_i \mapsto b_{i-1} & \text{for $i=0,\ldots,n-3$,}\\
        c \mapsto c.
    \end{cases}
\]

Then our desired map $\calF^{-1}$ for the inverse folding sequence is just the inverse of the folding map $\calF$ of our unfolding sequence:
    \begin{align*}
        \calF^{-1} &= \left(\rho \circ (\cL_{a_{n-3}} \circ \cR_{b_{n-3}})\circ (\cL_{b_{n-3}} \circ \cR_{a_{n-4}}) \circ \cdots \circ (\cL_{a_1} \circ \cR_{b_1}) \circ (\cL_{b_1} \circ \cR_{a_0})\right)^{-1}\\
        &= \left(\prod_{i=1}^{n-3}[\cR_{a_{i-1}}^{-1} \circ \cL_{b_i}^{-1} \circ \cR_{b_i}^{-1} \circ \cL_{a_{i}}^{-1}] \right) \circ \rho^{-1},
    \end{align*}
    which is a map from $\G_0$ to $\G_0$.

To make it a train track map, we endow the train track structure as follows. For
oriented edges $e_1,e_2$ of the graph $\G_e$, denote by $\G_{e_1,e_2}$ the graph
$\G_{e_1}$(which is defined in \Cref{ssec:lowerboundFinale_folding}), but with different
train track structure: it has two illegal turns given by either (1) if neither
$e_1$ nor $e_2$ is $c$ or $\ov{c}$, then the two illegal turns are given by the
two gates $\{e_1,e_2\}$ and $\{c,\ov{c}\}$, or (2) otherwise, the two illegal turns are given by the one gate
$\{e_1,e_2\} \cup \{c, \ov{c}\}$ consisting of three oriented edges. See
\Cref{fig:lowerbound_inverses} for an example.

\begin{figure}[ht!]
    \centering
    \includegraphics[width=.8\textwidth]{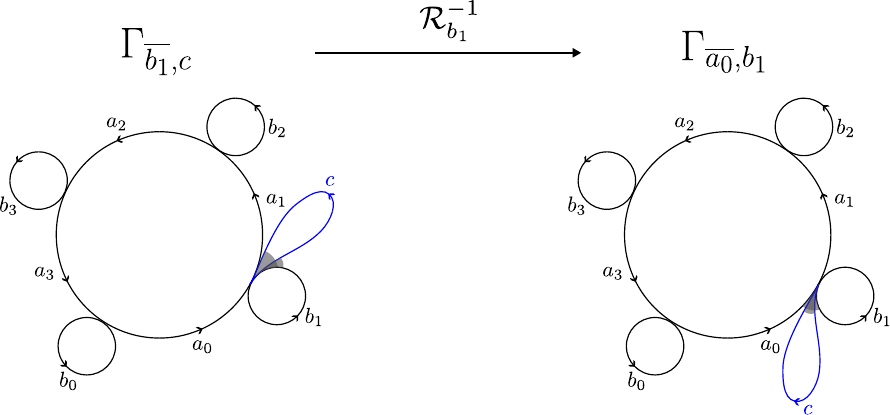}
    \caption{An illustration of the inverse $\cR_{b_1}^{-1} : \G_{\ov{b_1},c} \to \G_{\ov{a_0},b_1}$ of the rein mover map $\cR_{b_1} : \G_{b_1} \to \G_{a_1}$. Here note $\G_{\ov{b_1},c} \cong \G_{a_1}$ and $\G_{\ov{a_0},b_1} \cong \G_{b_1}$ as graphs, but they have different train track structures.}
    \label{fig:lowerbound_inverses}
\end{figure}

Then we rewrite the domain and codomain of the rein movers and loopers as follows.
\begin{align*}
    &\cR_{a_i}^{-1}: \G_{\ov{a_i},c} \longrightarrow \G_{\ov{b_i}, a_i} &&\cR_{b_i}^{-1}: \G_{\ov{b_i},c} \longrightarrow \G_{\ov{a_{i-1}}, b_i}\\
    &\cL_{a_i}^{-1}: \G_{\ov{b_i},a_i} \longrightarrow \G_{\ov{b_i}, c}  && \cL_{b_i}^{-1}: \G_{\ov{a_{i-1}}, b_i} \longrightarrow \G_{\ov{a_{i-1}},c} \\
    &\rho^{-1}: \G_{\ov{b_0}, a_0} \longrightarrow \G_{\ov{b_{n-3}}, a_{n-3}}.
\end{align*}

Then $\calF$ becomes a train track map, making the inverse of our unfolding
sequence into a folding sequence. Following the proof of
\Cref{lem:arational_main}, we can show its limiting tree is arational with some
insertion of maps to the original unfolding sequence such that the inverses of
the inserted maps are folding maps(e.g. inserting $c \mapsto
\ov{b}_i^kcb_i^{k+1}$ to the unfolding sequence, has the desired inverse $c \mapsto
b_i^k c \ov{b}_i^{k+1},$ which is a folding map with respect to the train track
structure on the inverse folding sequence and it is the same insertion as in
Case 1 of the proof of \Cref{lem:arational_main}.) while maintaining the lower
bound for the number of ergodic measures. This concludes that the legal lamination of our unfolding sequence is arational as well. Hence we just proved:

\begin{theorem} \label{thm:arational_unfolding}
    There exists an unfolding sequence whose legal lamination is non-geometric, arational and admits $2n-6$ linearly independent ergodic currents.
\end{theorem}

\bibliography{bib}

\begin{thebibliography}{10}

\bibitem{bestvina2014PCMI}
Mladen Bestvina.
\newblock Geometry of outer space.
\newblock In {\em Geometric group theory}, volume~21 of {\em IAS/Park City
  Math. Ser.}, pages 173--206. Amer. Math. Soc., Providence, RI, 2014.

\bibitem{bestvina1994outerlimits}
Mladen Bestvina and Mark Feighn.
\newblock Outer limits.
\newblock preprint.

\bibitem{bestvina2014hyperbolicity}
Mladen Bestvina and Mark Feighn.
\newblock Hyperbolicity of the complex of free factors.
\newblock {\em Adv. Math.}, 256:104--155, 2014.

\bibitem{bestvina2024limit}
Mladen Bestvina, Radhika Gupta, and Jing Tao.
\newblock Limit sets of unfolding paths in outer space.
\newblock {\em Journal of the Institute of Mathematics of Jussieu}, pages
  1--39, 2024.

\bibitem{bestvina2015boundary}
Mladen Bestvina and Patrick Reynolds.
\newblock The boundary of the complex of free factors.
\newblock {\em Duke Math. J.}, 164(11):2213--2251, 2015.

\bibitem{cohen1995verysmall}
Marshall~M. Cohen and Martin Lustig.
\newblock Very small group actions on {${\bf R}$}-trees and {D}ehn twist
  automorphisms.
\newblock {\em Topology}, 34(3):575--617, 1995.

\bibitem{collins2007asymmetric}
Julia Collins and Johannes Zimmer.
\newblock An asymmetric {A}rzel\`a-{A}scoli theorem.
\newblock {\em Topology Appl.}, 154(11):2312--2322, 2007.

\bibitem{CHL2008RtreeII}
Thierry Coulbois, Arnaud Hilion, and Martin Lustig.
\newblock {$\Bbb R$}-trees and laminations for free groups. {II}. {T}he dual
  lamination of an {$\Bbb R$}-tree.
\newblock {\em J. Lond. Math. Soc. (2)}, 78(3):737--754, 2008.

\bibitem{CHR2015indecomposable}
Thierry Coulbois, Arnaud Hilion, and Patrick Reynolds.
\newblock Indecomposable {$F_N$}-trees and minimal laminations.
\newblock {\em Groups Geom. Dyn.}, 9(2):567--597, 2015.

\bibitem{culler1987groupactions}
Marc Culler and John~W. Morgan.
\newblock Group actions on {${\bf R}$}-trees.
\newblock {\em Proc. London Math. Soc. (3)}, 55(3):571--604, 1987.

\bibitem{culler1986moduli}
Marc Culler and Karen Vogtmann.
\newblock Moduli of graphs and automorphisms of free groups.
\newblock {\em Invent. Math.}, 84(1):91--119, 1986.

\bibitem{gabai2009almost}
David Gabai.
\newblock Almost filling laminations and the connectivity of ending lamination
  space.
\newblock {\em Geom. Topol.}, 13(2):1017--1041, 2009.

\bibitem{gabai2014topology}
David Gabai.
\newblock On the topology of ending lamination space.
\newblock {\em Geom. Topol.}, 18(5):2683--2745, 2014.

\bibitem{guirardel2000dynamics}
Vincent Guirardel.
\newblock Dynamics of {${\rm Out}(F_n)$} on the boundary of outer space.
\newblock {\em Ann. Sci. \'{E}cole Norm. Sup. (4)}, 33(4):433--465, 2000.

\bibitem{hamenstaedt2014boundary}
Ursula Hamenstaedt.
\newblock The boundary of the free splitting graph and the free factor graph.
\newblock {\em arXiv preprint, arXiv:1211.1630}, 2014.

\bibitem{handel-mosher:FSishyperbolic}
Michael Handel and Lee Mosher.
\newblock The free splitting complex of a free group, {I}: hyperbolicity.
\newblock {\em Geom. Topol.}, 17(3):1581--1672, 2013.

\bibitem{horbezverysmall}
Camille Horbez.
\newblock The boundary of the outer space of a free product.
\newblock {\em Israel J. Math.}, 221(1):179--234, 2017.

\bibitem{ilya:currents}
Ilya Kapovich.
\newblock Currents on free groups.
\newblock In {\em Topological and asymptotic aspects of group theory}, volume
  394 of {\em Contemp. Math.}, pages 149--176. Amer. Math. Soc., Providence,
  RI, 2006.

\bibitem{kapovich2009intersection}
Ilya Kapovich and Martin Lustig.
\newblock Geometric intersection number and analogues of the curve complex for
  free groups.
\newblock {\em Geom. Topol.}, 13(3):1805--1833, 2009.

\bibitem{keane1977nonergodic}
Michael Keane.
\newblock Non-ergodic interval exchange transformations.
\newblock {\em Israel J. Math.}, 26(2):188--196, 1977.

\bibitem{LM2010criteria}
Anna Lenzhen and Howard Masur.
\newblock Criteria for the divergence of pairs of {T}eichm\"{u}ller geodesics.
\newblock {\em Geom. Dedicata}, 144:191--210, 2010.

\bibitem{levitt1983}
Gilbert Levitt.
\newblock {\em Feuilletages des Surfaces}.
\newblock Thèse de doctorat d'État, Université Paris VII, Paris, France,
  June 1983.

\bibitem{reiner:thesis}
Reiner Martin.
\newblock {\em Non-uniquely ergodic foliations of thin-type, measured currents
  and automorphisms of free groups}.
\newblock ProQuest LLC, Ann Arbor, MI, 1995.
\newblock Thesis (Ph.D.)--University of California, Los Angeles.

\bibitem{martin1997nonuniquely}
Reiner Martin.
\newblock Non-uniquely ergodic foliations of thin type.
\newblock {\em Ergodic Theory Dynam. Systems}, 17(3):667--674, 1997.

\bibitem{mcmullen2013diophantine}
Curtis~T. McMullen.
\newblock Diophantine and ergodic foliations on surfaces.
\newblock {\em J. Topol.}, 6(2):349--360, 2013.

\bibitem{morgan1988ergodic}
John~W. Morgan.
\newblock Ergodic theory and free actions of groups on {${\bf R}$}-trees.
\newblock {\em Invent. Math.}, 94(3):605--622, 1988.

\bibitem{morgan1992trees}
John~W. Morgan.
\newblock {$\Lambda$}-trees and their applications.
\newblock {\em Bull. Amer. Math. Soc. (N.S.)}, 26(1):87--112, 1992.

\bibitem{NPR2014}
Hossein Namazi, Alexandra Pettet, and Patrick Reynolds.
\newblock Ergodic decompositions for folding and unfolding paths in outer
  space.
\newblock {\em arXiv preprint arXiv:1410.8870}, 2014.

\bibitem{reynolds2012reducing}
Patrick Reynolds.
\newblock Reducing systems for very small trees.
\newblock {\em arXiv preprint arXiv:1211.3378}, 2012.

\bibitem{stallings1983topology}
John~R. Stallings.
\newblock Topology of finite graphs.
\newblock {\em Invent. Math.}, 71(3):551--565, 1983.

\bibitem{yoccoz2010interval}
Jean-Christophe Yoccoz.
\newblock Interval exchange maps and translation surfaces.
\newblock In {\em Homogeneous flows, moduli spaces and arithmetic}, volume~10
  of {\em Clay Math. Proc.}, pages 1--69. Amer. Math. Soc., Providence, RI,
  2010.

\end{thebibliography}
\bibliographystyle{plain}
\end{document}